\documentclass[10pt,twoside,reqno]{amsart}
\usepackage{amssymb, amsmath, graphicx, amsfonts, euscript,mathrsfs}
\usepackage{color}
\usepackage{tikz}
\usepackage{dsfont}
\newtheorem{thm}{Theorem}[section]

\newtheorem{prop}[thm]{Proposition}

\newtheorem{lemma}[thm]{Lemma}
\theoremstyle{remark}
\newtheorem{remark}{Remark}
\newcommand{\I}{\mathbb{I}}
\newcommand{\e}{\epsilon}
\newcommand{\va}{\varepsilon}
\newcommand{\p}{\partial}

\newcommand{\A}{\alpha}
\newcommand{\bb}{\beta}

\newcommand{\na}{\nabla}
\newcommand{\la}{\lambda}
\newcommand{\de}{\delta}

\newcommand{\R}{\mathbb{R}}
\newcommand{\Z}{\mathbb{Z}}
\newcommand{\N}{\mathbb{N}}
\newcommand{\LL}{\mathfrak{L}}
\newcommand{\F}{\mathcal{F}}

\newcommand{\m}{\mathfrak{e}}

\newcommand{\Q}{\mathbb{Q}}
\newcommand{\B}{\mathbb{B}}

\newcommand{\les}{\lesssim}

\newcommand{\T}{\mathbb{T}}

\newcommand{\ka}{\kappa}

\newcommand{\K}{\mathfrak{K}}

\newcommand{\V}{\mathbf{V}}
\newcommand{\ind}[1]{\mathds{1}_{{#1}}}
\newcommand{\KK}{\mathscr{K}}

\newcommand{\VV}{\mathscr{V}}

\newcommand{\ma}{m^{(1)}}
\newcommand{\mb}{m^{(2)}}
\newcommand{\mc}{m^{(3)}}
\newcommand{\cc}{\mathbf{c}}
\newcommand{\NE}{\mathfrak{N}}
\def\beq{\begin{equation}}
\def\eeq{\end{equation}}
\numberwithin{equation}{section}
\begin{document}
\title[Almost sharp variational estimates]{Almost sharp variational  estimates \\ for discrete truncated    operators of  Stein-Wainger type}

\author[J. Chen and R. Wan]{Jiecheng Chen$^{\dagger}$ and  Renhui Wan$^{\ddagger}$}
%\author[ R. Wan]{ Renhui Wan}

\address{$^\dagger$ School of Mathematical Sciences, Zhejiang Normal University, Jinhua 321004, People's Republic of China}

\email{jcchen@zjnu.edu.cn}

%\address{$^2$ School of Mathematical Sciences, Zhejiang University, Hangzhou 310027, People's Republic of China}

%\email{mjhy@zju.edu.cn}

\address{$^{\ddagger}$Ministry of Education Key Laboratory of NSLSCS, School of Mathematical Sciences, Nanjing Normal University, Nanjing 210023, People's Republic of China}

\email{wrh@njnu.edu.cn (Corresponding author)}

\vskip .2in
\begin{abstract}
 We establish  $r$-variational estimates for discrete truncated Stein-Wainger type operators on $\ell^p$ for $1<p<\infty$. Notably, these estimates are sharp and  enhance the results obtained  by Krause and Roos (J. Eur. Math. Soc. 2022, J. Funct. Anal. 2023), up to a logarithmic loss related to the scale. On the other hand, as $r$ approaches infinity,   the consequences  align with the estimates proved  by Krause and Roos. Moreover, for the case of quadratic phases, we  remove this logarithmic loss with respect to the scale in two and higher dimensions, 
 at the cost of increasing  $p$ slightly. 
\end{abstract}

\maketitle

%\tableofcontents

\section{Introduction}
\label{intrs1}
\subsection{Motivation and main results}
\label{Mr}
The variational inequality is a fundamental concept that holds significant importance in various mathematical disciplines such as  harmonic analysis, probability theory and ergodic theory. It provides a quantitative measure of how functions or operators fluctuate within a defined range. In harmonic analysis, it assists in delineating the regularity and characteristics of functions (see, e.g., \cite{MSZ203, PP36, PP37}). In probability theory, it is crucial for comprehending stochastic processes and their dynamics (see, e.g., \cite{Le76, FZ23}). In ergodic theory, it plays a key role in the development of algorithms by establishing pointwise convergence and quantifying convergence rates; notably, recent advancements in addressing the Furstenberg-Bergelson-Leibman conjecture rely on the foundation of variational inequalities (see \cite{KMT22, IMMS23}). In this paper, we will establish variational inequalities for discrete truncated    operators of Stein-Wainger type.

Let $n$ and $d$ be positive integers,
 and let $\la(x)$ be an arbitrary  function mapping
 from $\Z^n$ to [0,1].
Define  the    discrete truncated Stein-Wainger type operators $\{\mathscr{C}_{N}\}_{N\in\N}$  by the formula
\beq\label{ope1}
\mathscr{C}_{N} f(x):=\sum_{y\in \B_N\setminus \{0\}}f(x-y)e\big(\la(x)|y|^{2d}\big){K}(y)\ \ \ (x\in\Z^n),
\eeq
where $e(\theta):=e^{2\pi i \theta}$,   $\B_t=\{x\in \Z^n:\ |x|\le t\}$ with $t>0$, and
  $K$ is a homogeneous Calder\'{o}n-Zygmund
kernel, characterized by
\beq\label{v1}
K(y)=\frac{\Omega(y)}{|y|^n}
\eeq
for some function $\Omega \in \mathcal{C}^1(\R^n\setminus\{0\})$, which is homogeneous of degree 0.\footnote{The assumption of homogeneity for $K$ is not strictly essential; its inclusion is intended to simplify the proof of main results and enhance the clarity of this paper.} Additionally, $K$ exhibits the property of mean value zero, implying that $\int_{\mathbb{S}^{n-1}} \Omega(x) d\sigma(x) = 0$, where $\sigma$ represents the surface measure on $\mathbb{S}^{n-1}$. This paper aims to investigate  $\ell^p$ inequalities for $r$-variations of $\{\mathscr{C}_{N}f\}_{N\in\mathbb{N}}$ for all $f\in\ell^p(\mathbb{Z}^n)$, which is related to a variational  seminorm $V^r$.  
See Subsection \ref{semi} below for a general definition of the  variational seminorm $V^r$.  %and the corresponding variational  norm $\V^r$.   
 As described at the beginning of this section, this seminorm plays a pivotal role in addressing pointwise convergence concerns.
Traditionally, tackling pointwise convergence issues involves proving  $L^p(X,\mu)$ boundedness for the associated maximal function, which simplifies the task to proving the pointwise convergence across a dense set of $L^p(X,\mu)$ functions. Nonetheless, achieving the  pointwise convergence over a dense class can pose challenges (as exemplified by Bourgain's averaging operator along the squares in \cite{Bour89}).
In this context, if $\|\big(\mathscr{C}_{N}f(x)\big)_{N\in\mathbb{N}}\|_{V^r}<\infty$ for certain $r\in [1,\infty)$ and $x\in\mathbb{Z}^n$, then the limit $\lim_{N\to \infty}\mathscr{C}_{N}f(x)$ exists. Consequently, there is no necessity to establish the pointwise convergence over a dense class.
%%%
 In addition, the seminorm $V^r$ governs the supremum norm as follows: For any $N_0\in \N$, we can infer  the pointwise estimate
 $$\sup_{N\in\N}|\mathscr{C}_{N}f(x)|\le |\mathscr{C}_{N_0}f(x)|+\|\big(\mathscr{C}_{N}f(x)\big)_{N\in \N}\|_{V^r}.$$

For the case of $\la(x)\equiv0$,  the operator  (\ref{ope1}) simplifies to a specific instance of the discrete truncated singular Radon transform, which has been extensively studied by various mathematicians (see \cite{MST17, MSZ20, MSZ202, Wo22} and references therein),
and is defined by the formula
 $$\mathcal{T}_Nf(x):=\sum_{y\in \B_N\setminus \{0\}}f(x-\mathcal{P}(y))K(y)\ \ \ (x\in\Z^k)$$
with $\mathcal{P}=(\mathcal{P}_1,\dots,\mathcal{P}_k):\Z^n\to \Z^k$ a polynomial
mapping, where for each $j\in \{1,\dots,k\}$, the function  $\mathcal{P}_j:\Z^n\to \Z$ is an integer-valued polynomial of $n$ variables satisfying   $\mathcal{P}_j(0)=0$.
 A really significant  job  is \cite{MST17}, where   Mirek, Stein and Trojan   established  a sharp  $\ell^p$ inequality  for the $r$-variation of this  truncated singular Radon transform.
Specifically, given that  Bourgain's logarithmic lemma (see \cite{NOT10}) is generally very inefficient for $\ell^p$ estimates when $p\neq2$, they developed a new and flexible approach to cover the full range $p\in(1,\infty)$. This approach was  based on 
Rademacher-Menshov-type inequalities (numerical inequalities) and a direct analysis of the associated multiplier.
  In the present work,   since $\la(x)\not\equiv0$, the problem becomes more intricate, making it  difficult to apply the methodology from \cite{MST17}.  Nevertheless, numerical inequalities proven in \cite{MST17} will remain important  in the proofs of our main results. 
  For studies regarding related jump inequalities, we refer \cite{MSZ20, MSZ202}. For a comprehensive examination of the connections between variational inequalities and jump inequalities, we refer \cite{MSZ203}.

For the operator (\ref{ope1}) when $N=\infty$, represented by $\mathscr{C}_{\infty}$, it is closely linked to the discrete version of a maximal operator on $\mathbb{R}^n$, which was studied by Stein and Wainger \cite{SW01}, and  considered  as a generalization of the Carleson operator (see, e.g., \cite{Car66,Fe73,OSTTW12,LT20}).
Through linearization, the $\ell^p(\mathbb{Z}^n)$ estimate of the operator $\mathscr{C}_{\infty}$ is equivalent to that of the maximal operator $\mathscr{C}$ defined by
\beq\label{opp2}
\mathscr{C} f(x)=\sup_{u \in[0,1]}\big|\sum_{y\in \Z^n\setminus \{0\}}f(x-y)e\big(u|y|^{2d}\big){K}(y)\big|\ \ \ (x\in\Z^n).
\eeq
%For the case $d=n=1$, Krause and Lacey \cite{KL17} obtained a quadratic Carleson theorem on $\ell^2(\Z)$ with a restricted supremum,
The $\ell^2$ estimate  of the operator $\mathscr{C}$ in the case where $d=n=1$  was the focus of a question raised by Lillian Pierce during an AIM workshop in 2015.
 Krause and Roos \cite{KR22} proved the $\ell^2$ estimate for the operator $\mathscr{C}$ whenever $d\ge 1$ and $n\ge 1$,
 which resolved the above question;
  we refer    \cite{KL17,CHKL18} for  related works with a restricted supremum on $u$.
 Instead of  using Bourgain's logarithmic lemma,
 they handled  the full supremum on $u$ by combining  number-theoretic  components with a sophisticated multi-frequency analysis inspired by \cite{KL17}, and utilizing the Rademacher-Menshov-type inequality  demonstrated in \cite{MST17}.
   %%%
   Afterwards, through a fusion of the Ionescu-Wainger-type multiplier theorem (see \cite{IW05,MM18,TT21}) with techniques from \cite{KR22}, Krause and Roos \cite{KR23} successfully attained $\ell^p$ estimates for the operator $\mathscr{C}$ across all $p\in(1,\infty)$.  Very recently, Krause \cite{K24} considered
 a multi-parameter version   of
 $\mathscr{C}_N$,  featuring generic polynomials without linear terms in its phase, and established  $\ell^p$ estimates of the associated maximal function; in particular, the operator  $\mathscr{C}_*$, defined by
  $\mathscr{C}_*f:=\sup_{N\in\N}|\mathscr{C}_N f|$, is $\ell^p$ bounded for all $p\in(1,\infty)$.    While the maximal operator $\mathscr{C}_*$ presents a more robust framework, the techniques employed to bound $\mathscr{C}_N$ or $\mathscr{C}$ are equally applicable. However, 
  a distinctive approach is imperative to establish the variational inequality for $\{\mathscr{C}_N\}_{N\in\N}$ 
  since its validation necessitates a desired   multi-frequency analyse and vector-valued inequalities with respect to  the seminorm $V^r$ at this juncture.

 Motivated by the studies  in \cite{MST17} on  variational inequalities for  truncated singular Radon transforms, and the works in \cite{KR22,KR23,K24} on $\ell^p$ estimates for the  operator (\ref{ope1}), we are interesting in establishing     variational inequalities for the operator (\ref{ope1}). 
Our main result of this paper is the following theorem.
\begin{thm}\label{t1}
Let $n$ and $d$ be positive integers,
 and  let $\la(x)$ be an arbitrary  function from $ \Z^n$ to [0,1]. Suppose $r\in (2,\infty)$ and $p\in (1,\infty)$.
Then for any $R\ge 1$ and any $\e>0$, we have
\beq\label{t11}
\|(\mathscr{C}_{N} f)_{N\in\N}\|_{\ell^p(\B_R;V^r)}\les_{\e} R^{\e/r} \|f\|_{\ell^p(\Z^n)}
\eeq
with  the implicit constant  independent of $R$, $f$ and the function $\la(x)$.
\end{thm}

Note that the variational seminorm is taken in the scale parameter $N$; it would also be interesting to consider variational estimates in the modulation parameter $\lambda$ (in the real-variable setting, such estimates were established by Guo–Roos–Yung \cite{GRY20}).

The   $R^{\e/r}$-loss in the upper bound of  (\ref{t11})  could be  improved to  a logarithmic loss in $R$ (for instance, $(\ln \langle R \rangle)^{C/r}$ for some constant   $C>0$), we choose not to persue this avenue  in order to enhance the clarity and presentation of this paper. For the details, see the reduction of (\ref{t11}) in Section \ref{pth1} and $Remark$ \ref{rr4.1} in Subsection \ref{vi1}. 

In  the special case where $d=1$ and $n\ge2$, we can eliminate this loss related to the scale $R$ on the right-hand side of (\ref{t11}) by  increasing  $p$ slightly. We now present our second  result.
\begin{thm}\label{co1}
Let $n\ge2$ be a positive integer and $d=1$,
and  let $\la(x)$ be an arbitrary  function from $ \Z^n$ to [0,1].
If $r\in (2,\infty)$ and $p\in [1+2/n,\infty)$,   we have
\beq\label{T12}
\|(\mathscr{C}_{N} f)_{N\in\N}\|_{\ell^p(\Z^n;V^r)}\les  \|f\|_{\ell^p(\Z^n)}
\eeq
with  the implicit constant  independent of $f$ and the function $\la(x)$.
\end{thm}

Comments on Theorems \ref{t1} and \ref{co1}     are given as follows:

\begin{itemize}

\item
The upper bound $R^{\epsilon/r}$ in (\ref{t11}) converges to 1 as $r$ approaches infinity, ensuring that (\ref{t11}) aligns with the estimates derived by \cite{KR22,KR23,K24}.
Indeed, the inequality (\ref{t11}),  allowing for a logarithmic loss with respect to  the scale $R$,  is sharp and strengthens the estimates by Krause-Roos \cite{KR22,KR23}.
% and the corresponding maximal estimate by Krause \cite{K24}. 
 Moreover, the domain $\B_R$ on the left side of (\ref{t11}) can be substituted by any $\B_R(z):=\{x\in \Z^n: |x-z|\le R\}$ with $z\in \Z^n$, which guarantees the convergence of $\lim_{N\to \infty} \mathscr{C}_{N} f(x)$ for $x\in \Z^n$.

\item
The regularity assumption\footnote{While this assumption aligns with that in Krause's recent work \cite{K24}, his approach may not be well-suited for demonstrating the requisite variational inequalities.}  $\Omega \in \mathcal{C}^1(\mathbb{R}^n \setminus {0})$ relaxes the higher regularity requirements for $\Omega$ found in \cite{KR22, KR23} (see the proof of (7.11) in \cite{KR22}). Furthermore, in the one-dimensional case, our approach can be applied to the operator (\ref{ope1}) with the phase $|y|^{2d}$ replaced by  $y^m$, where $m \geq 3$ is any odd integer.

\item
%The range of $p$ for  
   The range of $p$ in Theorem \ref{co1} can be extended to a slightly bigger interval $(1+2/n-\tilde c,\infty)$ with some small $\tilde c>0$ (see $Remark$ \ref{r100} in Section \ref{slong5} for the details).
%By the way, (\ref{T12}) for the case $n=1$ notably enhances the $\ell^2(\mathbb{Z})$ estimate, a central focus in the question posed by Lillian Pierce.
Moreover, the inequality (\ref{T12}) 
extends  to the operator (\ref{ope1}) with the phase $|y|^2$ replaced by a generic phases (including forms like  $y_1^d\pm\cdots\pm y_n^d$, $d\ge 2$). This extension holds under the conditions
 $n\ge C$ and $p\in [1+c,\infty)$ 
 for some positive   $C=C(d)$ and $c=c(n,d)$.

 \item
 We expect the jump inequalities (for $r=2$) associated with (\ref{t11}) and (\ref{T12}) to hold, though we omit the details here. These can be derived by combining the techniques from our current work, additional properties of jump inequalities from \cite{MSZ203}, and a variant of the transference principle stated in Proposition \ref{PMST} below (which can be deduced by suitably adapting its proof).

 %\item Add an $\e$-remove arguments??
 \end{itemize}
\subsection{Overview of the proof}
\label{diff}
We first  provide  the  novelties   applied in the proof of our  results. Specifically, the  novelties  primarily arise in establishing the major arcs estimates.

 \begin{itemize}

  \item
The primary innovation in this paper lies in establishing a crucial multi-frequency variational inequality (see Lemma \ref{endle} below), which serves as a key element in proving major arcs estimate II (see Proposition \ref{892} below) and subsequently attaining the desired long variational inequality (\ref{long}). Essentially, this innovative multi-frequency variational inequality can be viewed as an extension of the double maximal estimate presented in Lemma 7.2 of \cite{KR22}. However, since the seminorm $V^r$ does not guarantee that $\|(f_N)_{N\in\mathbb{N}}\|_{V^r}\lesssim \|(g_N)_{N\in\mathbb{N}}\|_{V^r}$ whenever $|f_N|\lesssim |g_N|$ for all $N\in\mathbb{N}$, applying the approach yielding  Lemma 7.2 in \cite{KR22} to achieve this objective becomes challenging.
To overcome this difficulty, we introduce a practical multi-frequency square function estimate (see Lemma \ref{ccz} below) and combine  various techniques such as the classical variational inequality in the continuous setting, the Ionescu-Wainger-type multiplier theorem, a transference principle by Mirek-Stein-Trojan,  and a Rademacher-Menshov-type inequality.
For more details, see Subsection \ref{vi1} below.

   \item
Another novelty is the strategy used  to overcome the difficulty posed by the rough and variable-dependent kernel associated with the operator (\ref{ope1}), which hinders the application of numerical inequalities in addressing major arcs estimate III (see Proposition \ref{t31} below) concerning the short variation. 
To tackle this issue, we will combine the previously mentioned multi-frequency square function estimate, the shifted square function estimate on $\mathbb{R}^n$ (see   Appendix \ref{app}), the Plancherel-P\'{o}lya inequality.
%Remarkably, the Plancherel-P\'{o}lya inequality,  typically used for establishing variational inequalities in continuous settings, prove to be unexpectedly useful in this context.

 \item
Apart from the aforementioned difficulties, two tricks are pivotal in the forthcoming proof. Firstly, directly achieving  $\ell^p$ ($1<p<2$)   major arcs estimates I and II  poses a significant challenge. This scenario is a common occurrence in studies based on  interpolation. To address this issue, we shall  revisit the original operator and approach this problem from a fresh perspective. This strategic maneuver constitutes the first trick that will be employed.
Secondly,  we leverage the Gauss sum bounds to establish an inequality (see (\ref{L149}) below)  that is more suitable for bounding the seminorm $V^r$ of the target operator compared to the maximal estimate (see Lemma \ref{de1} below). This is the second trick, which aids in eliminating the loss related to the scale $R$ in the upper bound of (\ref{T12}).

 \end{itemize}

We will now outline the proofs of our  results. The first  step in establishing both (\ref{t11}) in Theorem \ref{t1} and (\ref{T12}) in Theorem \ref{co1} is to reduce  the focus to the long and short variational estimates.

 \begin{itemize}

  \item

  Sketch of the proof of Theorem \ref{t1}:  The long and short variational inequalities  are formulated in (\ref{long}) and (\ref{short8}), respectively.
  Due to the presence of $\lambda(x)$,   it is hard to employ 
    the approach in \cite{MST17}   to bound the $r$-variation for the operator (\ref{ope1}).
    To address this, we will initially adopt the strategy from \cite{KR22}, dividing the multiplier into a number-theoretic approximation and an error term (this procedure goes back to  Bourgain \cite{Bour89}, and is an application of the Hardy-Littlewood circle method). 
    By combining   minor arcs estimates from \cite{KR22} and a numerical inequality (see  (\ref{k01}) below) from \cite{MST17}, we can establish desired minor arcs estimates (see (\ref{minor1}), (\ref{minor2}), and Proposition \ref{addpp1} below) in this paper. As a result,
    we reduce  the matter   to proving  major arcs estimates I, II and III (see  Propositions \ref{t21}, \ref{892}, and \ref{t31} below).
%While major arcs estimates are pivotal in the proof, the reduction facilitated by  minor arcs estimates from \cite{KR22} holds significant importance.

    To prove   major arcs estimates I and II with respect  to the long variation, we conduct a direct analysis of three associated  multipliers. This involves establishing a multi-frequency square function estimate (see Lemma \ref{ccz} below) and two multi-frequency variational inequalities (see Lemmas \ref{ccz2} and \ref{endle} below). The scale loss in the upper bound of (\ref{t11}) arises from these variational inequalities. Additionally, we utilize  a transference principle (see (\ref{cv1})) proven in \cite{MST17} and a Rademacher-Menshov-type inequality to support our analysis.
For major arcs estimate III with respect  to the short variation, we rely on the above mentioned  multi-frequency square function estimate and the shifted square function estimate detailed in the Appendix. Furthermore, the proof benefits from two maximal estimates  established by Krause and Roos in \cite{KR22, KR23} (see  Lemma \ref{de1} below), along with the application of the Stein-Wainger-type theorem. 

    \item

 Sketch of the proof of Theorem \ref{co1}: The proof mirrors the arguments leading to Theorem \ref{t1}, with  (\ref{dou12}) replaced by a new estimate (\ref{DDou12}), which removes the loss related to the scale  $R$. This inequality (\ref{DDou12}) is established by amalgamating Lemma \ref{L147}, a more robust rendition of Lemma \ref{de1}, with the Gauss sum bounds.

  \end{itemize}

\subsection{Organization}
In Section \ref{pres2}, we introduce some important theorems, inequalities  and  related notations used in the following proofs of our main results.   In Section \ref{pth1}, we give the proof of  Theorem \ref{t1} and make a crucial  reduction of  (\ref{t11}); we shall use the minor arcs estimate obtained by Krause and Roos \cite{KR22} as a black box, and reduce the proof of Theorem \ref{t1} to proving three major arcs estimates given by Propositions \ref{t21}, \ref{892} and \ref{t31}. In Section \ref{Fpre}, we provide crucial auxiliary results for establishing these  major arcs estimates.   In   Section \ref{slong2},  Section \ref{slong3} and  Section \ref{slong4},
we prove Proposition \ref{t21}, Proposition \ref{892} and Proposition \ref{t31}  in order.
In Section \ref{slong5}, we prove
 Theorem  \ref{co1}. In Appendix \ref{app}, we provide a shifted square estimate used in the proof of Lemma \ref{l763} (Section \ref{slong4}). Finally, Appendix \ref{appendixB} contains the proof of the second minor arcs estimate \eqref{azq40}.

\subsection{Notation.} \label{Not}
We use the Japanese bracket notation $ \langle x \rangle := (1 + |x|^2)^{1/2}$ for any real or complex $x$.
   For any two quantities $x,y$ we will write $x\lesssim y$  to denote   $x\le Cy$ for some absolute constant $C$. The notation $A=B+\mathcal{O}(X)$ means
   $|A-B|\les X$.
   If we need the implied constant $C$
to depend on additional parameters,  we will denote this by subscripts. 
 If both $x\lesssim y$ and $y\les x$ hold, we use $x\sim y$. To abbreviate the notation we will sometimes permit the implied constant to depend on certain fixed parameters when the issue of uniformity with respect to such parameters is not of relevance. The constant $C$ may vary at each appearance in this paper.
 
 We denote the positive integers by $\N:=\{1,2,\dots\}$ and the natural numbers by  $\N_0:=\N\cup\{0\}$.  The set of dyadic numbers is defined as
 $\mathcal{D}=\{2^n:n\in\N_0\}$.
 For any $a> 0$, $\lfloor a\rfloor$ denotes the largest integer not larger than $a$.
 For any $N > 0$, we use $[N]$ or $\N_N$ to denote the discrete interval $\{ n \in \N: n \leq N \}$.   If $a,q \in \N$, we let $(a,q)$ denote the greatest common divisor of $a$ and $q$.
  Moreover, $\ind{E}$ denotes the indicator function of a set $E$, that is, $\ind{E}(x):= \ind{x \in E}$. 

 We use $f*g$ and $f*_{\R^n}g$ to represent the convolution  on $\Z^n$ and $\R^n$, respectively, that is,
$$f*g(x):=\sum_{y\in\Z^n}f(x-y)g(y)\ \ (x\in\Z^n)\ \ {\rm and}\ \  f*_{\R^n}g(z):=\int_{\R^n}f(z-y)g(y)dy\ \ (z\in\R^n).$$
 We denote by  $M_{HL}$   the classical  Hardy-Littlewood maximal operator on $\R^n$, and by   $M_{DHL}$ the discrete Hardy-Littlewood maximal operator on $\Z^n$. For each $S\subset \Z$, we utilize
$\|(a_k)_{k\in S}\|_{\ell^r}$ or $\|a_k\|_{\ell^r(k\in S)}$ to denote $(\sum_{k\in S}|a_k|^r)^{1/r}$  if $r<\infty$, and use $\|(a_k)_{k\in S}\|_{\ell^\infty}$ or  $\|a_k\|_{\ell^\infty(k\in S)}$ to denote $\sup_{k\in S}|a_k|$.
%  And we use $f*_{\R^n}g$ to represent the convolution between $f$ and $g$ on $\R^n$, which is defined  by  $(f*_{\R^n}g)(x):=\int_{\R^n}f(x-y)g(y)dy$.
 %and   $|S|$   represents the  Lebesgue measure of the set $S$.
%apply
%$\|\cdot\|_p$ to stand for $\|\cdot\|_{L^p(\R^n)}$ in some places of this paper.
Throughout   this paper, we fix a cutoff function $\psi:\R^n\to [0,1]$,
%and $\psi_\circ:\R\to [0,1]$,
which   is supported  in $\{\xi\in \R^n:\ 1/2\le |\xi|\le 2\}$,
%,  and $\psi_\circ$ equals one on supp$\psi$.
and set $\psi_l(\xi):=\psi(2^{-l}\xi)$ for any $l\in\Z$ such that  the partition of unity $\sum_{l\in\Z}\psi_l(\xi)=1$ holds for all $\xi\in \R^n\setminus \{0\}$.  Moreover, we also need another   partition of unity
 $\chi(\xi)+\sum_{l\ge 1}\psi_l(\xi)=1$  for all $\xi\in \R^n$, which implies that
 $\chi(\xi)=\sum_{l\le 0}\psi_l(\xi)$ whenever $\xi\in \R^n\setminus \{0\}$.
For each $j\in\Z$, we denote by $P_j$ the Littlewood-Paley projection on $\R^n$, which is  defined by
%Littlewood-Paley decomposition $f=\sum_{j\in\Z}P_j f$, where
 %$\{P_j\}_{j\in\Z}$ are the usual  Littlewood-Paley projections on $\R^n$ with
$\widehat{P_jf}(\xi):=\psi_j(\xi)\widehat{f}(\xi)$.% we use $\tilde{P}_{j}$ which may vary line by line  to denote the  variant of the Littlewood-Paley operator ${P}_{j}$.
% whenever $j\in\Z$.
 %%
 % \subsection{\bf Acknowledgements.}
% \section*{Acknowledgements}
 %The author would like to thank  Prof. Jiecheng Chen and Prof. Meng Wang for discussions related to the topic of this paper.
 %This work was supported by the NSF of China 11901301.
 %%
 \vskip.1in
\section{Preliminaries}
\label{pres2}
\subsection{Fourier transforms and Fourier multipliers}
For Fourier transform of functions $f:\Z^n\to \mathbb{C}$, $g:\T^n\to \mathbb{C}$, we use the notations
 $$\widehat{f}(\xi)=\F_{\Z^n}f(\xi):=\sum_{x\in\Z^n} e(-\xi\cdot x)f(x),\ \ \ \
 \  \F_{\Z^n}^{-1}(g)(x):=\int_{\T^n} e(\xi\cdot x) g(\xi) d\xi,
 $$
 where $\T^n=(\R/ \Z)^n$. For Fourier transform of  function $h:\R^n\to \mathbb{C}$,  we write
  $$\widehat{h}(\xi)=\F_{\R^n}h(\xi):=\int_{\R^n}e(-\xi\cdot x)h(x)dx,\ \ \ \
 \  \check{h}(x)=\F_{\R^n}^{-1}(h)(x):=\widehat{h}(-x).
 $$
 In particular,  we will denote by $\widehat{f}$ the Fourier transform of $f$ on $\Z^n$ or $\R^n$ unless
   the distinction is not  clear from the context or is emphasized for other reasons.

For a bounded  function $m:\R^n\to \mathbb{C}$,  we define   
%denote by   $m(D)$ the associated   Fourier
%multiplier  on $\R^n$, which is %given by
\beq\label{multi2}
m(D)g(x):=\ \F^{-1}_{\R^n}(m~\F_{\R^n}g)(x)\ \ \ \ \ \  (x\in\R^n).
\eeq
In addition, if $m$ is 1-periodic, we also let 
%denote by   $m(D)$ the corresponding  Fourier
%multiplier  on $\Z^n$, which is defined as
\beq\label{multi1}
m(D)f(x):=\ \F^{-1}_{\Z^n}(m~\F_{\Z^n}f)(x)\ \ \ \ \ \ (x\in\Z^n).
\eeq
It will always be clear from the context which one is meant.

\subsection{$V^r$, $\V^r$  and related inequalities}
\label{semi}
Let $1 \leq r < \infty$.
For  any sequence
 $(\mathfrak{a}_t)_{t\in \I}$ of complex number with $\I\subset \Z$, the
 $r$-variation  seminorm is defined by the formula
 \begin{equation}\label{var-seminorm}
  \| (\mathfrak {a}_t)_{t \in \I} \|_{V^r}:=
 \sup_{J\in\N} \sup_{\substack{t_{0} <  \dotsb < t_{J}\\ \{t_{j}\}\subset\I}}
\big(\sum_{j=0}^{J-1}  |\mathfrak {a}_{t_{j+1}}-\mathfrak {a}_{t_{j}}|^{r} \big)^{1/r},
 \end{equation}
where the  supremum is taken over all finite increasing sequences in $\mathbb I$, and is set by convention to equal zero if $\I$ is empty.   This seminorm $V^r$ governs the supremum  norm as follows: For  any $t_0\in \I$,
\beq\label{Ad1}
\sup_{t\in\I}|\mathfrak {a}_t|\le |\mathfrak {a}_{t_0}|+\| (\mathfrak {a}_t)_{t \in \I} \|_{V^r}.
\eeq
%Taking limits as $r \to \infty$ we also adopt the convention
 %\begin{align*}
 %\| (\mathfrak{a}_t)_{t \in \I} \|_{V^\infty(\I;B)}:= \sup_{t \leq t' \in \I} \|\mathfrak{a}(t') - \mathfrak{a}(t)\|_B.
 %\end{align*}
 Let $\mathcal{B}\subset \N$.
 The long variation seminorm $V^r_L$ of a sequence $\big(\mathfrak {a}_j: j\in \mathcal{B}\big)$ is defined  by
$$\|(\mathfrak {a}_j)_{j\in \mathcal{B}}\|_{V^r_L}:=\|(\mathfrak {a}_j)_{j\in \mathcal{B}\cap \mathcal{D}}\|_{V^r},$$
while
the associated  short variation  seminorm $V^r_S$  is given by
$$\|(\mathfrak {a}_j)_{j\in \mathcal{B}}\|_{V^r_S}:=\big(\sum_{n\in\N_0}\|(\mathfrak {a}_j)_{j\in \mathcal{B}_n}\|_{V^r}^r\big)^{1/r},\ \ {\rm where}\ \    \mathcal{B}_n:=\mathcal{B}\cap [2^n,2^{n+1}).$$
We can reduce the $r$-variation seminorm estimate to bounding the long and short variation seminorm estimates by the following inequality:
\beq\label{redu21}
\|(\mathfrak {a}_j)_{j\in \N}\|_{V^r}
\les \|(\mathfrak {a}_j)_{j\in \N}\|_{V^r_L}+\|(\mathfrak {a}_j)_{j\in \N}\|_{V^r_S}.
\eeq
For the proof of (\ref{redu21}), we refer \cite{JSW08}.
Next, we introduce two  numerical   inequalities, which play an important role in proving  our main results.
\begin{prop}\label{number}
(i) (Rademacher-Menshov inequality) Let $\mathfrak{s}\in\N$ and  $2\le r<\infty$. For any sequence $(\mathfrak {a}_j:0\le j\le 2^\mathfrak{s})$ of complex numbers, we
have
\beq\label{num1}
\|(\mathfrak {a}_j)_{j\in[0,2^\mathfrak{s}]}\|_{V^r}\le \sqrt{2}\sum_{i=0}^\mathfrak{s}\big( \sum_{j=0}^{2^{\mathfrak{s}-i}-1}|\mathfrak {a}_{(j+1)2^i}-\mathfrak {a}_{j2^i}|^2\big)^{1/2}.
\eeq
(ii) Let $1\le r\le p<\infty$ and $v-u\ge 2$ with $u,v\in\N$. If  $\{f_j:j\in\N\}$ is a sequence of functions in $\ell^p(\Z^n)$,  then we have 
\beq\label{k01}
\big\|\|(f_j)_{ j\in[u,v]}\|_{V^r}\big\|_{\ell^p(\Z^n)}
\les \max\big\{U_p,~(v-u)^{1/r}U_p^{1-1/r}V_p^{1/r}\big\},
\eeq
where
$U_p:=\max_{u\le j\le v}\|f_j\|_{\ell^p(\Z^n)}$ and $V_p:=\max_{u\le j< v}\|f_{j+1}-f_j\|_{\ell^p(\Z^n)}$.
\end{prop}
%To handle the short variation,
%we also need the following inequality (see  %Lemma 2.2 or (2.8) in  \cite{MST17}):
%%
(\ref{num1}) originates in \cite{LL12}.
For the  proofs of (\ref{num1}) and (\ref{k01}), we refer the arguments yielding  \cite[Lemma 2.1]{MST17} and   \cite[Lemma 2.2]{MST17}, respectively.  
In addition,  we refer \cite{BMSW18,BMSW19,MST199,JKRW98,ZK15} for some applications of these numerical inequalities and  other related numerical inequalities.

The above two numerical inequalities are efficient in many works dealing with discrete operators,
however, it is insufficient  for bounding  the operator (\ref{ope1}) in the present paper. As we shall see later in controlling   the short variation,
we  will also require the utilization of  the Besov norm, commonly employed  in  establishing the variational inequalities  on $\R^n$.
%This is a little beyond our expectations. 
More precisely,
from the Plancherel-P\'{o}lya inequality \cite{PP36,PP37}, it can be observed that for all $r\in [1,\infty)$, $B_{r,1}^{1/r}\hookrightarrow V^r\hookrightarrow B_{r,\infty}^{1/r},$ where the notation $B_{p,q}^s$ represents the inhomogeneous Besov space (see  \cite{Gra14,BCD}).
By utilizing the first embedding and recognizing the convenience of working with Besov space, it is sufficient to manage the $B_{r,1}^{1/r}$ norm to control the seminorm $V^r$ sometimes. %However, as we will later see, some crucial modifications, derived from the definition of the variational seminorm, are needed before bounding this norm.
Furthermore, by  the fundamental theorem of calculus, we deduce   that for all $r\in [1,\infty)$,
\beq\label{rou1}
\|(\mathfrak {a}_u )_{u\in \mathcal{K}}\|_{V^r}
\le \|\p_u (\mathfrak {a}_u)\|_{L^1(u\in \mathcal{K})}
\eeq
 whenever  $\mathcal{K}$  is  an interval; this inequality (\ref{rou1}) is  used in  bounding the  short variation  as well.
For convenience, we also introduce the $r$-variation norm for $1 \leq r \leq \infty$ defined by
\begin{equation}\label{vardef}
  \| (\mathfrak {a}_t)_{t \in \I} \|_{\V^r}:=\sup_{t\in\I}|\mathfrak a_t|+
\| (\mathfrak {a}_t)_{t \in \I} \|_{V^r}.
\end{equation}
%This clearly defines a norm on the space of functions from $\I$ to $B$.
%If $B=\C$, then we will abbreviate $V^r(\I;X)$ to $V^r(\I)$ or $V^r$, and $\V^r(\I;X)$ to $ \V^r(\I)$ or $\V^r$.
%If $(X,\mu)$ is a measure space, then  using \eqref{vardef}, one can explicitly write
%\[
%L^p(X;\V^r)=\left\{F\in L^0(X;\V^r):\|F\|_{L^{p}(X;\V^r)}:=\left\|\|F\|_{\V^r}\right\|_{L^{p}(X)}<\infty\right\}.
%\]
%Note that the $\V^r$ norm is non-decreasing in $r$, and comparable to the $\ell^\infty$ norm when $r=\infty$.
Observe that the simple triangle inequality
\begin{equation}\label{simple}
\| (\mathfrak {a}_t)_{t \in \I} \|_{\V^r} \lesssim \| (\mathfrak {a}_t)_{t \in \I_1} \|_{\V^r} + \| (\mathfrak {a}_t)_{t \in \I_2} \|_{\V^r}
\end{equation}
holds whenever $\I = \I_1 \uplus \I_2$ is an ordered partition of $\I$,
% (thus $t_1<t_2$ for all $t_1 \in \I_1, t_2 \in \I_2$).
and 
\begin{equation}\label{varsum}
 \| (\mathfrak {a}_t)_{t \in \I} \|_{\V^r} \lesssim \| (\mathfrak {a}_t)_{t \in \I} \|_{\ell^r}
 \leq \| (\mathfrak {a}_t)_{t \in \I} \|_{\ell^1}.
 \end{equation}
From H\"older's inequality one easily establishes the algebra property
\begin{equation}\label{var1}
\| (\mathfrak {a}_t \mathfrak {b}_t)_{t \in \I} \|_{\V^r} \lesssim \| (\mathfrak {a}_t)_{t \in \I} \|_{\V^r} \| (\mathfrak {b}_t)_{t \in \I} \|_{\V^r}
\end{equation}
for any scalar sequences $(\mathfrak {a}_t)_{t \in \I}$ and $(\mathfrak {b}_t)_{t \in \I}$.
For  any sequence
 $(\mathfrak{f}_t(x))_{t\in \I}$ of complex-valued function defined on $X$, where
  $X$ denotes $\Z^n$ or $\R^n$,
  we will frequently use  the following notations:
$$\|(\mathfrak{f}_t)_{t\in \I}\|_{L^{p}(X;\V^r)}:=\left\|\|(\mathfrak{f}_t)_{t\in \I}\|_{\V^r}\right\|_{L^{p}(X)},
\ \|(\mathfrak{f}_t)_{t\in \I}\|_{L^{p}(X;V^r)}:=\left\|\|(\mathfrak{f}_t)_{t\in \I}\|_{V^r}\right\|_{L^{p}(X)},$$
where   $L^p(\Z^n)$ represents $\ell^p(\Z^n)$.

\subsection{Ionescu-Wainger-type multiplier theorem}
We call a set $\Theta\subset \R^n$ periodic if $z+\Theta=\Theta$ for all $z\in\Z^n$, where
$z+\Theta=\{x\in \R^n:\ x=z+x'\ {\rm for\ some} \ x'\in \Theta\}$.
For any bounded function $m$ on $\R^n$ and any  periodic set $\Theta \subset \Q^n$, we define the associated  multi-frequency multiplier
$$\Delta_{\Theta}[m](\xi):=\sum_{\theta\in \Theta}m(\xi-\theta).$$
 For any set $S\subset \N$, we define
$$\mathcal{R}(S)=\{a/q\in \Q^n:\ (a,q)=1,\ q\in S\}.$$
Let $\eta$ be a compactly supported and smooth function, which equals 1 on $\{\xi\in\R^n:
|\xi|\le 1/2\}$. Denote $\eta_v(\xi)=\eta(\xi/v)$ with $0\neq v\in\R$.
\begin{prop}\label{PIW}
%Let $m$ be a bounded function on $\R^n$.
Suppose that for every $p\in (1,\infty)$, there exists a positive constant  $A_p$ such that
$$\|m(D)f\|_{L^p(\R^n)}\le A_p \|f\|_{L^p(\R^n)}.$$
 For each  $\kappa>0$ and  every $N\in \N$,  there exists a periodic set $\mathcal{U}_{N,\kappa}\subset \Q^n$ satisfying
$$\mathcal{R}(\N_N)\subset \mathcal{U}_{N,\ka} \subset \mathcal{R}(\N_{e^{N^\kappa}})$$
such that  for every $p\in(1,\infty)$,
\beq\label{aaB1}
\|\Delta_{\mathcal{U}_{N,\ka}}[m ~\eta_{e^{-N^{2\kappa}}}](D)f\|_{\ell^p(\Z^n)}
\les_{\kappa,p} A_p\|f\|_{\ell^p(\Z^n)}.
\eeq
\end{prop}
For the construction of  $\mathcal{U}_{N,\kappa}$, we refer  Section 3.4 in \cite{MST17}.
 In 2005, Ionescu and Wainger  \cite{IW05} initially  proved (\ref{aaB1}) with a logarithmic loss in $N$. Mirek \cite{MM18} weakened this logarithmic loss in $N$ later, and   Tao
\cite{TT21}
finally
removed this logarithmic loss in $N$, and
established (\ref{aaB1}) with the upper bound independent of  $N$.  In fact,  this logarithmic loss in $N$ is not crucial in the proofs of our results.
%Following the proof of Proposition \ref{PIW} line by line, one may obtain a more efficient variant for $p\ge 2$.  However, we do not know what happens for $p\in (1,2)$ without any additional conditions.  The situation for $p\in (1,2)$  is quite different from that in  Proposition \ref{PIW}, because   naive dual arguments  do not work  at this moment.
%We omit the details
%\begin{prop}\label{PIW2}
%Let   $\kappa>0$,  $N\in \Z_+$,
 %and let $\mathcal{U}_{N,\kappa}$ be given as in Proposition \ref{PIW}.
%Suppose that the functions $\{m_t\}_{t\in [0,1]}$    satisfy that, for each $p\in [2,\infty)$,
%$$\|m_t(D)f(x)\|_{L^p_{x,t}(\R^n\times [0,1])}\le A_p \|f\|_{L^p(\R^n)}$$
%for some $A_p>0$.  Then for every $p\in[2,\infty)$,
%$$\|\Delta_{\mathcal{U}_{N,\ka}}[m _t\ \chi(2^{N^{2\kappa}}\cdot)](D)f\|_{\ell^p_{x,t}(\R^n\times [0,1])}
%\les_{\kappa,p} A_p\|f\|_{\ell^p(\Z^n)}.$$
%\end{prop}
%\vskip.4in
\subsection{Transference principle by Mirek, Stein and Trojan}
\label{TMST2}
 Let $\eta_\circ:\R^n\to \R$ be a smooth function  such that $\eta_\circ\in [0,1]$ is supported in $\{|x|\le 1/(8n)\}$,
 and $\eta_\circ(x)=1$ on $|x|\le 1/(16 n)$. Let $\{\Theta_N:N\in\N\}$ be a sequence of multipliers on $\R^n$ satisfying  that, for each $p\in(1,\infty)$ and each $r\in (2,\infty)$, there is a positive constant $B_{p,r}$ such that
 \beq\label{cv1}
\|\big(\Theta_N(D)f\big)_{N\in\N}\|_{L^p(\R^n;V^r)}
\le B_{p,r}\|f\|_{L^p(\R^n)}.
\eeq
Assume that $\mathfrak{R}$ is a diagonal $n\times n$ matrix with positive entries $(r_\gamma:\gamma\in \Gamma)$ such that $\inf_{\gamma\in\Gamma}r_\gamma\ge \mathfrak{h}$ for  $\mathfrak{h}>0$. We list the following version of the transference principle provided  by Mirek-Stein-Trojan \cite[Proposition 3.1]{MST17} (see  \cite{MT16,MST199} for its proof).
\begin{prop} \label{PMST}
Let $ p\in  (1,\infty)$,  $r\in  (2,\infty)$, and
suppose that  (\ref{cv1}) holds. Then for each
 $Q\in\N$ and $\mathfrak{h}\ge 2^{2n+2}Q^{n+1}$ and any $\mathfrak{m}\in \N_Q^n$,
 $$\|\Big(\F^{-1}_{\Z^n}\big(\Theta_N ~\eta_\circ(\mathfrak{R}\cdot) \hat{f}\big)(Qx+\mathfrak{m})\Big)_{N\in\N}\|_{\ell^p(x\in \Z^n;V^r)}
\les B_{p,r}\|\F^{-1}_{\Z^n}\big(\eta_\circ(\mathfrak{R}\cdot)
\hat{f}\big)(Qx+\mathfrak{m})\|_{\ell^p(x\in \Z^n)}$$
with $B_{p,r} $ given as in (\ref{cv1}).
\end{prop}
Obviously,  we can  infer from the case $Q=1$ and $\mathfrak{m}\in \N_1^n$ that Proposition \ref{PMST}  also  holds for the case $Q=1$ and $\mathfrak{m}=0$, which will be used
 in the following context.

\section{Proof of Theorem \ref{t1} and reduction of  (\ref{t11})}
\label{pth1}
In this section, we prove Theorem \ref{t1}  by assuming that    the desired  associated long and short variational inequalities hold, and then  give reductions of these assumed inequalities.

Let  $[p_1,p_2]$ denote    an arbitrary closed interval with $1<p_1<2<p_2<\infty$.\footnote {In this paper,  
 $p\in [p_1,p_2]$ means that  $p$ belongs  to  an arbitrary closed interval $ [p_1,p_2]$, where   $1<p_1<2<p_2<\infty$.}
To prove (\ref{t11}), by interpolation,
 it suffices  to show that for  each $p\in[p_1,p_2]$ and every $r\in (2,\infty)$,
\beq\label{t12}
\|(\mathscr{C}_{N} f)_{N\in\N}\|_{\ell^p(\B_R;V^r)}\les_{\e} R^{\e} \|f\|_{\ell^p(\Z^n)}
\eeq
for all $R\ge 1$.
Indeed, by interpolating (\ref{t12}) with
the case $r=\infty$ (namely,  the  maximal estimate obtained   by Krause-Roos \cite{KR23}, which is independent of  $R$),   we achieve   (\ref{t11}) immediately.  As a consequence,  we reduce the matter to proving the above   (\ref{t12}).
By a standard process  (\ref{redu21}),   we can achieve  (\ref{t12})  from the following inequalities: for  each $p\in[p_1,p_2]$ and every $r\in (2,\infty)$,
% the following   inequalities: %For every $(p,r)\in (1,\infty)\times (2,\infty)$,
 \begin{align}
    \|(\mathscr{C}_{2^j} f)_{j\in\N}\|_{\ell^p(\B_R;V^r)}&\les_{\e}  R^{\e}  \|f\|_{\ell^p(\Z^n)}\ \ {\rm and}  \label{long}\\
  \|\big( \sum_{j\ge 0}\|(\mathscr{C}_{N} f-\mathscr{C}_{2^j} f)_{N\in [2^j,2^{j+1})}\|_{V^2}^2\big)^{1/2}\|_{\ell^p(\Z^n)}&\les\  \|f\|_{\ell^p(\Z^n)},\label{short8}
 \end{align}
 where (\ref{long})   and (\ref{short8}) are the long variational inequality and  the short variational inequality, respectively. In other words,  we can prove Theorem \ref{t1}  under the assumptions that
 (\ref{long})  and (\ref{short8}) hold. Thus, it remains to prove (\ref{long})  and (\ref{short8}).
 In the followed subsections, we will reduce the proofs of  (\ref{long})  and (\ref{short8})
 to showing three major arcs estimates given by  Propositions \ref{t21}, \ref{892} and \ref{t31} below.
\subsection{General operators and minor arcs estimates}
\label{ss4.1}
Let
$N_\circ$ and $\Pi$ be two large \footnote{In the following context, we only need   $ \Pi\ge C_0$ with $C_0$ given by (\ref{df1}).} positive integers
with $cN_\circ\le \Pi\le N_\circ$, where  $0<c<1$. Let $\la(x)$ be an arbitrary  function from $ \Z^n$ to [0,1],   let
$$\bar{\KK}_{\Pi,N_\circ}(y)=\KK(y){\ind {\Pi\le |y|\le N_\circ}}$$
 with $\KK:\R^n\to \R$
 %the function ${\KK}:\R^n\to \R$
 satisfying
%with $\Om_{M,N}:=\{y:M\le |y|\le N\}$ which satisfies
\beq\label{s1}
|{\KK}(y)|+N_\circ |\na {\KK}(y)|\les N_\circ^{-n}\ \ \ {\rm for\ all}\ \Pi\le |y|\le N_\circ,
\eeq
and define a family of periodic multipliers
\beq\label{az11}
m_{\Pi,N_\circ,v}(\xi)=\sum_{y\in\Z^n}e(v|y|^{2d}+y\cdot\xi)\bar{\KK}_{\Pi,N_\circ}(y),
\ \ \ \ v\in\R,\ \ \xi\in\R^n,
\eeq
where  the function $\bar{\KK}_{\Pi,N_\circ}$ satisfies that for every $q\in [1,\infty]$,
\beq\label{kere}
\||\bar{\KK}_{\Pi,N_\circ}|*|f|\|_{\ell^q(\Z^n)}%\les \|M_{DHL}f\|_{\ell^q(\Z^n)}
\les  \|\bar{\KK}_{\Pi,N_\circ}\|_{\ell^1(\Z^n)}\|f\|_{\ell^q(\Z^n)}
\les  \|f\|_{\ell^q(\Z^n)}.
\eeq
%where $M_{DHL}$ is the discrete Hardy-Littlewood maximal operator.
We shall
consider   the function
\beq\label{aa12}
\big(m_{\Pi,N_\circ,\la(x)}(D)f\big)(x):=\F^{-1}_{\Z^n}\big(m_{\Pi,N_\circ,\la(x)}~\F_{\Z^n} f \big)(x),
\eeq
%where the multiplier $m_{\Pi,N_\circ,\la(x)}$ is given by
where the notation (\ref{multi1}) is used, and
the multiplier   $m_{\Pi,N_\circ,\la(x)}$ is defined as  (\ref{az11}) with ($v=\la(x)$).
As the multiplier depends  on the variable $x$ in this instance, the scenario becomes  more complex than situations where it remains independent of $x$.
To show  the desired   result, we introduce first  the associated exponential sums  of the above multiplier:
$$S(\frac{a}{q},\frac{b}{q})=\frac{1}{q^n}\sum_{r\in [q]^n}e(\frac{a}{q}|r|^{2d}+\frac{b}{q}\cdot r),$$
where $a/q\in \Q $ and $b/q\in\Q^n$ satisfy  $(a,b,q)=1$ (otherwise the notation  $S(a/q,b/q)$ is not well-defined).
Let  $\Phi_{\Pi,N_\circ,v}$ be the  real-variable version of  the multiplier  (\ref{az11}) defined by
\beq\label{con1}
\Phi_{\Pi,N_\circ,v}(\xi)=\int_{\R^n} e(v |y|^{2d}+y\cdot\xi)\bar{\KK}_{\Pi,N_\circ}(y) dy.
%\ \ \ \ v\in\R,\ \ \xi\in\R^n.
\eeq
Below
we list a basic approximation result  for the multiplier  $m_{\Pi,N_\circ,\la(x)}(\xi)$.% defined by  (\ref{az11}). %by  $\Phi_{\Pi,N_\circ,u(x)}$.
\begin{prop}\label{p12}
Let $0<c<1$ and $q\in\N$. Let   $N_\circ$ and $\Pi$  be two large positive constants satisfying $cN_\circ\le \Pi\le N_\circ$   and  $q\le c \sqrt{N_\circ}/8$.  Let  $a\in\Z$ and $b\in \Z^n$ with $(a,b,q)=1$.  Denote
$$\mathscr{J}_{\Pi,N_\circ,a,b,q}:=\big\{(x,\xi)\in (\Z^n,\T^n):\ |\la(x)-a/q|\le \de N_\circ^{-(2d-1)},\  |\xi-b/q|\le \de\big\},$$
with $\de\in (N_\circ^{-1},1)$. Then  for each $(x,\xi)\in \mathscr{J}_{\Pi,N_\circ,a,b,q}$, 
$$m_{\Pi,N_\circ,\la(x)}(\xi)=
S(a/q,b/q)~\Phi_{\Pi,N_\circ,\la(x)-a/q}(\xi-b/q)+\mathcal{O}(\de q) $$
with the implicit constant independent of  $\Pi,N_\circ,a,b,q$ and $\la(x)$.
\end{prop}
\begin{proof}
It suffices to show that
$$m_{\Pi,N_\circ,u}(\xi)=S(a/q,b/q)\Phi_{\Pi,N_\circ,u-a/q}(\xi-b/q)+\mathcal{O}(\de q) $$
whenever $|u-a/q|\le \de N_\circ^{-(2d-1)}$ and  $|\xi-b/q|\le \de$.
%%%%%%%
%Since $\bar{\KK}_{\Pi,N_\circ}$ is supported in $\{y\in \Z^n: \Pi\le |y|\le N_\circ\}$,
We may
rewrite  $m_{\Pi,N_\circ,u}(\xi)$  as follows:
$$
\begin{aligned}
&\sum_{r\in [q]^n}\sum_{y\in W_{\Pi,N_\circ,q,r}}e\big(u|qy+r|^{2d}+(qy+r)\cdot\xi\big){\KK}(qy+r)\\\
=&\ q^{-n}\sum_{r\in [q]^n}\Big( e(a |r|^{2d}/q+r\cdot b/q)\  I_{\Pi,N_\circ,q,r}(u-a/q,\xi-b/q)\Big),
\end{aligned}$$
where   $W_{\Pi,N_\circ,q,r}$ and  $I_{\Pi,N_\circ,q,r}$ are defined by
\beq\label{DEP1}
\begin{aligned}
W_{\Pi,N_\circ,q,r}:=&\ \{y\in\Z^n:\ \Pi\le |qy+r|\le N_\circ\}\  \ \ \ \ \ {\rm and}\\
I_{\Pi,N_\circ,q,r}(\eta,\nu):=&\ q^n
\sum_{y\in W_{\Pi,N_\circ,q,r}}\mathscr{B}_{\eta,\nu,q,r}(y) {\KK}(qy+r)
\end{aligned}
\eeq
with
$\mathscr{B}_{\eta,\nu,q,r}(y):=e\big(\eta|qy+r|^{2d}+(qy+r)\cdot \nu \big)$.
Then we further reduce the matter to proving  that, for each $r\in [q]^n$,
\beq\label{101}
|I_{\Pi,N_\circ,q,r}(\eta,\nu)-\Phi_{\Pi,N_\circ,\eta}(\nu)|\le q\de
\eeq
whenever $|\eta|\le \de N_\circ^{-(2d-1)}$ and $|\nu|\le \de$.  Changing variables $y\to qy+r$, we write $\Phi_{\Pi,N_\circ,\eta}(\nu)$ as
\beq\label{vm1}
\Phi_{\Pi,N_\circ,\eta}(\nu)=q^n\int_{\Pi\le |qy+r|\le N_\circ}
\mathscr{B}_{\eta,\nu,q,r}(y)  \KK(qy+r)dy.
\eeq
Claim  that the right-hand side of (\ref{vm1}) equals
\beq\label{DEP2}
q^n\sum_{y\in W_{\Pi,N_\circ,q,r}}\int_{y+[-{1}/{2},{1}/{2}]^n}
\mathscr{B}_{\eta,\nu,q,r}(y') \KK(qy'+r)dy'+ \mathcal{O}\big(q^n\int_{\mathcal{Y}_{\Pi,N_\circ,q,r}}|\KK(qy+r)|dy\big),
\eeq
where the set  $\mathcal{Y}_{\Pi,N_\circ,q,r}$ is given by
$$\mathcal{Y}_{\Pi,N_\circ,q,r}:= \{y\in\R^n:  \big||qy+r|-N_\circ\big|\le 2q\ \ {\rm or}\ \big||qy+r|-\Pi\big|\le 2q\}.$$
Let $\mathfrak{S}_{1},\mathfrak{S}_{2}$ be two sets given by
$$
\begin{aligned}
\mathfrak{S}_{1}:=&\ \mathfrak{S}_{1,\Pi,N_\circ}^{q,r}=\{y\in\R^n:\ \Pi\le |qy+r|\le N_\circ\},\\\mathfrak{S}_{2}:=&\ \mathfrak{S}_{2,\Pi,N_\circ}^{q,r}=\bigcup_{y\in W_{\Pi,N_\circ,q,r}}\{y+[-{1}/{2},{1}/{2}]^n\}.
\end{aligned}$$
Since $q\le c \sqrt{N_\circ}/8$,
%this inequality for enlarge the domain \mathfrak{S}_1 below
the above  claim follows  from the observation    that  the sets
$\mathfrak{S}_1\setminus \mathfrak{S}_2$  and $\mathfrak{S}_2\setminus \mathfrak{S}_1$ contained in two narrow annuli  $\mathcal{Y}_{\Pi,N_\circ,q,r}$
near  two spheres
$|qy+r|=N_\circ$ and $|qy+r|=\Pi$.
 Moreover,
simple  computation gives  that
  the measure of $\mathcal{Y}_{\Pi,N_\circ,q,r}$  is  $\les (N_\circ/q)^{n-1}$,
which with  (\ref{s1}) and $N_\circ^{-1}<\de$ leads to
$$q^n\int_{\mathcal{Y}_{\Pi,N_\circ,q,r}}|\KK(qy+r)|dy\les q^n(N_\circ/q)^{n-1} N_\circ^{-n}\les q/N_\circ \les  \de q.  $$
By combining (\ref{DEP1}) and (\ref{DEP2}),
to prove  (\ref{101}),   it suffices  to establish  that for all $y\in W_{\Pi,N_\circ,q,r}$,
\beq\label{102}
\begin{aligned}
&\ \ \Big|\mathscr{B}_{\eta,\nu,q,r}(y){\KK}(qy+r)-\int_{y+[-{1}/{2},{1}/{2}]^n}
\mathscr{B}_{\eta,\nu,q,r}(y') \KK(qy'+r)dy'\Big|\les q\de N_\circ^{-n},
\end{aligned}
\eeq
%\les \de q N^{-n},$$
where  $|qy+r|\sim |qy'+r|\sim N_\circ$ (since $q\le c \sqrt{N_\circ}/8$ and $|y-y'|\le 1/2$).  Note that
the left-hand side of (\ref{102}) is bounded  by the sum of
\begin{align}
&\ \Big|\int_{y+[-{1}/{2},{1}/{2}]^n}
\big\{\mathscr{B}_{\eta,\nu,q,r}(y')-\mathscr{B}_{\eta,\nu,q,r}(y)\big\} \KK(qy+r)dy'\Big|\ \  {\rm and}\label{71}\\
&\ \Big|\int_{y+[-{1}/{2},{1}/{2}]^n}
\mathscr{B}_{\eta,\nu,q,r}(y') \big\{ \KK(qy'+r)- \KK(qy+r)\big\}dy'\Big|.\label{72}
\end{align}
Since  $|\eta|\le \de N_\circ^{-(2d-1)}$, $|\nu|\le \de$ and $|qy+r|\sim |qy'+r|\sim N_\circ$,  the mean value theorem gives
\beq\label{po1}
|\mathscr{B}_{\eta,\nu,q,r}(y')-\mathscr{B}_{\eta,\nu,q,r}(y)|
\les q\de.
\eeq
In addition,  by the mean value theorem and (\ref{s1}), we also have
\beq\label{po2}
|\KK(qy'+r)- \KK(qy+r)|\les q N_\circ^{-n-1}\ \ \ {\rm and} \ \ \ |\KK(qy+r)|\les N_\circ^{-n}.
\eeq
Combining (\ref{po1}),  (\ref{po2}) and $N_\circ^{-1}<\de$  yields
$$(\ref{71})+ (\ref{72})\les q\de N_\circ^{-n},$$
which completes the proof of (\ref{102}).
\end{proof}
%We next  investigate (\ref{aa12}) by discussing its multiplier $m_{\Pi,N_\circ,\la(x)}$.
 %Let $d\ge 1$,
% Let $\e_\circ>0$ be a sufficiently small   number
 % (which will be determined later).
 Let   $j_\circ$ be a positive  integer such that $2^{j_\circ}\sim N_\circ$.
 %For
 %$\mathfrak{d}\in \Z_+$,  we d
We use the following notations:
\beq\label{notm1}
\begin{aligned}
\mathscr{S}_{j_\circ,\e_\circ}:=&\ \{a/q\in \Q:\ (a,q)=1,\ q\in [j_\circ^{\lfloor 1/\e_\circ \rfloor}]\},\\
X_{j_\circ,\e_\circ}:=&\ \bigcup_{\A\in \mathscr{S}_{j_\circ,\e_\circ}}\{u\in [0,1]:\ |u-\A|\le 2^{
-2dj_\circ}j_\circ^{\lfloor 1/\e_\circ \rfloor}\}\ \  {\rm and}\\
\Lambda_{j_\circ,\e_\circ,\la}:=&\ \{x\in\Z^n:~~\la(x)\in X_{j_\circ,\e_\circ}\},\ \ {\rm where}
\  \ 0<\e_\circ<1.
\end{aligned}
\eeq
In what follows, $x\notin \Lambda_{j_\circ,\e_\circ,\la}$  means $x\in \Z^n\setminus \Lambda_{j_\circ,\e_\circ,\la}$.
Repeating   the arguments yielding  \cite[Proposition 3.1]{KR23} (exponential sum estimates by Mirek, Stein and Trojan \cite{MST199} were used there,  see   \cite[Proposition 2.2]{KR23} for the details),   we  can   deduce that   for every $p\in [p_1,p_2]$ and  for  large enough $\mathcal{C} >0$ (which will be specified later), 
there is a sufficiently small constant $\e_\circ=\e_\circ(p_1,p_2,\mathcal{C})\in (0,1)$ such that 
% \footnote{The dependency on $d$ and $n$ for any such constants is omitted here and in what follows.} $\gamma=\gamma(\e_\circ)$ such that % for all $j_\circ\ge 1$ and  all  $N_\circ\sim 2^{j_\circ}$,
%for every sufficiently large constant  $\kappa>1$,
\beq\label{azq1}
\begin{aligned}
&\  \big\|{\ind {x\notin \Lambda_{j_\circ,\e_\circ,\la}}} |\big(m_{\Pi,N_\circ,\la(x)}(D)f\big)(x)| \big\|_{\ell^p(x\in\Z^n)}\\
\le&\  \big\|\sup_{\la\notin X_{j_\circ,\e_\circ}}|m_{\Pi,N_\circ,\la}(D)f|\big\|_{\ell^p(\Z^n)}\  \les~ j_\circ^{-\mathcal{C}} \|f\|_{\ell^p(\Z^n)}.
\end{aligned}
\eeq 
We call (\ref{azq1})
  the first minor arcs estimate for  (\ref{aa12}).
%For convenience, we will set
%$$\kappa=10\max\{p,p'\}.$$

Next, we show the second minor arcs estimate
 for  (\ref{aa12}).
 % In fact,  this so called  minor arcs estimate is an estimate for
%the associated error  of  the main part related to   major  arcs.  Here we also call it "minor arcs" as the statement in  \cite{KR22}.
For $s\in \N$, we define
\beq\label{Noo1}
\mathfrak{A}_s:=\{a/q\in \Q:\ (a,q)=1,\ \  q\in [2^{s-1},2^s)\cap \Z\}.
\eeq
For  each $\A=a/q\in \mathfrak{A}_s$, each bounded function  $m_\circ$ on $\R^n$,  and every $\kappa_1>0$,  we define
\beq\label{Noo2}
\mathscr{L}_{s,\A,\ka_1}[m_\circ](\xi):=\sum_{\bb\in \frac{1}{q}\Z^n}S(\A,\bb) m_\circ(\xi-\bb)\chi_{s,\kappa_1}(\xi-\bb)~~{\rm with }~~~\ \chi_{s,\kappa_1}(\xi):=\chi(2^{4s2^{2\kappa_1 s}}\xi),
\eeq
 %(the choice of $\chi_{s, \kappa}$  is due to  Proposition \ref{PIW}),
  %(see  Ionescu-Wainger \cite{IW05} and recent works)
 and
 let
\beq\label{Noo3}
\mathscr{L}^\#_{s,\ka_1}[m_\circ](\xi):=\sum_{\bb\in \mathcal{U}_{2^s,\ka_1}}m_\circ(\xi-\bb)\tilde{\chi}_{s,\kappa_1}(\xi-\bb)=\Delta_{\mathcal{U}_{2^s,\ka_1}}[m_\circ \tilde{\chi}_{s,\ka_1}](\xi),\eeq
with  the set ${\mathcal{U}_{2^s,\ka_1}}$  given as in Proposition \ref{PIW}, where
$\tilde{\chi}_{s,\ka_1}$ is a compactly supported and smooth function satisfying   $\tilde{\chi}_{s,\ka_1}=1$ on supp${\chi}_{s,\ka_1}$. Since the above notations (\ref{Noo1})-(\ref{Noo3}) initially  introduced  by Krause and Roos \cite{KR22,KR23} are convenient,  here we keep  these  unchanged; moreover, these unchanged notation can help   readers compare the details in this paper  with those in \cite{KR22,KR23}. 
Then  we have the following important  factorization
\beq\label{CC2}
\mathscr{L}_{s,\A,\ka_1}[m_\circ](\xi)=\mathscr{L}_{s,\A,\ka_1}[1](\xi)~\mathscr{L}^\#_{s,\kappa_1}[m_\circ](\xi).
\eeq
Moreover, for each $y\in\Z^n$, simple computations give
\beq\label{MI10}
\F^{-1}_{\Z^n}(\mathscr{L}_{s,\A,\ka_1}[m_\circ])(y)=
\sum_{\bb\in \frac{1}{q}\Z^n\cap [0,1)^n}S(\A,\bb) e(\bb\cdot y)\F^{-1}_{\R^n}(m_\circ \chi_{s,\kappa_1})(y),
\eeq
which will play an important role in proving our main results.
Let us   define
$$\Phi_{\Pi,N_\circ,\nu,\e_\circ}^*(\xi):=\Phi_{\Pi,N_\circ,\nu}(\xi)\ {\ind  {|\nu|\le 2^{-2dj_\circ}j_\circ^{\lfloor 1/\e_\circ \rfloor} }},$$
where $\Phi_{\Pi,N_\circ,\nu}$ is defined by  (\ref{con1}).
%with $N_\circ\sim 2^{j_\circ}$.
Let 
\beq\label{zhong}
\e_\circ(j_\circ):={\lfloor 1/\e_\circ \rfloor}\log_2 j_\circ.\footnote{We will also  frequently use this notation  with $j_\circ$ replaced by $j$ or $l$, and $\e_\circ$ replaced by   $\e_\circ'$,
$\e_\circ''$, $\bar{\e}_\circ$,  $\tilde{\e}_\circ$ and so on.}
\eeq
 For each pair  $(s,\ka)$ with $1\le s\le \e_\circ(j_\circ)$ and $0<\kappa<\e_\circ$ (say $\kappa=\e_\circ/8$ ), and for $x\in\Z^n$,
 we
 define
\beq\label{exp1}
L_{\Pi,N_\circ,\la(x),\e_\circ,\ka}^s(\xi):=\mathscr{L}_{s,\A,\ka}[\Phi_{\Pi,N_\circ,\mu(x),\e_\circ}^*](\xi),
\eeq
where $\mu(x)$ is given as
\beq\label{not1}
\mu(x)=\la(x)-\A,
\eeq
and
$\A$ is the unique element satisfying $\A\in \mathfrak{A}_s$ so that
$|\mu(x)|\le 2^{-10s}$ or an arbitrary element of the complement of $ \mathfrak{A}_s$ if no such $\A$ exists (this case will yield $(\ref{exp1})=0$). As a result, $\A$ may depend on the variable $x$, and we shall keep this fact in mind.
Moreover, these restrictions  $\la(x)\in [0,1]$ and $|\mu(x)|\le 2^{-10s}$
yield that,
$ \A\in  \mathfrak{A}_s$   shall be   $ \A\in \mathcal{A}_{s}:= \mathfrak{A}_s\cap [0,1]$ satisfying 
%but we always   use the notation
%$ \mathcal{A}_s$ to make the notation clearer and only require to    keep in mind that
\beq\label{new1}
\# \mathcal{A}_{s}\le 2^{2s}.
\eeq
This rough  bound will be used  in the proof of  Theorem \ref{co1}.
%This yields that  (\ref{exp1}) is supported in   $\Lambda_{j_\circ,\e_\circ,\la}$.
%Note that since $u$ depends on the variable $x$,  so does $\A$.
 %return  to the analysis of (\ref{aa12}), and
 
Decompose
%{\ind {x\in \Lambda_{j_\circ,\e_\circ,\la}}}
$m_{\Pi,N_\circ,\la(x)}(\xi)$  as
\beq\label{azq2}
{\ind {x\in \Lambda_{j_\circ,\e_\circ,\la}}}~~ m_{\Pi,N_\circ,\la(x)}(\xi)
=\sum_{1\le s\le \e_\circ(j_\circ)}L_{\Pi,N_\circ,\la(x),\e_\circ,\ka}^s(\xi)+  E_{\Pi,N_\circ,\la(x),\e_\circ,\ka}(\xi).
\eeq
Then, by following the proof of \cite[Proposition 3.2]{KR23} (Proposition \ref{p12} in the present paper and exponential sum estimates in Stein and Wainger \cite{SW99}  shall be used in this process), we can infer that for every $p\in(1,\infty)$,   there is   $\gamma_{1,p}>0$ such that for each $\e_\circ>0$,
\beq\label{azq40}
\|{\ind {x\in \Lambda_{j_\circ,\e_\circ,\la}}} \big(E_{\Pi,N_\circ,\la(x),\e_\circ,\ka}(D)f\big)(x)\|_{\ell^p(x\in\Z^n)}
%\les \|\sup_{u\in X_{j_\circ,\e_\circ}}|E_{j_\circ,u,\e_\circ}(D)f|\|_{\ell^2(x\in\Z^n)}
\les 2^{-\gamma_{1,p} j_\circ}\|f\|_{\ell^p(\Z^n)}.
\eeq
To streamline the main text, 
we postpone  its proof to Appendix \ref{appendixB}.
%(We sketch the proof of \eqref{azq40} in the Appendix.) 
This is the second minor arcs estimate.  
While the major arcs estimate remains the most challenging aspect in estimating numerous discrete operators through the Hardy-Littlewood circle method, the minor arcs estimate, which draws upon number theory techniques, holds significant importance.
 With  the above minor arcs estimates  %($\e_\circ=\e_\circ(p,\mathcal{C})$ is sufficiently small and is fixed  by (\ref{azq1})) 
 in hand,  to  estimate (\ref{aa12}), it suffices  to give the desired bound for the first term on the right hand-side of (\ref{azq2}), which is called
 the
 major arcs estimate in the following context.

 In what follows, we will use the above arguments multiple times. Particularly,
 we denote 
 %since the value of $\kappa$  is not important, we omit it from the above mentioned  notations except the functions $\chi_{s,\kappa}$ and $\tilde{\chi}_{s,\kappa}$, that is,   
$$
 \begin{aligned}
%\big(\mathscr{L}_{s,\A},~\mathscr{L}^\#_{s},~\mathcal{U}_{2^s}\big):=&\ \big(\mathscr{L}_{s,\A,\ka},~\mathscr{L}^\#_{s,\ka},~\mathcal{U}_{2^s,\ka}\big)\ \ \ \ \ \ \ \ {\rm and }\\
 \big(L_{\Pi,N_\circ,\la(x),\e_\circ}^s, ~E_{\Pi,N_\circ,\la(x),\e_\circ}\big):=&\ \big( L_{\Pi,N_\circ,\la(x),\e_\circ,\ka}^s,~E_{\Pi,N_\circ,\la(x),\e_\circ,\ka}\big)
 \end{aligned}
$$
 since $\kappa$ only depends on $\e_\circ$.
   In the followed two subsections, we  will show further  reductions of (\ref{long}) and (\ref{short8}). Keep two minor arcs  estimates (\ref{azq1}) and (\ref{azq40}) in mind.
 \subsection{Reduction of   (\ref{long}) and major arcs estimates I and II}
\label{ss2long}
%First, we decompose the operator $\mathscr{C}_{2^j}$ for each $j\in\N$.
Define
\beq\label{df1}
\N^{B}:=\N\cap [C_0,\infty)
\eeq
with $C_0$ sufficiently large.
For all $0\le   j\les1$ and  every $p\in(1,\infty)$, we have  $\|\mathscr{C}_{2^j}f\|_{\ell^p(\Z^n)}
\les \|f\|_{\ell^p(\Z^n)}$, which implies that 
\beq\label{AZZ21}
\|(\mathscr{C}_{2^j}f)_{j\in \N\setminus \N^{B}}\|_{\ell^p(\Z^n;\V^r)}
\les \|f\|_{\ell^p(\Z^n)}.
\eeq
%with the implicit constant independent of $f$.
By (\ref{simple}), (\ref{vardef}), (\ref{Ad1}) and (\ref{AZZ21}),
to show  (\ref{long}), it suffices to prove  that for each $(r,p)\in (2,\infty)\times [p_1,p_2]$,
\begin{align}
\|(\mathscr{C}_{2^j}f)_{j\in \N^{B}}\|_{\ell^p(\B_R;V^r)}&\les_\e R^\e \|f\|_{\ell^p(\Z^n)}
\ \ (R\ge 1).
\label{long900}
\end{align}
%where the set  $\N^{B}$
%is given by
%%%%\beq\label{df1}
%%%\N^{B}:=\N\cap [C_0,\infty)
%%\eeq
%with $C_0$ sufficiently large.
For each $l\in\Z$, we denote
   \beq\label{Noo12}
K_l:=K~\psi_l.
   \eeq
%Let $\B_{2^l}^*=\{x\in\Z:\ |x|\le 2^l\}$ for $l\in\Z$,
Using the partition of  unity $\sum_{l\in\Z}\psi_l=1$ and (\ref{Noo12}),
 we decompose the operator $\mathscr{C}_{2^j} $ as
$$
\mathscr{C}_{2^j}   f(x)=M_j  f(x)+T'_j f(x) \hskip.2in (j\in\N),
$$
where  $M_j$ and $T_j'$ are  defined by
$$
\begin{aligned}
M_j  f(x):=&\ \sum_{y\in \B_{2^j}}f(x-y)e\big(\la(x)|y|^{2d}\big) K_j(y)\ \ \ \ \ {\rm and}\\
T'_j f(x):=&\ \sum_{0\le l< j}\sum_{y\in \Z^n}f(x-y)e\big(\la(x)|y|^{2d}\big) K_l(y).
\end{aligned}
$$
Then we reduce the proof of
 (\ref{long900}) to demonstrating that for each $(r,p)\in (2,\infty)\times [p_1,p_2]$,
%that for each $(r,p)\in (2,\infty)\times (1,\infty)$,  and for all $R\ge 1$,
%\begin{align}
%\|(M_j  f)_{j\in \N^{B}}\|_{\ell^p(\Z^n;\V^r)}&\les\  \|f\|_{\ell^p(\Z^n)},\label{long222}\\
%\|(T'_j f)_{j\in \N^{B}}\|_{\ell^p(\B_R;\V^r)}&\les_\e R^\e \|f\|_{\ell^p(\Z^n)}.\label{long221}
%\end{align}
%where  $ \N^{B}$, a subset of $\N$, is given by
%\beq\label{df1}
%%\N^{B}:=\N\cap [C_0,\infty)
%\eeq
%with $C_0$ sufficiently large.
%By   (\ref{Ad1}),(\ref{vardef}) and (\ref{AZA1}), to show (\ref{long222}) and (\ref{long221}),
 \begin{align}
\|(M_j  f)_{j\in \N^{B}}\|_{\ell^p(\Z^n;V^r)}&\les\  \|f\|_{\ell^p(\Z^n)}\ \ \ \ {\rm and}\label{long2}\\
\|(T_j f)_{j\in \N^{B}}\|_{\ell^p(\B_R;V^r)}&\les_\e R^\e \|f\|_{\ell^p(\Z^n)},\label{long11}
\end{align}
where  $T_j f$ is given by
$$T_j f(x)=T'_jf(x)-T'_{C_0}f(x)=\sum_{C_0\le  l< j}\sum_{y\in \Z^n}f(x-y)e\big(\la(x)|y|^{2d}\big) K_l(y).$$
Here we have shifted our attention from bounding $T'_jf$ to estimating $T_j f$ by invoking  the definition of the  seminorm $V^r$. In the remainder of this subsection, the arguments in Subsection \ref{ss4.1} are used to further provide the reductions of
(\ref{long2}) and (\ref{long11}). %Maintain the notation $\e_\circ$ as is. 
Let $\e_\circ=\e_\circ(p_1,p_2,\mathcal{C})$ ($\mathcal{C}$ large enough) be the constant   given as in 
 Subsection \ref{ss4.1} (see (\ref{azq1}) above).

 \subsubsection{Reduction of  long variational inequality  (\ref{long2})}
\label{sslong1}
 Consider $M_j  f$.
By repeating the arguments presented in Subsection \ref{ss4.1}
%, we can derive two minor arcs estimates along with the primary component related to the major arcs for $M_jf$. More precisely,
    %using the arguments in Subsection  \ref{ss4.1} 
     with
$$j_\circ=j,\ N_\circ=2^{j},\ \Pi=2^{j-1},\  \KK={K}_j,$$
%and applying (\ref{azq1}), (\ref{azq2}) and (\ref{azq3}) with
%$$L^{s}_{j_\circ,u(x),M_\circ} =L^{(1),s}_{j,\la(x),M} ,\ E_{j_\circ,u(x),M_\circ} =E^{(1)}_{j,\la(x),M},\ \Phi=\Phi^{(1)},$$
(since ${K}_j(y){\ind {|y|\le 2^j}}=K_j(y){\ind {2^{j-1}\le |y|\le 2^j}}$),
and using the notations
$$
\begin{aligned}
&\ \big(m_{2^{j-1},2^j,\la(x)},~\Phi_{2^{j-1},2^j,\mu(x),\e_\circ}^*, ~L_{2^{j-1},2^j,\la(x),\e_\circ}^s, ~E_{2^{j-1},2^j,\la(x),\e_\circ}\big)\\
=:&\ \big(m^{(1)}_{j,\la(x)},~\phi^{(1),*}_{j,\mu(x),\e_\circ}, ~L^{(1),s}_{j,\la(x),\e_\circ}, ~E^{(1)}_{j,\la(x),\e_\circ}\big),
\end{aligned}
$$
%$$L =L^{(1)} ,\ E =E^{(1)},\ \Phi=\phi^{(1)},$$
we write $M_j  f$ as
$$M_j  f(x)=:\big(\ma_{j,\la(x)}(D)f\big)(x),$$
and 
obtain  that  for $\e_\circ=\e_\circ(p_1,p_2,\mathcal{C})$, %with $\mathcal{C}$ large enough,
%that for all $\e_\circ\in (0,\gamma')$ with   $\gamma'$ given by (\ref{azq3}),
\begin{align}
\|{\ind {x\notin\Lambda_{j,\e_\circ,\la}}}~ \big(\ma_{j,\la(x)}(D)f\big)(x)\|_{\ell^p(x\in\Z^n)}
\les&\  j^{-\mathcal{C}}\|f\|_{\ell^p(\Z^n)}\ \ \ \ \ (p_1\le p\le p_2),\label{0123}\\
{\ind {x\in \Lambda_{j,\e_\circ,\la}}}~ \ma_{j,\la(x)}(\xi)
=&\ \sum_{1\le s\le \e_\circ (j)} L^{(1),s}_{j,\la(x),\e_\circ} (\xi)+  E^{(1)}_{j,\la(x),\e_\circ}(\xi)\ {\rm and}\  \label{p87}\\
\|{\ind {x\in \Lambda_{j,\e_\circ,\la}}}~\big(E^{(1)}_{j,\la (x),\e_\circ}(D)f\big)(x)\|_{\ell^p(x\in\Z^n)}
%\les \|\|\sup_{\la \in X_{j,M}}|E_{j,\la,M}(D)f|\|_{{\ell^p}}
\les&\  2^{-\gamma_{1,p} j}\|f\|_{\ell^p(\Z^n)}\ \ \  (1<p<\infty),\ \label{00}
\end{align}
where
\begin{equation}\label{f1}
\begin{aligned}
L^{(1),s}_{j,\la(x),\e_\circ} (\xi):=&\ \mathscr{L}_{s,\A,\ka}[\phi_{j,\mu(x),\e_\circ}^{(1),*}](\xi)\ \hskip.6in \hskip.6in{\rm with}\\
\phi_{j,\mu(x),\e_\circ}^{(1),*}(\xi):=&\ \phi_{j,\mu(x)}^{(1)}(\xi)~{\ind {|\mu(x)|\le 2^{-2dj}j^{\lfloor 1/\e_\circ \rfloor}}}\  \hskip.6in {\rm and}\\
\phi_{j,\mu(x)}^{(1)}(\xi):=&\ \int_{2^{j-1}\le |y|\le 2^j} e\big(\mu(x) |y|^{2d}+y\cdot\xi\big)~{K}_j(y) dy.
\end{aligned}
\end{equation}
We will also write $L^{(1),s}_{j,\la(x),\e_\circ} (\xi)=L^{(1),s}_{j,\A+\mu(x),\e_\circ} (\xi)$.
Notice that
the major part is  the first term on the right-hand side of (\ref{p87}). Remember that   $\la(x)$ is an arbitrary  function from $ \Z^n$ to [0,1], and keep the notation (\ref{multi1}) in mind.
%Here and in what follows we replace the notation $\Phi$ by $\phi$.
By  a routine computation,   (\ref{0123}) and (\ref{00}),  we obtain that   for each $p\in [p_1,p_2]$,
\beq\label{minor1}
\begin{aligned}
\|\big({\ind {x\notin \Lambda_{j,\e_\circ,\la}}}  \big(\ma_{j,\la(x)}(D)f\big)(x)\big)_{j\in\N^B}\|_{\ell^p(x\in\Z^n; V^r)}\les&\  \sum_{j\in\N^B} j^{-\mathcal{C}} \|f\|_{\ell^p(\Z^n)}
\les \|f\|_{\ell^p(\Z^n)},\\
\|\big({\ind {x\in \Lambda_{j,\e_\circ,\la}}}  \big(E^{(1)}_{j,\la (x),\e_\circ}(D)f\big)(x) \big)_{j\in\N^B}\|_{\ell^p(x\in\Z^n; V^r)}\les&\  \sum_{j\in\N^B} 2^{-\gamma_{1,p} j}\|f\|_{\ell^p(\Z^n)}
\les \|f\|_{\ell^p(\Z^n)},
\end{aligned}
\eeq
where $\e_\circ=\e_\circ(p_1,p_2,\mathcal{C})$.
Consequently, in order to achieve (\ref{long2}), it suffices to show  
 the proposition below, which is deferred until Section \ref{slong2}.  
 \begin{prop}\label{t21}{\rm (Major arcs estimate I)}
%Let $L_{l,\la(x)}^{(1),s}$ be the operator   defined by (\ref{f2}).
%There is a positive constant $\e_1$ (smaller than $\gamma'$) such that  
 For each $r\in(2,\infty)$ and each $p\in [p_1,p_2]$, we have
$$
\big\|\big(\sum_{1\le s\le \e_\circ (j)}[L_{j,\la(x),\e_\circ}^{(1),s}(D)f](x)\big)_{j\in\N^B}\big\|_{\ell^p(x\in\Z^n; V^r)}
\les~\|f\|_{\ell^p(\Z^n)},
$$
where $\e_\circ=\e_\circ(p_1,p_2,\mathcal{C})$ and $\e_\circ (j)$ is defined by (\ref{zhong}) with $j_\circ=j$.
\end{prop}
%%%%
\subsubsection{Reduction of long variational inequality  (\ref{long11})}
\label{sslong}
Consider $T_j  f$.
By reiterating the arguments provided in Subsection \ref{ss4.1} with 
$$j_\circ=l,\ \ N_\circ=2^{l+1},\ \ \Pi=2^{l-1},\ \ {\KK}={K}_l,$$
(since ${K}_l(y)=K(y) \psi_l(y)=K(y) \psi_l(y){\ind {2^{l-1}\le |y|\le 2^{l+1}}}$), 
and applying the notations  
$$
\begin{aligned}
&\ \big(m_{2^{l-1},2^{l+1},\la(x)},\Phi_{2^{l-1},2^{l+1},\mu(x),\e_\circ}^*, ~L_{2^{l-1},2^{l+1},\la(x),\e_\circ}^s, ~E_{2^{l-1},2^{l+1},\la(x),\e_\circ}\big)\\
=:&\ \big(\mb_{l,\la(x)},\phi^{(2),*}_{l,\mu(x),\e_\circ}, ~L^{(2),s}_{l,\la(x),\e_\circ}, ~E^{(2)}_{l,\la(x),\e_\circ}\big),
\end{aligned}
$$
we can write $T_j f$ as
$$T_j  f(x)=:\sum_{C_0 \le l< j} \big(\mb_{l,\la(x)}(D)f\big)(x)$$
and get that for  $\e_\circ=\e_\circ(p_1,p_2,\mathcal{C})$,
\begin{align}
\|{\ind {x\notin \Lambda_{l,\e_\circ,\la}}} ~\big(\mb_{l,\la(x)}(D)f\big)(x)\|_{\ell^p(x\in\Z^n)}
%\le \|\sup_{\la \in X_{j,M}^c}|m_{2^{j-1},2^{j+3},\la}(D)f|\|_{\ell^p}
\les&\  l^{-\mathcal{C}}\|f\|_{\ell^p(\Z^n)}\ \ \ \ \ (p_1\le p\le p_2), \label{10}\\
{\ind {x\in \Lambda_{l,\e_\circ,\la}}} \mb_{l,\la(x)}(\xi)
=&\ \sum_{1\le s\le \e_\circ (l)}L_{l,\la(x),\e_\circ}^{(2),s}(\xi)+  E^{(2)}_{l,\la(x),\e_\circ}(\xi) \ {\rm and}\label{ds1}\\
\|{\ind {x\in \Lambda_{l,\e_\circ,\la}}}~ \big(E_{l,\la (x),\e_\circ}(D)f\big)(x)\|_{\ell^p(x\in\Z^n)}
%\les \|\|\sup_{\la \in X_{j,M}}|E_{j,\la,M}(D)f|\|_{{\ell^p}}
\les&\  2^{-\gamma_{1,p} l}\|f\|_{\ell^p(\Z^n)}\ \ (1<p<\infty),\label{100}
\end{align}
where
\begin{equation}\label{f2}
\begin{aligned}
L^{(2),s}_{l,\la(x),\e_\circ} (\xi):=&\ \mathscr{L}_{s,\A,\ka}[\phi_{l,\mu(x),\e_\circ}^{(2),*}](\xi)\  \ \ \ \ \ \ \ \ \ \ \ \ \ \ \ \ {\rm with}\\
\phi_{l,\mu(x),\e_\circ}^{(2),*}(\xi):=&\ \phi_{l,\mu(x)}^{(2)}(\xi)~~{\ind {|\mu(x)|\le 2^{-2dl}l^{\lfloor 1/\e_\circ \rfloor}}}\  \ \ \ {\rm and}\\
\phi_{l,\mu(x)}^{(2)}(\xi):=&\ \int_{\R^n} e\big(\mu(x) |y|^{2d}+y\cdot\xi\big)~{K}_l(y) dy.
\end{aligned}
\end{equation}
We may also  write $L^{(2),s}_{l,\la(x),\e_\circ} (\xi)=L^{(2),s}_{l,\A+\mu(x),\e_\circ} (\xi)$ in the following context.
Note that the main part  is  the first term on the right hand side of (\ref{ds1}).
%Here and in what follows we replace the notation $\Phi$ by $\phi$.
By a simple computation, we can  obtain from (\ref{10}) and (\ref{100}) that for each $p\in[p_1,p_2]$,
\beq\label{minor2}
\begin{aligned}
\|\Big(\sum_{C_0 \le l< j} {\ind {x\notin \Lambda_{l,\e_\circ,\la}}} ~\big(\mb_{l,\la(x)}(D)f\big)(x)\Big)_{j\in\N^B}\|_{\ell^p(x\in\Z^n; V^r)}\les&\  \sum_{l\in\N^B} l^{-\mathcal{C}}\|f\|_{\ell^p(\Z^n)}
\les \|f\|_{\ell^p(\Z^n)},\\
\|\Big(\sum_{C_0 \le l< j} {\ind {x\in \Lambda_{l,\e_\circ,\la}}} \big(E^{(2)}_{l,\la (x),\e_\circ}(D)f\big)(x) \Big)_{j\in\N^B}\|_{\ell^p(x\in\Z^n; V^r)}\les&\  \sum_{l\in\N^B} 2^{-\gamma_{1,p}l} \|f\|_{\ell^p(\Z^n)}
\les \|f\|_{\ell^p(\Z^n)},
\end{aligned}
\eeq
where $\e_\circ=\e_\circ(p_1,p_2,\mathcal{C})$.
Hence, once the proposition below %establishing the major arcs estimate 
is affirmed, we can derive (\ref{long11}) from (\ref{ds1}) and (\ref{minor2}) immediately.
\begin{prop}\label{892}{\rm (Major arcs estimate II)}
%Let $L_{l,\la(x)}^{(2),s}$ be the operator   defined by (\ref{f2}).
%Let $\la:\Z^n\to [0,1]$ be an arbitrary  function.
Let
 $(R,r,p)\in [1,\infty)\times (2,\infty)\times [p_1,p_2]$. For  any $\e>0$,  we have
$$
\big\|\Big(\sum_{C_0 \le l< j}  \sum_{1\le s\le \e_\circ (l)}[L_{l,\la(x),\e_\circ}^{(2),s}(D)f](x)\Big)_{j\in\N^B}\big\|_{\ell^p(x\in \B_R; V^r)}
\les_\e  R^\e\|f\|_{\ell^p(\Z^n)},
$$
where $\e_\circ=\e_\circ(p_1,p_2,\mathcal{C})$ and $\e_\circ (l)$ is defined by (\ref{zhong}) with $j_\circ=l$.
\end{prop}
The proof of Proposition \ref{892} is delayed until Section \ref{slong3}.

\subsection{Reduction of   (\ref{short8}) and major arcs estimate III}
%maximal operator
For all  $0 \le j< C_0$,
we  may deduce that for each $p\in(1,\infty)$,
$$\sup_{N\in [2^j,2^{j+2}]}\|\mathscr{C}_{N} f-\mathscr{C}_{2^j} f\|_{\ell^p(\Z^n)}
\les \|f\|_{\ell^p(\Z^n)}\ \ {\rm and} \ \sup_{N\in [2^j,2^{j+2}]}\|\mathscr{C}_{N+1} f-\mathscr{C}_{N} f\|_{\ell^p(\Z^n)}
\les \|f\|_{\ell^p(\Z^n)},$$
which with (\ref{k01}) yields that  the estimate 
\beq\label{hk1}
\|(\mathscr{C}_{N} f-\mathscr{C}_{2^j} f)_{N\in [2^j,2^{j+2}]}\|_{\ell^p(\Z^n;V^2)}\les \|f\|_{\ell^p(\Z^n)}\ \ (1<p<\infty)
\eeq
holds for all  $0 \le j< C_0$. Then, it follows from (\ref{hk1}) that
%for every $p\in (1,\infty)$,
$$
 \|\big( \sum_{j\in \N\setminus\N^B}\|(\mathscr{C}_{N} f-\mathscr{C}_{2^j} f)_{N\in [2^j,2^{j+1})}\|_{V^2}^2\big)^{1/2}\|_{\ell^p(\Z^n)}\les \|f\|_{\ell^p(\Z^n)}.
$$
As a result, we reduce the proof of
 (\ref{short8}) to showing that  for each $p\in[p_1,p_2]$,
\beq\label{shorttt}
 \|\big( \sum_{j\in \N^B}\|(\mathscr{C}_{N} f-\mathscr{C}_{2^j} f)_{N\in [2^j,2^{j+1})}\|_{V^2}^2\big)^{1/2}\|_{\ell^p(\Z^n)}\les \|f\|_{\ell^p(\Z^n)}.
\eeq

We next use the arguments presented in Subsection \ref{ss4.1} to 
 $\mathscr{C}_{N} f-\mathscr{C}_{2^j} f$.   Let $\e_\circ=\e_\circ(p_1,p_2,\mathcal{C})$  be given as in  Subsection \ref{ss4.1}  (see (\ref{azq1}) above).  Precisely, 
as in the previous subsections,
by revisiting the arguments presented in Subsection \ref{ss4.1} with 
$$j_\circ=j,\ \ N_\circ=N\in [2^j,2^{j+1}),\ \ \Pi=2^{j},\ \ {\KK}(y)=K(y),$$
and using the notations
$$
\big(m_{2^j,N,\la(x)},\Phi_{2^j,N,\mu(x),\e_\circ}^*, ~L_{2^j,N,\la(x),\e_\circ}^s, ~E_{2^j,N,\la(x),\e_\circ}\big)
=:\big(\mc_{2^j,N,\la(x)},\phi^{(3),*}_{2^j,N,\mu(x),\e_\circ}, ~L^{(3),s}_{2^j,N,\la(x),\e_\circ}, ~E^{(3)}_{2^j,N,\la(x),\e_\circ}\big),
$$
we can write  the operator
$\mathscr{C}_{N} -\mathscr{C}_{2^j}$ as
$$
\mathscr{C}_{N} f(x)-\mathscr{C}_{2^j} f(x)
=:\big(\mc_{2^j,N,\la(x)}(D)f\big)(x),
%\sum_{y\in \B_N\setminus \B_{2^j}}f(x-y)e(\la(x)|y|^{2d}) K(y).
$$
and obtain   that for $\e_\circ=\e_\circ(p_1,p_2,\mathcal{C})$,
\begin{align}
\|{\ind {x\notin\Lambda_{j,\e_\circ,\la}}} ~\big(\mc_{2^{j},N,\la(x)}(D)f\big)(x)\|_{\ell^p(x\in\Z^n)}
%\le \|\sup_{\la \in X_{j,M}^c}|m_{2^{j-1},2^{j+3},\la}(D)f|\|_{\ell^p}
\les&\  j^{-\mathcal{C}}\|f\|_{\ell^p(\Z^n)}\ \ \ \ \ \ (p_1\le p\le p_2),\label{0}\\
{\ind {x\in \Lambda_{j,\e_\circ,\la}}} ~\mc_{2^{j},N,\la(x)}(\xi)
=&\ \sum_{1\le s\le \e_\circ (j)}L_{2^j,N,\la(x),\e_\circ}^{(3),s}(\xi)+  E^{(3)}_{2^j,N,\la(x),\e_\circ}(\xi)\ {\rm and}\label{cx3}\\
\|{\ind {x\in\Lambda_{j,\e_\circ,\la}}} ~\big(E^{(3)}_{2^j,N,\la (x),\e_\circ}(D)f\big)(x)\|_{\ell^p(x\in\Z^n)}
%\les \|\|\sup_{\la \in X_{j,M}}|E_{j,\la,M}(D)f|\|_{{\ell^p}}
\les&\  2^{-\gamma_{1,p} j}\|f\|_{\ell^p(\Z^n)}\ \ (1<p<\infty),\label{009}
\end{align}
where
\begin{equation}\label{f3}
\begin{aligned}
L_{2^j,N,\la(x),\e_\circ}^{(3),s} (\xi):=&\ \mathscr{L}_{s,\A,\ka}[\phi_{2^j,N,\mu(x),\e_\circ}^{(3),*}](\xi)\  \ \ \ \ \ \ \ \ \ \ \ \ \ {\rm with}\\
\phi_{2^j,N,\mu(x),\e_\circ}^{(3),*}(\xi):=&\ \phi_{2^j,N,\mu(x)}^{(3)}(\xi)~{\ind {|\mu(x)|\le 2^{-2dj}j^{\lfloor 1/\e_\circ \rfloor}}}\  \ \ \ \ \  {\rm and}\\
\phi_{2^j,N,\mu(x)}^{(3)}(\xi):=&\ \int_{2^j\le |y|\le N} e\big(\mu(x) |y|^{2d}+y\cdot\xi\big)~K(y) dy.
\end{aligned}
\end{equation}
Note that the primary component related to the major arcs is the first term on the right-hand side of (\ref{cx3}). In order to establish (\ref{shorttt}),  routine  calculations indicate that it suffices to demonstrate  the following propositions:
 \begin{prop}\label{addpp1}
Suppose that $j\in\N^B$, $p\in[p_1,p_2]$ and $\e_\circ=\e_\circ(p_1,p_2,\mathcal{C})$. There exists  
large enough $C_{p_1}>0$ such that for all $\mathcal{C}\ge C_{p_1} $,
\begin{align}
\|\Big({\ind {x\notin\Lambda_{j,\e_\circ,\la}}}~\big(\mc_{2^j,N,\la(x)}(D)f\big)(x)\Big)_{N\in [2^j,2^{j+1})}\|_{\ell^p(x\in\Z^n;V^2)}\les&\  j^{-2} \|f\|_{\ell^p(\Z^n)}\ \ \  {\rm and}\label{i0}\\
\|\Big({\ind {x\in\Lambda_{j,\e_\circ,\la}}}~\big(E^{(3)}_{2^j,N,\la(x),\e_\circ}(D)f\big)(x)\Big)_{N\in [2^j,2^{j+1})}\|_{\ell^p(x\in\Z^n;V^2)}\les&\  j^{-2}\|f\|_{\ell^p(\Z^n)}\label{i1}.
\end{align}
%hold for all $\e_\circ\in (0,\gamma')$ with   $\gamma'$ given by (\ref{azq3}).
\end{prop}
%Meanwhile, as the previous subsections, (\ref{zhong5}) is a direct consequence of the following proposition.
\begin{prop}\label{t31}{\rm (Major arcs estimate III)}
For each $(r,p)\in(2,\infty)\times (1,\infty)$ and every $\tilde{\e}_\circ\in (0,1)$, 
%There is a positive constant $\e_2$ such that for any $\e_\circ\in (0,\e_2)$, we have
$$
\Big\|\Big( \sum_{j\in \N^B}\big\|\big(\sum_{1\le s \le \tilde{\e}_\circ (j)}[L_{2^j, N,\la(x),\tilde{\e}_\circ}^{(3),s}(D)f](x)\big)_{N\in [2^j,2^{j+1})}\big\|_{V^2}^2\Big)^{1/2}\Big\|_{\ell^p(\Z^n)}\les \|f\|_{\ell^p(\Z^n)}.
$$
\end{prop}
The rest of this subsection is dedicated to proving Proposition \ref{addpp1}, while the proof of Proposition \ref{t31} is deferred to Section \ref{slong5}.
\begin{proof}[Proof of Proposition \ref{addpp1}]
Since the value of $\e_\circ$  is not important for estimating  $E^{(3)}_{2^j,N,\la(x),\e_\circ}$, we will omit it from this notation, that is, 
$$E^{(3)}_{2^j,N,\la(x)}:=E^{(3)}_{2^j,N,\la(x),\e_\circ}.$$
We shall use the numerical inequality (\ref{k01}) to achieve the goal.
Denote  $(Y_{j,\e_\circ,\la}^{(1)},\m^{(1)}):=(\Z^n\setminus\Lambda_{j,\e_\circ,\la},\mc)$, $(Y_{j,\e_\circ,\la}^{(2)},\m^{(2)}):=(\Lambda_{j,\e_\circ,\la},E^{(3)})$,    and define
$$
\begin{aligned}
U_{p,\e_\circ}(i,j):=&\ \sup_{2^j\le N\le 2^{j+1}}\|{\ind {x\in Y_{j,\e_\circ,\la}^{(i)}}}\big(\m^{(i)}_{2^j,N,\la(x)}(D)f\big)(x)\|_{\ell^p(x\in\Z^n)},\\
 V_{p,\e_\circ}(i,j):=&\ \sup_{2^j\le N< 2^{j+1}}\|{\ind {x\in Y_{j,\e_\circ,\la}^{(i)}}}\big(\m^{(i)}_{N,N+1,\la(x)}(D)f\big)(x)\|_{\ell^p(x\in\Z^n)},\ \ i=1,2,
 \end{aligned}
%\les j^{-2}\|f\|_{\ell^p}
$$
where  $\m^{(i)}_{N,N+1,\la(x)}$ is given by
$\m^{(i)}_{N,N+1,\la(x)}:=\m^{(i)}_{2^j,N+1,\la(x)}-\m^{(i)}_{2^j,N,\la(x)}.$
% Here $\m^{(i)}_{N,N+1,\la(x)}$ is defined as $\m^{(i)}_{2^j,N,\la(x)}$ with $(2^j,N)$ replaced by $(N,N+1)$.
Let $r_p=\min\{2,p\}$. To prove (\ref{i0}) and (\ref{i1}),   it suffices to show that for each $j\in\N^B$ and each $p\in[p_1,p_2]$,
\beq\label{ep1}
\|\Big({\ind {x\in Y_{j,\e_\circ,\la}^{(i)}}}\big(\m^{(i)}_{2^j,N,\la(x)}(D)f\big)(x)\Big)_{N\in [2^j,2^{j+1})}\|_{\ell^p(x\in\Z^n;V^{r_p})}\les j^{-2}\|f\|_{\ell^p(\Z^n)},\ \ i=1,2.
\eeq
%where $r_p=\min\{2,p\}$.
Invoking (\ref{k01}),   we can bound the left-hand side of (\ref{ep1}) by a constant times
\beq\label{pp1}
\begin{aligned}
U_{p,\e_\circ}(i,j)+2^{j/r_p}U_{p,\e_\circ}(i,j)^{1-1/r_p}V_{p,\e_\circ}(i,j)^{1/r_p}.
\end{aligned}
\eeq
Thus we reduce the matter  to proving
\beq\label{pp0}
\begin{aligned}
(\ref{pp1})\les j^{-2}\|f\|_{\ell^p(\Z^n)},\ \ i=1,2.
\end{aligned}
\eeq
Using (\ref{0}) and (\ref{009}), we first  have
\beq\label{vv1}
U_{p,\e_\circ}(i,j)\les (2^{-\gamma_{1,p} j}+j^{-\mathcal{C}})\|f\|_{\ell^p(\Z^n)},\ \ i=1,2.
\eeq
Notice  $ \|{\ind {N\le |y|\le N+1}} K(y) \|_{\ell^1_y(\Z^n)}\les 2^{-j}$ whenever  $N\in[2^j,2^{j+1})$. Then, by Young's convolution inequality, we deduce
%$\Phi_{N,N+1}$ is supported in $\{|y|\sim 2^j\}$ and   $|\Phi_{N,N+1}(y)|\les 2^{-j}$ whenever  $N\in[2^j,2^{j+1}]$.  Then
\beq\label{cx1}
\|\big(\m^{(1)}_{N,N+1,\la(x)}(D)f\big)(x)\|_{\ell^p(x\in\Z^n)}
%\les
%\|\sup_{\la\in [0,1]}|\mc_{N,N+1,\la}(D)f|\|_{\ell^p}\\
\les
 \|{\ind {N\le |y|\le N+1}} K(y) \|_{\ell^1_y(\Z^n)}\|f\|_{\ell^p(\Z^n)}
\les 2^{-j}\|f\|_{\ell^p(\Z^n)},
\eeq
which implies    $V_{p,\e_\circ}(1,j)\les 2^{-j}\|f\|_{\ell^p(\Z^n)}$.
 This  estimate with (\ref{vv1}) yields (\ref{pp0}) for the case ${i=1}$ by setting $ C_{p_1}$    large enough such that $C_{p_1}(1-1/p_1)\ge 10$.
 
Next,  we  consider  (\ref{pp0}) for the case  $i=2$.
Since    (\ref{vv1}) (with $i=2$) holds and $C_{p_1}(1-1/p_1)\ge 10$,  it  suffices to show
$V_{p,\e_\circ}(2,j)\les \e_\circ (j)~2^{-j} \|f\|_{\ell^p(\Z^n)}$. Using (\ref{cx1}) and (\ref{cx3}), we may reduce the matter  to   proving  that for all $1\le s\le \e_\circ (j)$ and  all   $N\in [2^j,2^{j+1})$,
\beq\label{B2}
\|\big(L_{N,N+1,\la(x),\e_\circ}^{(3),s}(D)f\big)(x)\|_{\ell^p(x\in\Z^n)}\les 2^{-j}\|f\|_{\ell^p(\Z^n)},
\eeq
where  $L_{N,N+1,\la(x),\e_\circ}^{(3),s}$ is given by
$L_{N,N+1,\la(x),\e_\circ}^{(3),s}=L_{2^j,N+1,\la(x),\e_\circ}^{(3),s}-L_{2^j,N,\la(x),\e_\circ}^{(3),s}.$
Using an equality like   (\ref{zhu2}) below and
$
\sup_{z\in\R^n}\|\F^{-1}_{\R^n} (\chi_{s,\ka})(\cdot-z)\|_{\ell^1(\Z^n)}\les 1,
$
we have
\beq\label{B3}
\begin{aligned}
\|L_{N,N+1,\la(x),\e_\circ}^{(3),s}(D)f\|_{\ell^p(\Z^n)}\les&\  \||\F^{-1}_{\R^n} (\chi_{s,\ka})|*_{\R^n} |{\ind {N\le |\cdot|\le N+1}}  K(\cdot)|\|_{\ell^1(\Z^n)}\|f\|_{\ell^p(\Z^n)}\\
\les&\ \int_{\R^n} \|\F^{-1}_{\R^n} (\chi_{s,\ka})(\cdot-y)\|_{\ell^1(\Z^n)}|{\ind {N\le |y|\le N+1}}  K(y)|dy~ \|f\|_{\ell^p(\Z^n)}\\
\les&\ N^{-1}\|f\|_{\ell^p(\Z^n)},
\end{aligned}
\eeq
which yields (\ref{B2}) since $N\in[2^j,2^{j+1})$.  This completes the proof of Proposition \ref{addpp1}.
 \end{proof}
Hence,  to finish  the proof of  the short variational estimate (\ref{short8}), it remains to prove the above  Proposition \ref{t31}.
%Accepting Theorem \ref{t31} and this claim, we now prove (\ref{short}). In fact,  the above claim  yields
%the desired $\ell^p$ short variation-norm  estimates for
%$${\ind {\la(x)\in X_{j,M}^c}}\big(m_{2^j,N,\la(x)}(D)f\big)(x),\ {\rm and}\  {\ind {\la(x)\in X_{j,M}}}\big(E_{2^j,N,\la(x)}(D)f\big)(x),$$
 %which with Theorem \ref{t31} yields  (\ref{short}) immediately.

\section{Crucial auxiliary results for proving   major arcs estimates}
\label{Fpre}
In this section, we  gather some significant results obtained in   \cite{KR22, KR23},   establish  a novel multi-frequency square function estimate and ultimately verify the key  multi-frequency variational inequalities.
Keep notations (\ref{Noo1}), (\ref{Noo2}) and (\ref{Noo3})  in mind.   This section is to give the  crucial estimates with respect to   $\mathscr{L}_{s,\A,\ka}$ and $\mathscr{L}_{s,\ka}^\#$.
Since the value of $\ka$
is not important for obtaining these estimates, we use the notation 
\beq\label{not32}
(\mathscr{L}_{s,\A},\ \mathscr{L}_{s}^\#,\ \mathcal{U}_{2^s}):=(\mathscr{L}_{s,\A,\ka},\ \mathscr{L}_{s,\ka}^\#,\ \mathcal{U}_{2^s,\kappa}).
\eeq
Then    (\ref{Noo3})-(\ref{MI10}) give
\begin{align}
\mathscr{L}_{s,\A}[m](\xi)=&\ \mathscr{L}_{s,\A}[1](\xi)~\mathscr{L}^\#_{s}[m](\xi), \label{CC1}\\
\mathscr{L}^\#_{s}[m](\xi)=&\ \Delta_{\mathcal{U}_{2^s}}[m \tilde{\chi}_{s,\ka}](\xi) \ \ \ \ \ \hskip.3in {\rm and }\label{MO1}\\
\mathscr{L}_{s,\A}[m](D)f(x)=&\
\sum_{\bb\in \frac{1}{q}\Z^n\cap [0,1)^n}S(\A,\bb) e(\bb\cdot x)\big\{\F^{-1}_{\R^n}(m  \chi_{s,\ka})*\NE_{-\bb}f\big\} (x),\label{MI11}
\end{align}
where  $m$ is a bounded function on $\R^n$, 
and $
\mathfrak{N}_uf(y):=e(u\cdot y)f(y)$ denotes modulation by $u$.
%These equalities will be frequently used in the proof of our main result.
\subsection{Maximal estimates  by Krause and Roos}
\label{EKR1}
Let $s\in\N$.
For each $y\in\Z^n$,
\beq\label{zhu2}
\F_{\Z^n}^{-1}\big(\mathscr{L}_{s,\A}[m]\big)(y)
=e(\A|y|^{2d})\F_{\R^n}^{-1}\big(m \chi_{s,\kappa}\big)(y),
\eeq
we refer  Lemma 4.1 in \cite{KR23} for the details. Below we  state two  important maximal estimates proved by Krause and Roos  \cite{KR22,KR23}. Keep  the notations
(\ref{multi1}), (\ref{multi2}), (\ref{Noo1}) and  (\ref{new1}) in mind.
\begin{lemma}\label{de1}
(i) Let $s\in\N$. For every $p\in (1,\infty)$, there is a constant   $\gamma_p\in (0,1)$ such that
\beq\label{de11}
\| \sup_{\A\in \mathcal{A}_{s}}|\mathscr{L}_{s,\A}[1](D)f|\|_{\ell^p(\Z^n)}
\les 2^{-\gamma_p s}\|f\|_{\ell^p(\Z^n)}.
\eeq
(ii) Let $s\in\N$. Let $\theta$ denote a smooth and nonnegative  function
on $\R^n$ with compactly support and $\int \theta=1$, and let $\theta_l(y)=
 2^{-ln}\theta(2^{-l}y)$ with $l\in\Z$. Then  for every $p\in (1,\infty)$, we have
 \beq\label{yiny}
 \|\sup_{j\ge 1}\sup_{\A\in\mathcal{A}_{s}}|\mathscr{L}_{s,\A}[\widehat{\theta_j}](D)f|\|_{\ell^p( \Z^n)}\les   2^{-\gamma_ps} \|f\|_{\ell^p( \Z^n)}
 \eeq
 with $\gamma_p$ given as in (\ref{de11}).
\end{lemma}
We refer the arguments  in \cite[Section 6]{KR22} for the details of the proof of (\ref{de11}). As for  (\ref{yiny}),
 it emerges not as a theorem or lemma but rather within the course of the proof. Precisely,
 it follows by expanding $\theta_j$ as a  telescoping sum $\theta_0+\sum_{l=1}^{j-1} (\theta_{l+1}-\theta_{l})$ and applying   Lemma 4.4 in \cite{KR23}.
%Obviously, since $\mathscr{L}_{s,\A}[1]=\mathscr{L}_{s,\A+1}[1]$,
%the above restriction $\A\in \mathcal{A}_{s}$  can also be replaced by $\A\in \mathfrak{A}_{s}$.
%The constant  $c_p$ will vary at each appearance.
\subsection{Multi-frequency square function estimate}
\label{SFE1}
Below we provide a new and practical multi-frequency  square function estimate, which plays a crucial role in proving our main results, and will be frequently used.
\begin{lemma}\label{ccz}
Let  $s\in\N$, $A>0$ and $B>0$. Let $\{\mathfrak{M}_j\}_{j\in\Z}$ be a sequence of  bounded  functions satisfying
\beq\label{X2}
|\mathfrak{M}_j(\xi)|\les A\min\big\{|2^j\xi|^\gamma,|2^j\xi|^{-\gamma}\big\}\ \ \ (\xi\in\R^n)
\eeq
for some $\gamma>0$. Suppose that for every $p\in(1,\infty)$,   we have   the vector-valued inequality
\beq\label{X3}
\|\big(\sum_{j\in\Z}|\mathfrak{M}_j(D)f_j|^2\big)^{1/2}\|_{L^p(\R^n)}\les B \|\big(\sum_{j\in\Z}|f_j|^2\big)^{1/2}\|_{L^p(\R^n)}.
\eeq
Then for each $p\in(1,\infty)$, there is \footnote{The constants $c,c_p,C$ may vary at each appearance.}   a  constant  $c_p\in (0,1)$  such that
\beq\label{cou1}
\|\big( \sum_{j\in\Z}\sup_{\A\in \mathcal{A}_s}|\mathscr{L}_{s,\A}[\mathfrak{M}_j](D)f|^2\big)^{1/2}\|_{\ell^p(\Z^n)}
\les A^{c_p}  B^{1-c_p} 2^{-\gamma_p s}\|f\|_{\ell^p(\Z^n)},
\eeq
%with $c_p=2\min\{1/p,1/p'\}$.
with $\gamma_p$  given as in  (\ref{de11}).
\end{lemma}
\begin{proof}
We denote by  $\{\va_i(t)\}_{i=0}^\infty$   the sequence of Rademacher functions (see e.g., \cite{Gra14}) on $[0,1]$  satisfying
\beq\label{xinqing}
\|\sum_{i=0}^\infty z_i \va_i(t)\|_{L^q_t([0,1])}\sim \big(\sum_{i=0}^\infty |z_i|^2\big)^{1/2}.
\eeq
Claim that for each $p\in (1,\infty)$,  there exists a constant  $c_p\in (0,1)$ such that
\beq\label{chen1}
\|\sum_{j\in\Z}\va_j(t)\mathfrak{M}_j(D)f\|_{L^p(\R^n)}\les A^{c_p}  B^{1-c_p}  \|f\|_{L^p(\R^n)}.
\eeq
 By accepting this claim and utilizing Proposition \ref{PIW} along with the notation in (\ref{MO1}), we can infer
$$
\|\mathscr{L}_{s}^\#[\sum_{j\in\Z}\va_j(t)\mathfrak{M}_j](D)f\|_{\ell^p(\Z^n)}\les A^{c_p}  B^{1-c_p}  \|f\|_{\ell^p(\Z^n)}.
$$
which with (\ref{CC1}) and (\ref{de11}) gives   that for all $t\in [0,1]$
\beq\label{cou21}
\|\sup_{\A\in \mathcal{A}_s}|\mathscr{L}_{s,\A}[\sum_{j\in\Z}\va_j(t)\mathfrak{M}_j](D)f|\|_{\ell^p(\Z^n)}
\les A^{c_p}  B^{1-c_p} 2^{-\gamma_p s} \|f\|_{\ell^p(\Z^n)}.
\eeq
By  employing linearization  and  (\ref{xinqing}), the desired (\ref{cou1}) directly follows from (\ref{cou21}).

Next, we shall prove the claim (\ref{chen1}).  For  $j\in\Z$, the Littlewood-Paley decomposition $\sum_{v\in \Z}P_{v-j}f=f$ will be used.  From (\ref{X2}) and Plancherel's identity we have
\beq\label{X4}
\begin{aligned}
\|\sum_{j\in\Z}\va_j(t)\mathfrak{M}_j(D)P_{v-j}f\|_{L^2(\R^n)}\les&\
\|\sum_{j\in\Z}\va_j(t)\mathfrak{M}_j(\xi)\psi(2^{j-v}\xi)\|_{L^\infty_\xi}\|f\|_{L^2(\R^n)}\\
\les&\  A 2^{-\gamma|v|}\|f\|_{L^2(\R^n)}.
\end{aligned}
\eeq
On the other hand,  by (\ref{X3}) and the Littlewood-Paley theory, we obtain that for each $p\in(1,\infty)$,
\beq\label{X5}
\begin{aligned}
\|\sum_{j\in\Z}\va_j(t)\mathfrak{M}_j(D)P_{v-j}f\|_{L^p(\R^n)}\les&\
\|\big(\sum_{j\in\Z}|\mathfrak{M}_j(D)P_{v-j}f|^2\big)^{1/2}\|_{L^p(\R^n)}\\
\les&\  B \|\big(\sum_{j\in\Z}|P_{v-j}f|^2\big)^{1/2}\|_{L^p(\R^n)}
\les B\|f\|_{L^p(\R^n)}.
\end{aligned}
\eeq
Interpolating (\ref{X4}) with (\ref{X5}) gives that there is a constant  $c_p'\in (0,1)$ such that
$$\|\sum_{j\in\Z}\va_j(t)\mathfrak{M}_j(D)P_{v-j}f\|_{L^p(\R^n)}\les A^{c_p'}B^{1-c_p'}2^{-\gamma c_p'|v|}\|f\|_{L^p(\R^n)},$$
which with the Littlewood-Paley decomposition $\sum_{v\in \Z}P_{v-j}f=f$ and the triangle inequality yields the above claim (\ref{chen1}) (with $c_p=c_p'$). This ends the proof of Lemma \ref{ccz}.
\end{proof}
\subsection{Multi-frequency  variational inequalities}
\label{vi1}
In this subsection, we derive two crucial multi-frequency variational inequalities, which  play the key role in proving Theorem \ref{t1}. Their proofs are based on various techniques such as the classical variational inequality, the Ionescu-Wainger-type multiplier theorem, a transference principle by Mirek-Stein-Trojan, and a Rademacher-Menshov-type inequality.

Let $s\in\N$, and  let $Q_s$ denote the least common multiple of all integers in the range $[1,2^s]$.
  Let  $C_1$  be a  large integer    such that
  \beq\label{L1}
  2^{2^{C_1s}}\ge (2^{2^{2s}}Q_s)^{100n}.
  \eeq
 Below we provide  a variational inequality which is used to prove  Lemma \ref{endle} below.
\begin{lemma}\label{ccz2}
Let $s\in\N$, and let $\A(x)$ denote an arbitrary function from $\Z^n$ to  $\mathcal{A}_s$.
% and let $C_1$ be the constant given by (\ref{L1}).
   Let $\mathscr{V}$ be a smooth function  on $\R^n$ with supp$\mathscr{V}\subset\{\xi\in\R^n:\ |\xi|\le 2^{-2(n+4)}\}$  and $\mathscr{V}(0)=1$,
  let  $\mathscr{V}_j(\cdot)=\mathscr{V}(2^{j}\cdot)$ for $j\in\N$. Suppose that
  $\mathscr{B}$ is a bounded function on $\R^n$ satisfying  $\|\mathscr{B}(D)f\|_{L^p(\R^n)}\les \|f\|_{L^p(\R^n)}$ for each $p\in(1,\infty)$. Then  for every $(r,p)\in (2,\infty)\times (1,\infty)$ and  each $R\ge 1$, we have
  \beq\label{dou12}
\| \big(\mathscr{L}_{s,\A(x)}[\mathscr{V}_j~ \mathscr{B}](D)f(x)\big)_{j> 2^{C_1s}}\|_{\ell^p(x\in\B_R;V^r)}
\les_\e R^\e 2^{-\gamma_p s}\|f\|_{\ell^p(\Z^n)}
\eeq
with $\gamma_p$ given as in (\ref{de11}).
\end{lemma}
\begin{remark}\label{rr4.1}
As we will observe in the proof of (\ref{Goo2}) below,  the $R^\e$-loss on the right-hand side of (\ref{dou12}) can be refined to a logarithmic loss in terms of the scale $R$ (say $\ln \langle R \rangle$).
Likewise,  such an improvement is also applicable to (\ref{dou1}) in Lemma \ref{endle} below. However, for the sake of clarity in the exposition, we will not explore this direction further.
\end{remark}
The choice of supp$\mathscr{V}$  is based on the arguments in Subsection \ref{TMST2} as we rely on Proposition \ref{PMST}.
\begin{proof}
We  split the goal into two cases:  $2^{2^s}> R$ and $2^{2^s}\le  R$,
and claim  that for each $p\in (1,\infty)$,
\begin{align}
\| \big(\mathscr{L}_{s,\A(x)}[\mathscr{V}_j~ \mathscr{B}](D)f(x) \big)_{j> 2^{C_1s}}\|_{\ell^p(x\in\B_R;V^r)}
&\les\   2^{-\gamma_p s}\|f\|_{\ell^p(\Z^n)}\ \  \ \ \  \ {\rm if}\  2^{2^s}> R\ \ {\rm and}\label{Goo1}\\
\| \big(\mathscr{L}_{s,\A(x)}[\mathscr{V}_j~ \mathscr{B}](D)f (x)\big)_{j> 2^{C_1s}}\|_{\ell^p(x\in\B_R;V^r)}
&\les_\e R^\e  2^{-\gamma_p s}\|f\|_{\ell^p(\Z^n)}\ \  {\rm if}\  2^{2^s}\le R.\label{Goo2}
\end{align}
%where $\gamma_p$ is  given as in (\ref{de11}).
Accepting this claim, we obtain (\ref{dou12}) immediately. Thus, it remains to prove    (\ref{Goo1}) and (\ref{Goo2}).

We first consider (\ref{Goo1}).
%\underline{Proof of (\ref{Goo1})}\
Since  $2^{2^s}> R$,  we have
\beq\label{AABB1}
2^{2^{C_1s-1}}>2^{2^s} V_{s,R}\ \ \ {\rm with}\ \ \ V_{s,R}:=(R~Q_s)^{10n}.
  \eeq
  Denote by $h$ the Fourier inverse transform on $\R^n$ of $\VV$,   and let
  $$h_j(y):=2^{-jn}h(2^{-j}y)=2^{-jn}\check{\VV}(2^{-j}y)\ \hskip.2in (y\in\R^n).$$
  %with $(u,y)\in \R^n\times \R^n$.
 %Let $\mathfrak{N}_ug(x):=e(u\cdot x)g(x)$.
 % Setting  $V_{s,R}:=(RQ_s)^{10n}$,
  We  first prove that  for any $u\in [V_{s,R}]^n$,
\beq\label{c93}
\|\big((\mathscr{L}_{s,\A(x)} [(\mathscr{V}_j-\NE_u\VV_j)~ \mathscr{B}](D)f)(x)\big)_{j>2^{C_1s}}\|_{\ell^p(x\in \Z^n;V^1)}\les 2^{-s}\|f\|_{\ell^p}.
\eeq
Since $j>2^{C_1s}$ and $u\in [V_{s,R}]^n$, we infer from (\ref{AABB1}) that $2^{-j}u\le 2^{-j} V_{s,R}\le 2^{-2^s}2^{-j/2}$, which with (\ref{MI11}) and (\ref{new1}) yields that   the left-hand side of (\ref{c93}) is bounded by
  \beq\label{oom}
  \begin{aligned}
&\  \|\big(\sum_{\A\in \mathcal{A}_s}|\mathscr{L}_{s,\A} [(\mathscr{V}_j-\NE_u\VV_j)~ \mathscr{B}](D)f|^p\big)^{1/p}\|_{\ell^p(\Z^n)}\\
  \les&\  2^{Cs} \sup_{\bb\in [0,1)^n}\|
 (h_j-h_j(\cdot-u)) *\bar{B}_s*(\NE_{-\bb}f)\|_{\ell^p(\Z^n)}\\
 \les&\ 2^{Cs} 2^{-j}u \sup_{\bb\in [0,1)^n}\|M_{DHL}\big(\bar{B}_s*(\NE_{-\bb}f)\big)\|_{\ell^p(\Z^n)}\\
 \les&\ 2^{Cs-2^s}2^{-j/2} \sup_{\bb\in [0,1)^n}\|M_{DHL}\big(\bar{B}_s*(\NE_{-\bb}f)\big)\|_{\ell^p(\Z^n)}
 \end{aligned}
  \eeq
  for some $C>0$,
  where $\bar{B}_s:=\F^{-1}_{\R^n}(\mathscr{B}\chi_{s,\kappa})$, and $M_{DHL}$ is the discrete Hardy-Littlewood maximal operator.   Since the operator associated to the  multiplier $\mathscr{B}$ is $L^p(\R^n)$ bounded, and $\chi_{s,\kappa}$ is  supported in a small neighborhood of the original,  we deduce by transference principle
  \beq\label{lp1}
  \|\bar{B}_s(D)f\|_{\ell^p(\Z^n)}\les \|f\|_{\ell^p(\Z^n)}.
  \eeq
  Hence,  the left-hand side of (\ref{oom})  is
  $$\les 2^{-j/2}2^{-s}\sup_{\bb\in [0,1)^n}\|\bar{B}_s*(\NE_{-\bb}f)\|_{\ell^p(\Z^n)}
  \les 2^{-j/2}2^{-s} \|f\|_{\ell^p(\Z^n)}.$$
  This with  (\ref{varsum}) leads to that  the left-hand side of (\ref{c93}) is
  $$\les \sum_{j> 2^{C_1s}}2^{-j/2}2^{-s} \|f\|_{\ell^p(\Z^n)}\les 2^{-s} \|f\|_{\ell^p(\Z^n)},$$
  which  completes the proof of (\ref{c93}).  As a consequence,
  to complete the proof of (\ref{Goo1}),
 it suffices to show that for each $p\in (1,\infty)$,
\beq\label{c96}
\Big(V_{s,R}^{-n}\sum_{u\in [V_{s,R}]^n}\|\big((\mathscr{L}_{s,\A(x)} [(\NE_u\VV_j) ~\mathscr{B}](D)f)(x)\big)_{j>2^{C_1s}}\|_{\ell^p(x\in \B_R;V^r)}^p\Big)^{1/p}\les 2^{-\gamma_p s}\|f\|_{\ell^p(\Z^n)}
\eeq
with $\gamma_p$ given as in (\ref{de11}).
Note that the function $\A(x)$, when restricted to $x$ in $\B_R$, can be extended to a function that is $2R$-periodic in each coordinate. Thus,
to achieve (\ref{c96}),  it suffices to show that %for each $p\in (1,\infty)$,   the estimate
\beq\label{c9494}
\Big(V_{s,R}^{-n}\sum_{u\in [V_{s,R}]^n}\|\big((\mathscr{L}_{s,\A(x)} [(\NE_u\VV_j)~\mathscr{B}](D)f)(x)\big)_{j>2^{C_1s}}\|_{\ell^p(x\in \Z^n;V^r)}^p\Big)^{1/p}\les 2^{-\gamma_p s}\|f\|_{\ell^p(\Z^n)}
\eeq
 for any function $\A(x)=\frac{a(x)}{q(x)}$ that is $2R$-periodic in each coordinate and belongs to $\mathcal{A}_s$.
By using (\ref{MI11}) to expand the operator $\mathscr{L}_{s,\A(x)}$,  we reduce the proof of  (\ref{c9494}) to showing
%that  for all $2R$-periodic (in every coordinate)  $\A(x)=\frac{a(x)}{q(x)}\in \mathcal{A}_s$,
\beq\label{c94}
\begin{aligned}
&\  V_{s,R}^{-n}\sum_{u\in [V_{s,R}]^n}\sum_{x\in\Z^n}\|\big(
\sum_{\bb\in \frac{1}{q(x)}[q(x)]^n}S(\A(x),\bb)e(x\cdot \bb)(h_j*\bar{B}_s*\NE_{-\bb}f)(x-u)
\big)_{j>2^{C_1s}}\|_{V^r}^p\\
\les&\  2^{-\gamma_p s p}\|f\|_{\ell^p(\Z^n)}^p.
\end{aligned}
\eeq
By changing variables $x\to x+u$ and  $u\to v-x$ in order,   we rewrite the left-hand side of (\ref{c94}) as
\beq\label{Key1}
\begin{aligned}
 V_{s,R}^{-n}\sum_{x\in\Z^n}\sum_{v\in [V_{s,R}]^n+x} \|\big(\mathcal{B}^s_{j}(v,x)\big)_{j>2^{C_1s}}\|_{V^r}^p
\end{aligned}
\eeq
with $\mathcal{B}^s_{j}(v,x):=\sum_{\bb\in \frac{1}{q(v)}[q(v)]^n}S(\A(v),\bb)e(v\cdot \bb)(h_j*\bar{B}_s*\NE_{-\bb}f)(x)$.
Since  $\A(\cdot)$ is  $2R$-periodic in each coordinate  and $V_{s,R}$ is divisible by $2R$,   $\A(\cdot)$ is also  $V_{s,R}$-periodic in each coordinate. Moreover, since $V_{s,R} \ \bb\in \Z^n$ (by the definitions of $Q_s$ and $V_{s,R}$), the function  $\mathcal{B}^s_{j}(\cdot,x)$ is   $V_{s,R}$-periodic in every coordinate. So
(\ref{Key1}) equals
\beq\label{KEY6}
V_{s,R}^{-n}\sum_{v\in [V_{s,R}]^n}\sum_{x\in\Z^n} \|\big(\mathcal{B}^s_{j}(v,x)\big)_{j>2^{C_1s}}\|_{V^r}^p=:V_{s,R}^{-n}\sum_{v\in [V_{s,R}]^n}\sum_{x\in\Z^n} \|\big((h_j*F^s(v,\cdot))(x)\big)_{j>2^{C_1s}}\|_{V^r}^p
\eeq
%Write $\mathcal{B}_{j,v}(x)$ as  $(h_j*F^s(v,\cdot))(x)$
with $F^s(v,y)$ given by
\beq\label{NO71}
F^s(v,y):=
 \sum_{\bb\in \frac{1}{q(v)}[q(v)]^n}S(\A(v),\bb)e(v\cdot \bb)(\bar{B}_s*\NE_{-\bb}f)(y).
\eeq
By combining (\ref{Key1}),  (\ref{KEY6}) and  the left-hand side of (\ref{c94}),  to show (\ref{c94}), it suffices to prove
\beq\label{c951}
V_{s,R}^{-n}\sum_{v\in [V_{s,R}]^n}\sum_{x\in\Z^n} \|\big((h_j*F^s(v,\cdot))(x)\big)_{j>2^{C_1s}}\|_{V^r}^p\les \ 2^{-\gamma_p  s p}\|f\|_{\ell^p(\Z^n)}^p.
\eeq
Let $\theta$ be the function as in Lemma \ref{de1}. Since
 $|\hat{h}(2^j\xi)-\hat{\theta}(2^j\xi)|\les \min\{2^j|\xi|,(2^j|\xi|)^{-1}\}$ for $\xi\in\R^n$, we deduce by the classical Calder\'{o}n-Zygmund  and Littlewood-Paley
 theories that
 \beq\label{ENC1}
 \|\big(\sum_{j\in\Z}|(\theta_j-h_j)*g|^2\big)^{1/2}\|_{L^p(\R^n)}
\les\ \|g\|_{L^p(\R^n)}.
 \eeq
By Theorem 1.1 in \cite{JSW08} and (\ref{ENC1}), we further obtain that for every $(p,r)\in (1,\infty)\times (2,\infty)$,
\beq\label{ENC2}
\begin{aligned}
\|\big(h_j*_{\R^n}g\big)_{j\in\Z}\|_{L^p(\R^n;V^r)}\les&\
\|\big(\theta_j*_{\R^n}g\big)_{j\in\Z}\|_{L^p(\R^n;V^r)}
+\|\big(\sum_{j\in\Z}|(\theta_j-h_j)*g|^2\big)^{1/2}\|_{L^p(\R^n)}\\
\les&\ \|g\|_{L^p(\R^n)}.
\end{aligned}
\eeq
Furthermore, invoking that  $\VV=\hat{h}$ and $C_1$ is sufficiently large,
using Proposition \ref{PMST} (with $Q=1$ and $\mathfrak{m}=0$) as well as (\ref{ENC2}), we can  infer
\beq\label{Key2}
\sum_{x\in\Z^n} \|\big((h_j*F^s(v,\cdot))(x)\big)_{j>2^{C_1s}}\|_{V^r}^p
\les \|F^s(v,\cdot)\|_{\ell^p(\Z^n)}^p.
\eeq
Specifically, the inequality (\ref{Key2}) remains valid when replacing ${j>2^{C_1s}}$ with $j\in\N$.
By combining   (\ref{NO71}) and  (\ref{Key2}), to prove (\ref{c951}), it suffices to show
%Then we can reduce the goal to the estimate
\beq\label{go1}%V_{s,R}^{-n}\sum_{v\in [V_{s,R}]^n}\|F_v\|_{\ell^p}^p=
V_{s,R}^{-n}\sum_{v\in [V_{s,R}]^n}\sum_{x\in\Z^n}\sup_{\A=\frac{a}{q}\in\mathcal{A}_s}|\sum_{\bb\in \frac{1}{q}[q]^n}S(\A,\bb)e(v\cdot \bb) \big(\bar{B}_s*\NE_{-\bb}f\big)(x)|^p\les 2^{-\gamma_p s p}\|f\|_{\ell^p(\Z^n)}^p.
\eeq
Subsequently changing variables back, $v\to u+x$ and  $x\to x-u$ in order, and using   $V_{s,R}~\bb\in\Z$ again,  we further streamline the proof of (\ref{go1}) to demonstrating
\beq\label{ll98}
V_{s,R}^{-n}\sum_{u\in [V_{s,R}]^n}\|\sup_{\A=\frac{a}{q}\in\mathcal{A}_s}|\sum_{\bb\in \frac{1}{q}[q]^n}S(\A,\bb)e(x\cdot \bb)[\bar{B}_s(\cdot-u)*\NE_{-\bb}f](x)\|_{\ell^p(x\in\Z^n)}^p\les 2^{-\gamma_p sp}\|f\|_{\ell^p(\Z^n)}^p.
\eeq
Notice that for each $u\in [V_{s,R}]^n$,
$$\sup_{\A=\frac{a}{q}\in\mathcal{A}_s}\big|\sum_{\bb\in \frac{1}{q}[q]^n}S(\A,\bb)e(x\cdot \bb)\big[\bar{B}_s(\cdot-u)*\NE_{-\bb}f\big]\big|(x)
=\sup_{\A\in\mathcal{A}_s}|\mathscr{L}_{s,\A}[\NE_{-u}\mathscr{B}](D)f|(x).$$
Hence, to obtain (\ref{ll98}), it suffices to show that for any $u\in [V_{s,R}]^n$,
\beq\label{ll99}
\|\sup_{\A\in\mathcal{A}_s}|\mathscr{L}_{s,\A}[\NE_{-u} \mathscr{B}](D)f |\|_{\ell^p(\Z^n)}\les 2^{-\gamma_p s}\|f\|_{\ell^p(\Z^n)}
\eeq
with the implicit constant  independent of $u$.
By (\ref{de11}) and Proposition \ref{PIW} with $m=\NE_{-u}\mathscr{B}$,
%which, with the estimate below independent of $u$
$$\|\sup_{\A\in\mathcal{A}_s}|\mathscr{L}_{s,\A}[\NE_{-u} \mathscr{B}](D)f |\|_{\ell^p(\Z^n)}\les
2^{-\gamma_p s}\|\mathscr{L}_{s}^\#[\NE_{-u} \mathscr{B}](D)f \|_{\ell^p(\Z^n)}
\les
2^{-\gamma_p s}\|f\|_{\ell^p(\Z^n)},$$
as desired. This ends the proof of (\ref{Goo1}).

We next prove (\ref{Goo2}).
Note the the $R^\e$-loss will be needed in this case (In fact, it is easy to check that this  loss can be mitigated to  a logarithmic loss with respect to  the scale $R$). Since $2^{2^s}\le R$,
to prove (\ref{Goo2}),
it suffices to show that for every $p\in(1,\infty)$,   %there is $c_p>0$ such that
\beq\label{c81}
(\sum_{\A\in \mathcal{A}_s}\big\|(\mathscr{L}_{s,\A} [{\VV_j}~ \mathscr{B}](D)f)_{j>2^{C_1s}}\big\|_{\ell^p(\Z^n;V^r)}^p)^{1/p}\les_\e 2^{(1/p+\e-\gamma_p)s} \|f\|_{\ell^p(\Z^n)}
\eeq
holds for any sufficiently small $\e>0$.
Let    $V_{s,1}$ be a constant defined  by
\beq\label{KL2}
V_{s,1}=V_{s,R}\big|_{R=1}.
\eeq
Repeating the previous arguments yielding  (\ref{c93}), we also obtain
 for any $u\in [V_{s,1}]^n$ and  any $\A\in \mathcal{A}_s$,
\beq\label{c939}
\Big(V_{s,1}^{-n}\sum_{u\in [V_{s,1}]^n}\|\big(\mathscr{L}_{s,\A} [(\mathscr{V}_j-\NE_u\VV_j)~ \mathscr{B}](D)f\big)_{j>2^{C_1s}}\|_{\ell^p( \Z^n;V^1)}^p\Big)^{1/p}\les 2^{-s}\|f\|_{\ell^p(\Z^n)}.
\eeq
Keep  (\ref{new1}) in mind.   To prove (\ref{c81}),  by
(\ref{c939}) and
the triangle inequality,
 it suffices to show that  for any $\A\in \mathcal{A}_s$,
\beq\label{c898}
\Big(V_{s,1}^{-n}\sum_{u\in [V_{s,1}]^n}\big\|\big(\mathscr{L}_{s,\A} [({\NE_u\VV_j})~\mathscr{B}](D)f\big)_{j>2^{C_1s}}\big\|_{\ell^p(\Z^n;V^r)}^p\Big)^{1/p}\les 2^{-\gamma_p s} \|f\|_{\ell^p(\Z^n)}.
\eeq
By preforming similar  arguments as  yielding (\ref{c9494}),
we can achieve (\ref{c898}) as well.  In fact, the proof at this moment   is easier.
This ends the proof of (\ref{Goo2}).
\end{proof}
Let  $\psi_j$ be the function  defined as in  Subsection \ref{Not}, and let $K$ be the kernel function  given by (\ref{v1}).
\begin{lemma}\label{endle}
Let $s\in\N$, $R\ge 1$ and let $\A(x)$ denote an arbitrary function from $\Z^n$ to  $\mathcal{A}_s$.
Let
$$\K_j=K\psi_j$$ with $j\in\N_0$, %with $K$ given by (\ref{v1}).
%let $\K_j$ be a mean zero $C^1$ function supported on $\{|x|\sim 2^j\}$
%such that
%%\beq\label{CO1}
%2^{jn}|\K_j|+2^{j(n+1)}|\na \K_j|\les1
%\eeq
%for all $j\ge 1$ and all  $y\in \R^n$.
and let  $\K^{a,b}=\sum_{a\le j<b}\K_j$ whenever $0\le a<b$. Then for each $p\in(1,\infty)$,
\beq\label{dou1}
\| \big(\mathscr{L}_{s,\A(x)}[\widehat{\K^{0,j}}](D)f(x)\big)_{j\in\N}\|_{\ell^p(x\in\B_R;\V^r)}
\les_\e R^\e 2^{-\gamma_p s}\|f\|_{\ell^p(\Z^n)}
\eeq
with $\gamma_p$ given as in (\ref{de11}), where $\widehat{\K^{0,j}}=\F_{\R^n}\K^{0,j}$.
\end{lemma}
We expect that this result  will also apply to more general functions $\K_j$, but we opt not to pursue this direction since Lemma \ref{endle} is sufficient for our proof. 
Remember that  the $\V^r$ norm is defined in (\ref{vardef}). Considering that we will use (\ref{simple}) and (\ref{var1}) in proving our main results, it is more convenient to use the $\V^r$ norm instead of the $V^r$ seminorm.
\begin{proof}
%Split $j\in\N$ in (\ref{dou1}) into $1\le j\le 2^{C_1s}$ and $j> 2^{C_1s}$, where $C_1$ is  given as (\ref{L1}).
We may reduce  the proof of    (\ref{dou1}) to proving
 \begin{align}
\| \big(\mathscr{L}_{s,\A(x)}[\widehat{\K^{0,j}}](D)f (x)\big)_{1\le j\le 2^{2C_1s}}\|_{\ell^p(x\in\Z^n;V^r)}
&\les\  2^{-\gamma_p s}\|f\|_{\ell^p(\Z^n)}\ \ \ {\rm and}\label{dou11}\\
\| \big(\mathscr{L}_{s,\A(x)}[\widehat{\K^{0,j}}](D)f (x)\big)_{j> 2^{C_1s}}\|_{\ell^p(x\in\B_R;V^r)}
&\les_\e R^\e 2^{-\gamma_p s}\|f\|_{\ell^p(\Z^n)}.\label{dou133}
\end{align}
In fact, by
using (\ref{CC1}), (\ref{de11}) and Proposition \ref{PIW},  $\| \mathscr{L}_{s,\A(x)}[\widehat{\K^{0,1}}](D)f \|_{\ell^p(x\in\Z^n)}$ is
$$
 \begin{aligned}
\les&\  \|\sup_{\A\in\mathcal{A}_s}| \mathscr{L}_{s,\A}[\widehat{\K_0}](D)f |\|_{\ell^p(\Z^n)}
\les\  2^{-\gamma_p s}\| \mathscr{L}_{s}^\#[\widehat{\K_0}](D)f \|_{\ell^p(\Z^n)}
\les\  2^{-\gamma_p s} \|f\|_{\ell^p(\Z^n)},
\end{aligned}
$$
which with (\ref{dou11}) and (\ref{Ad1}) gives that
 \beq\label{dou111}
\| \sup_{1\le j\le 2^{2C_1s}}\sup_{\A\in \mathcal{A}_s}|\mathscr{L}_{s,\A}[\widehat{\K^{0,j}}](D)f| \|_{\ell^p(\Z^n)}
\les 2^{-\gamma_p s} \|f\|_{\ell^p(\Z^n)}.
\eeq
Then, invoking  the definitions (\ref{simple}) and (\ref{vardef}), we achieve  (\ref{dou1})  by
combining (\ref{dou11}), (\ref{dou133}) and (\ref{dou111}).
Next, we prove (\ref{dou11}) and  (\ref{dou133}) in order.

We first prove (\ref{dou11}).  By the numerical inequality (\ref{num1}),
we have
\beq\label{R01}
\| \big(\mathscr{L}_{s,\A(x)}[\widehat{\K^{0,j}}](D)f(x) \big)_{1\le j\le 2^{2C_1s}}\|_{V^r}
\les
\sum_{l=0}^{2C_1s}\big(\sum_{j=0}^{2^{2C_1s-l}}|\mathscr{L}_{s,\A(x)} [\F_{\R^n}\{{\K^{j2^l,(j+1)2^l}}\}](D)f(x)|^2\big)^{1/2}.
\eeq
Let $\{\va_i(t)\}_{i=0}^\infty$ be the sequence of Rademacher functions on $[0,1]$ satisfying  (\ref{xinqing}).
By (\ref{R01}), to prove (\ref{dou11}), it suffices to show that for all $t\in [0,1]$ and $0\le l\le 2C_1 s$,
 \beq\label{dou211}
\|\sup_{\A\in\mathcal{A}_s} |\mathscr{L}_{s,\A} [\sum_{j=0}^{2^{2C_1s-l}} \va_j(t)\F_{\R^n}\{{\K^{j2^l,(j+1)2^l}}\}](D)f|\|_{\ell^p(\Z^n)}
\les 2^{-\gamma_p s} \|f\|_{\ell^p(\Z^n)}.
\eeq
Claim that for all $t\in [0,1]$,
\beq\label{pw1}
\|\sum_{j=0}^{2^{2C_1s-l}} \va_j(t) {\K^{j2^l,(j+1)2^l}(D)}{f}\|_{L^p(\R^n)}
%=\\|(\sum_{j=1}^{2^{Cs-l}}|\sum_{j2^l\le v<(j+1)2^l}\tilde{P}_{v}f|^2)^{1/2}\|_{p}
\les \|f\|_{L^p(\R^n)}
\eeq
with the  implicit constant independent of $t$, $s$ and $l$.
Using (\ref{pw1}) and Proposition \ref{PIW}, we deduce
$$
\|\mathscr{L}_s^\#[\sum_{j=0}^{2^{2C_1s-l}} \va_j(t)\F_{\R^n}\{{\K^{j2^l,(j+1)2^l}}\}](D)f\|_{\ell^p(\Z^n;L^p_t([0,1]))}
\les \|f\|_{\ell^p(\Z^n)},
$$
which with (\ref{de11}) and (\ref{CC1})  gives (\ref{dou211}).  Thus, to finish the proof of (\ref{dou11}), it remains to prove the above claim (\ref{pw1}).
 By the Littlewood-Paley decomposition $\sum_{v\in\Z} P_{v}f=f$ and $\int \K_k=0$  for all $k\in\Z$,
 we reduce the proof of   (\ref{pw1}) to showing
 that for each $p\in(1,\infty)$,
\beq\label{L2}
\|\sum_{j=0}^{2^{2C_1s-l}}  \sum_{k=j2^l}^{(j+1)2^l-1} \va_j(t) \K_k*_{\R^n} P_{v-k}f\|_{L^p(\R^n)}
%=\\|(\sum_{j=1}^{2^{Cs-l}}|\sum_{j2^l\le v<(j+1)2^l}\tilde{P}_{v}f|^2)^{1/2}\|_{p}
\les 2^{-\gamma_p |v|}\|f\|_{L^p(\R^n)}.
\eeq
 %Expanding the notation $\K^{j2^l,(j+1)2^l}=\sum_{k=j2^l}^{(j+1)2^l-1} \K_k$, we write the LHS of (\ref{L2}) as
%$\|\sum_{j=0}^{2^{Cs-l}}\sum_{k=j2^l}^{(j+1)2^l-1} \va_j(t) {\K_k(D)}{f}\|_{L^p(\R^n)}$
Then,
by the dual arguments, the Littlewood-Paley theory and interpolation, to prove (\ref{L2}), it suffices to show
\beq\label{L3}
\| \big(\sum_{k\in \Z}|\K_k*_{\R^n} P_{v-k}f|^2\big)^{1/2}\|_{L^p(\R^n)}
%=\\|(\sum_{j=1}^{2^{Cs-l}}|\sum_{j2^l\le v<(j+1)2^l}\tilde{P}_{v}f|^2)^{1/2}\|_{p}
\les 2^{-|v|{\ind {p=2}}}\|f\|_{L^p(\R^n)}
\eeq
Since
$|\K_k*_{\R^n}P_{v-k}f|\les M_{HL}(P_{v-k}f)$, where $M_{HL}$ denotes the   Hardy-Littlewood  maximal  operator on $\R^n$,   (\ref{L3}) for the cases  $p\neq2$ is a result of  the Fefferman-Stein inequality and the Littlewood-Paley  inequality. Hence, it remains to prove  (\ref{L3}) for the case $p=2$.
Noting
\beq\label{ine1}
|\widehat{\K_k}(\xi)|\les \min\{2^k|\xi|,|2^k\xi|^{-1}\},
\eeq
 we have
$$
\big(\sum_{k\in\Z}  | \widehat{\K_k}(\xi)|^2| \psi_{v-k}(\xi)|^2\big)^{1/2}
%\les    \sum_{k\in\Z}  | \widehat{\K_k}(\xi)|| \psi_{v-k}(\xi) |
\les  \sum_{k\in\Z} |\psi_{v-k}(\xi)| \min\{2^k|\xi|,|2^k\xi|^{-1}\}
\les 2^{-|v|},
$$
which with Plancherel's identity yields (\ref{L3}) for the case $p=2$.

Next, we  consider (\ref{dou133}).
By using the definition of the semi-norm $V^r$,  it suffices to show
\beq\label{c91}
\|\big((\mathscr{L}_{s,\A(x)} [\widehat{\K^{j,\infty}}](D)f)(x)\big)_{j>2^{C_1s}}\|_{\ell^p(x\in \B_R;V^r)}\les_\e R^\e 2^{-\gamma_p s}\|f\|_{\ell^p}.
\eeq
Let $\VV$ be  the function  as in Lemma \ref{ccz2}, %and let $\VV_j(\xi)=\VV(2^j\xi).$
and let
$$\mathfrak{M}_j^{(1)}(\xi):=\widehat{\K^{j,\infty}}(\xi)-{\VV_j}(\xi)~\widehat{\K^{0,\infty}}(\xi)\ \ \ \ (\xi\in \R^n),$$
which satisfies  by a  routine  computation that
\beq\label{ine3}
|\mathfrak{M}_j^{(1)}(\xi)|\les \min\{2^j|\xi|,|2^j\xi|^{-1}\}.
\eeq
We can reduce the proof of  (\ref{c91}) to proving
\begin{align}
\|\big(\sum_{j>2^{C_1s}}\sup_{\A\in\mathcal{A}_s}|\mathscr{L}_{s,\A} [\mathfrak{M}_j^{(1)} ](D)f|^2\big)^{1/2}\|_{\ell^p(\Z^n)}&\les\  2^{-\gamma_p s}\|f\|_{\ell^p(\Z^n)}\ \ \ {\rm and}\label{c51}\\
\|\big((\mathscr{L}_{s,\A(x)} [{\VV_j}\widehat{\K^{0,\infty}}](D)f)(x)\big)_{j>2^{C_1s}}\|_{\ell^p(x\in \B_R;V^r)}&\les_\e R^\e 2^{-\gamma_p s}\|f\|_{\ell^p}. \label{c52}
\end{align}
We first  use Lemma \ref{ccz2} to prove (\ref{c52}).
By  similar arguments as yielding (\ref{L3}),  we obtain
$$\|{\K^{0,\infty}}*_{\R^n} f \|_{L^p(\R^n)}=\|\sum_{k=0}^\infty\K_k*_{\R^n}f \|_{L^p(\R^n)}\les \|f\|_{L^p(\R^n)}$$ 
for each $p\in(1,\infty)$. This  with  (\ref{ine1}) and Lemma \ref{ccz2} ($\gamma=1$ and $\mathscr{B}={\K^{0,\infty}}$) leads to
 (\ref{c52}).
Thus,
to finish the proof of (\ref{c91}), it remains to  prove  (\ref{c51}). We will use Lemma  \ref{ccz} to achieve this goal.
By invoking  $\K_j=K \psi_j$,  we can bound  $\mathfrak{M}_j^{(1)}(D)f$ as
$$
\begin{aligned}
|\mathfrak{M}_j^{(1)}(D)f|\les&\  |{\K^{j,\infty}}*_{\R^n}f|+|{h_j}*_{\R^n}{
\K^{0,\infty}}*_{\R^n}f|\\
\les&\  M_{HL}(\mathcal{T}_0f)
+M_{HL}(\mathcal{T}f)+M_{HL}(M_{HL}f)+M_{HL}(\mathcal{T}_jf),
\end{aligned}$$
where
$$
\mathcal{T}_jf(x):={\rm p.v.}\int_{|y|\le 2^j}f(x-y)K(y)dy,\ \  \mathcal{T}f(x):={\rm p.v.}\int_{\R^n}f(x-y)K(y)dy.
$$
This with  the Fefferman-Stein inequality and the vector-valued inequalities of $\mathcal{T}_j$ and $\mathcal{T}$ yields
\beq\label{ine4}
\|\big(\sum_{j\in\Z}|\mathfrak{M}_j^{(1)}(D)f_j|^2\big)^{1/2}\|_{L^p(\R^n)}
\les \|\big(\sum_{j\in\Z}|f_j|^2\big)^{1/2}\|_{L^p(\R^n)}.
\eeq
Applying Lemma \ref{ccz} with (\ref{ine3}) and (\ref{ine4}), we finally  achieve  (\ref{c51}).
 This ends the proof of (\ref{dou133}).
\end{proof}
\section{Major arcs estimate I: Proof of Proposition \ref{t21}}
\label{slong2}
In this section, we  obtain major arcs estimate I in Proposition \ref{t21}.
 The proof is based on Proposition \ref{PIW}, Lemmas \ref{de1}, \ref{ccz}, the Stein-wainger-type estimate and the first trick mentioned  in Subsection \ref{diff}.
 In particular, we shall establish  a triple  maximal estimate (see (\ref{hee0}) below), which will also be  employed in the next section.
\subsection{Reduction of Proposition \ref{t21}}
Keep the notation (\ref{not1}) in mind.
For each $j\ge 1$,   we define
\beq\label{dc1}
\begin{aligned}
\mathcal{S}_j^m:=& ~\{x\in \Z^n:\ |\mu(x)|\in I_{j,m}\},\ \ I_{j,m}:= [2^{m-2dj},2^{m+1-2dj}),~\ \ m\ge 1,\\
\mathcal{S}_j^0:=&~\{x\in \Z^n:\ |\mu(x)|\in I_{j,0}\},\ \ \ \ I_{j,0}:=  (-\infty, 2^{1-2dj}).
\end{aligned}
\eeq
Obviously, for each $x\in\Z^n$, we have
\beq\label{le1}
{\ind {\mathcal{S}_j^0}}(x)+ \sum_{m\ge 1} {\ind {\mathcal{S}_j^m}}(x)=1
\ \ \ {\rm and}\ \  \ \sum_{j\in \Z}{\ind {\mathcal{S}_j^m}}(x)\le 1\ \ {\rm whenever}\ m\ge 1.
\eeq
%We can confine the parameter $m$ to $m \lesssim \epsilon_\circ (j)$ with $\e_\circ=\e_\circ(p_1,p_2,\mathcal{C})$ ($\epsilon_\circ (j)$ is given by  (\ref{zhong}) with $j_\circ=j$), although this restriction does not affect the proof.
 We provide first  two lemmas. % Remember  that  $L^{(1),s}_{j,\la(x)}$ is given by (\ref{f1}) and
 Let $\la(x)$  denotes an arbitrary function from $\Z^n$ to  $[0,1]$.
\begin{lemma}\label{l5.2}
Let  $s\in\N$ and  $p\in (1,\infty)$. Then for every $\e_\circ'\in (0,1)$,
 $$
 \|\big({\ind {\mathcal{S}_j^0}(x)}\ [L_{j,\la(x),\e_\circ'}^{(1),s}(D)f](x)\big)_{j\in\N^B}\|_{\ell^p(x\in \Z^n;\ell^2)}\les\  2^{-\gamma_p s} \|f\|_{{\ell^p}(\Z^n)}
 $$
 with $\gamma_p$ given as in (\ref{de11}) and $L^{(1),s}_{j,\la(x),\e_\circ'}$  given by (\ref{f1}) with $\e_\circ=\e_\circ'$.
  %where $L^{(1),s}_{j,\la(x)}$ is given by (\ref{f1}).
 \end{lemma}

 \begin{lemma}\label{l5.1}
Let  $m\ge 1$ and $s\in\N$.
Then for every $\e_\circ''\in (0,1)$, the inequality
\begin{align}
\big\| \sup_{j\in\N^B }
|{\ind {\mathcal{S}_j^m}(x)}~\big[L_{j,\la(x),\e_\circ''}^{(1),s}(D)f\big](x)|\big\|_{\ell^2(x\in \Z^n)}\les&\  2^{-c (s+m)}\|f\|_{{\ell^2}(\Z^n)}\label{hee90}
\end{align}
holds for some $c>0$, where $L^{(1),s}_{j,\la(x),\e_\circ''}$  given by (\ref{f1}) with $\e_\circ=\e_\circ''$.
%for some $c>0$.
%where $L^{(1),s}_{j,\la(x)}$ is given by (\ref{f1}).
\end{lemma}
%%%

\begin{proof}[Proof of Proposition  \ref{t21} accepting Lemmas \ref{l5.2} and \ref{l5.1} ]
By the equality in (\ref{le1}),  to achieve Proposition  \ref{t21},
it suffices to show that for each  $p\in [p_1,p_2]$ and $r\in(2,\infty)$,
there is a constant $c_p>0$ such that
\begin{align}
\|\big(
\sum_{1\le s\le \e_\circ (j)}{\ind {\mathcal{S}_j^m}}(x)\
[L_{j,\la(x),\e_\circ}^{(1),s}(D)f](x)\big)_{j\in\N^B}\|_{\ell^p(x\in \Z^n;V^r)}\les&\  2^{-c_pm}\|f\|_{{\ell^p}(\Z^n)} \label{high1}
\end{align}
for every  $m\ge 1$,  and
\begin{align}
\|\big(
\sum_{1\le s\le \e_\circ (j)}{\ind {\mathcal{S}_j^0}(x)}\ [L_{j,\la(x),\e_\circ}^{(1),s}(D)f](x)\big)_{j\in\N^B}\|_{\ell^p(x\in \Z^n;V^r)}\les&\  \|f\|_{{\ell^p}(\Z^n)}.\label{low}
\end{align}
Here $\e_\circ=\e_\circ(p_1,p_2,\mathcal{C})$ and $\e_\circ(j)$ is given by (\ref{zhong}) with $j_\circ=j$.
Notice that (\ref{low})  is a direct result of Lemma \ref{l5.2} and Minkowski's inequality  since
(\ref{varsum}).
Thus, %to end the proof of  Proposition  \ref{t21},
 it remains to show  (\ref{high1}).
%We split   the matter  into two parts: $p\in [2,\infty)$ and $p\in (1,2)$.
We first prove    (\ref{high1}) for  the case
$p=2$.   Indeed, by
(\ref{varsum})
and  the inequality in (\ref{le1}),
the $V^r$ semi-norm on the  left-hand side of (\ref
{high1})  is
$$\les \sum_{j\in\N^B}
\sum_{s\in\N}  {\ind {\mathcal{S}_j^m}(x)}| [L_{j,\la(x),\e_\circ}^{(1),s}(D)f](x)|
\les \sum_{s\in\N} \sup_{j\in\N^B }
|{\ind {\mathcal{S}_j^m}(x)}~[L_{j,\la(x),\e_\circ}^{(1),s}(D)f](x)|,$$
which with Lemma \ref{l5.1}  and the triangle inequality yields (\ref{high1}) for the case $p=2$.
%\in [2,\infty)$.
To end the proof of  (\ref{high1}), by  interpolation,  it suffices to prove  
% (\ref{high1}) without the factor $2^{-c_p m}$, that is, for  
 that for $\e_\circ=\e_\circ(p_1,p_2,\mathcal{C})$ and every $p\in (1,\infty)$,
%By interpolating, % with (\ref{high1}) ,
%it suffices  to show  that for every $p\in (1,\infty)$,
\beq\label{c1}
\|\big(
\sum_{1\le s\le  \e_\circ (j)}{\ind {\mathcal{S}_j^m}(x)}\ [L_{j,\la(x),\e_\circ}^{(1),s}(D)f](x)\big)_{j\in\N^B}\|_{\ell^p(x\in \Z^n;V^r)}\les\  \|f\|_{{\ell^p}(\Z^n)}.
\eeq
We next apply 
the first trick mentioned  in Subsection \ref{diff}.
Using  the expression  (\ref{p87}) and letting
\beq\label{NNM1}
\Lambda_{j,\e_\circ,\la,m}:=\Lambda_{j,\e_\circ,\la}\cap \mathcal{S}_j^m,
\eeq
with $\Lambda_{j,\e_\circ,\la}$ and $\mathcal{S}_j^m$ given as in Subsection \ref{sslong1} and   (\ref{dc1}), respectively,
 we have
\beq\label{bc98}{
\ind {\Lambda_{j,\e_\circ,\la,m}}}(x) \ma_{j,\la(x)}(\xi)
={\ind {\mathcal{S}_j^m}(x)}\sum_{1\le s\le \e_\circ (j)}L_{j,\la(x),\e_\circ}^{(1),s}(\xi)+ {\ind {\Lambda_{j,\e_\circ,\la,m}}}(x) E^{(1)}_{j,\la(x),\e_\circ}(\xi).
\eeq
By using a routine  computation and the inequality in (\ref{le1}), we can  infer that for each $p\in(1,\infty)$,
\beq\label{bc1}
\begin{aligned}
&\ \ \| {\ind {\mathcal{S}_j^m}(x)} \big (\ma_{{j},\la(x)}(D)f\big)(x)\|_{\ell^p(x\in\Z^n;\ell^1(j\in \N^B))}\\
\les&\  \|\sup_{j\in \N^B} \sup_{\la\in [0,1]}| \ma_{{j},\la}(D)f|\|_{\ell^p(\Z^n)}
\les  \|M_{D HL}f\|_{\ell^p(\Z^n)}\les \|f\|_{\ell^p(\Z^n)};
\end{aligned}
\eeq
moreover,  we can deduce from (\ref{00}) that for every $p\in(1,\infty)$,
\beq\label{bc2}
\begin{aligned}
&\ \|{\ind { \Lambda_{j,\e_\circ,\la,m}}(x)}  \big(E^{(1)}_{j,\la(x),\e_\circ}(D)f\big)(x)\|_{\ell^p(x\in\Z^n;\ell^1(j\in \N^B))}\\
\les&\   \sum_{j\in\N^B}
\|\sup_{\la\in X_{j,\e_\circ}}|E^{(1)}_{j,\la,\e_\circ}(D)f|\|_{{\ell^p(\Z^n)}}
\les \|f\|_{{\ell^p(\Z^n)}}.
\end{aligned}
\eeq
Finally, we can obtain (\ref{c1}) by
combining (\ref{bc98})-(\ref{bc2}).  This ends the proof of  Theorem \ref{t21} under the assumptions that Lemmas \ref{l5.2} and \ref{l5.1} hold.
\end{proof}
\subsection{Proof of Lemma \ref{l5.2}}
 In this subsection, we shall prove Lemma \ref{l5.2}. Since the value of $\ka$
is not important, hereafter we will use the notation (\ref{not32}).
  Since  $x\in {\mathcal{S}_j^0}$, we have  $|\mu(x) 2^{2dj} |\le 2$ at this moment.
 Changing variables $y\to 2^j y$ and using Taylor's expansion, we write
   $$
   \begin{aligned}
   \phi_{j,\mu (x)}^{(1)}(\xi)
   =&\ \int_{1/2\le |y|\le 1} e\big(\mu(x) 2^{2dj}  |y|^{2d}+2^jy\cdot\xi\big){K}_0(y)dy\\
   =&\  \rho_{0}(2^j\xi)+\sum_{l=1}^\infty \frac{(2\pi i )^l}{l!} (\mu(x) 2^{2dj})^l \rho_{l}(2^j\xi),
   \end{aligned}
  $$
  where
  $$\rho_{l}(\xi):=\int_{1/2\le |y|\le 1} e(y\cdot\xi)|y|^{2dl}{K}_0(y)dy\ \ \ \  (l\ge 0).$$
  %Since  $x\in {\mathcal{S}_j^0}$, we have  $|\mu(x) 2^{2dj} |\le 1$, and then
  Then  we reduce the matter to showing
  \begin{align}
 \|\big({\ind {\mathcal{S}_j^0}(x)}~ (\mathscr{L}_{s,\A}[\rho_{0}(2^j\cdot)](D)f)(x) \big)_{j\in\N^B}\|_{\ell^p(x\in \Z^n;\ell^2)}
 &\les\  2^{-c_p s}\|f\|_{\ell^p(\Z^n)}\ \ {\rm and}\label{al1}\\
  \|\big({\ind {\mathcal{S}_j^0}(x)} (\mu(x)2^{2dj})^l (\mathscr{L}_{s,\A}[\rho_{l}(2^j\cdot)](D)f)(x) \big)_{j\in\N^B}\|_{\ell^p(x\in \Z^n;\ell^2)}
& \les\  C^l 2^{-c_p s} \|f\|_{\ell^p(\Z^n)}\ (l\ge 1).\label{al2}
 \end{align}
 A routine   computation  gives $|\rho_{0}(2^j\xi)|\les  \min\{|2^j\xi|,|2^j\xi|^{-1}\}$
 and $|\F^{-1}_{\R^n}(\rho_{0}(2^j\cdot)\hat{f})|\les M_{HL}f$,
 so we can achieve  (\ref{al1}) by Lemma \ref{ccz} (with $\mathfrak{M}_j=\rho_{0}(2^j\cdot)$).   Thus it remains to prove (\ref{al2}).
 Note that
$
 \|{\ind {\mathcal{S}_j^0}(x)} (\mu(x)2^{2dj})^l f_j\|_{\ell^2({j\in\N^B})}\les \|(f_j)_{j\in \N^B}\|_{\ell^\infty}
$
whenever $l\ge 1$.
 To achieve  (\ref{al2}),  it suffices to show
   \begin{align}
   \big\| \sup_{j\in\N^B}\sup_{\A\in \mathcal{A}_s}|\mathscr{L}_{s,\A}[\rho_{l}(2^j\cdot)](D)f|\big\|_{\ell^p(\Z^n)}
& \les\  C^l 2^{-c_p s} \|f\|_{\ell^p(\Z^n)},\ \ \ \ l\ge1.\label{aA1}
 \end{align}
 Let $\theta$ be the  function    as in Lemma \ref{de1}, and
  let  $\mathfrak{M}^{(2)}_j(\xi):=\rho_{l}(2^j\xi)-\rho_l(0)\widehat{\theta_j}(\xi)$.  Since
 $|\rho_l(0)|\les C^l$,
 to achieve  (\ref{aA1}),  it suffices to show
 \begin{align}
  \big\|\big(\sum_{j\in\N}|\sup_{\A\in\mathcal{A}_s}|\mathscr{L}_{s,\A}[\mathfrak{M}^{(2)}_j](D)f|^2\big)^{1/2}\big\|_{\ell^p(\Z^n)}\les&\  C^l 2^{-\gamma_ps} \|f\|_{\ell^p(\Z^n)}\ \ {\rm and}\label{UU0}\\
    \big\| \sup_{j\in\N}\sup_{\A\in \mathcal{A}_s}|\mathscr{L}_{s,\A}[\widehat{\theta_j}](D)f|\big\|_{\ell^p(\Z^n)}
   \les&\  2^{-\gamma_ps} \|f\|_{\ell^p(\Z^n)}\label{UU1}
 \end{align}
 with $\gamma_p$ given as in (\ref{de11}). We shall use  Lemma \ref{ccz} to obtain (\ref{UU0}).
 Simple computation gives
 $|\F^{-1}_{\R^n}(\mathfrak{M}^{(2)}_j)*_{\R^n}f|\les M_{HL}f$, which with the Fefferman-Stein inequality   yields that    $\mathfrak{M}^{(2)}_j$
 satisfies a vector-valued inequality  like (\ref{X3}). This   with
 $|\mathfrak{M}^{(2)}_j(\xi)|\les C^l\min\{2^j|\xi|, (2^j|\xi|)^{-1}\}$   gives  (\ref{UU0}) by Lemma \ref{ccz} (with $\mathfrak{M}_j=\mathfrak{M}^{(2)}_j$).
 In addition,   (\ref{UU1}) is a direct result of (\ref{yiny}). Thus    we complete the proofs of (\ref{UU0}) and (\ref{UU1}).

\subsection{Proof of Lemma \ref{l5.1}}
%\begin{proof}[Proof of Lemma \ref{l5.1}]
%We can not directly obtain an estimate like
%$$
%\| \sup_{j\in\N^b }\sup_{\A\in \mathcal{A}_s} \sup_{\mu \in I_{j,m}}
%|L_{j,\A+\mu, M}^s(D)f|\|_{\ell^p}\les\  \|f\|_{{\ell^p}}.
%$$ since the appearance of $\sup_j$.
%This is different from Krause-Roos's paper since
Since the value of $\e_\circ''$ is not important, we will omit  from the notation $L_{j,\la(x),\e_\circ''}^{(1),s}$.  Moreover, since the value of $\ka$
is not important, we will utilize  the notation (\ref{not32}) in the subsequent text.
To achieve  Lemma \ref{l5.1}, it suffices to prove
\begin{align}
({\rm Triple\  maximal\  estimate})\ \ \ \ \ \| \sup_{j\in\N^B }\sup_{\A\in \mathcal{A}_s} \sup_{\mu \in I_{j,m}}
|L_{j,\A+\mu}^{(1),s}(D)f|\|_{\ell^2(\Z^n)}\les&\  2^{-c(s+m)}\|f\|_{{\ell^2(\Z^n)}},\label{hee0}
\end{align}
for some $c>0$, the proof of which can be reduced  to proving that for every $\e\in(0,1)$,
\begin{align}
\| \sup_{j\in\N^B }\sup_{\A\in \mathcal{A}_s} \sup_{\mu \in I_{j,m}}
|L_{j,\A+\mu}^{(1),s}(D)f|\|_{\ell^2(\Z^n)}&\les\  2^{-c s}2^{\e m}\|f\|_{{\ell^2(\Z^n)}}\ \ \ \ \ \ {\rm and}\label{he0}\\
\| \sup_{j\in\N^B }\sup_{\A\in \mathcal{A}_s} \sup_{\mu \in I_{j,m}}
|L_{j,\A+\mu}^{(1),s}(D)f|\|_{\ell^2(\Z^n)}&\les_\e\  2^{(n+2+\e)s-c m}\|f\|_{{\ell^2(\Z^n)}}.\label{he1}
\end{align}
hold for some constant $c\in(0,1)$.
In fact,   letting  ${\eta_0}=c/(n+4)$, we obtain by (\ref{he0}) and  (\ref{he1}) that
$$\| \sup_{j\in\N^B}\sup_{\A\in \mathcal{A}_s} \sup_{\mu \in I_{j,m}}
|L_{j,\A+\mu}^{(1),s}(D)f|\|_{\ell^2(\Z^n)}\les_\e\  \big\{2^{-cs}2^{\e m}\big\}^{
1-\eta_0}\big\{2^{(n+2+\e)s-cm}\big\}^{\eta_0} \|f\|_{{\ell^2(\Z^n)}},$$
which leads to the desired result by setting $\e$ small enough such that
$\e(1-\eta_0)<c\eta_0$. The specific constant $n+2+\epsilon$ in (\ref{he1}) is not essential for the proof; it can be substituted with any arbitrary constant $C>n+2+\epsilon$.
\subsubsection{Proof of (\ref{he0})}
%f we use  $\int_{\Omega_0} e(t|y|^{2d})\tilde{K}_0(y)dy=0$ (say $t>0$) since the cancelation property of $K_0$ and the radial phase function $|y|^{2d}$,  the following proof can not be extended to some cases (In particular, the 1D case with phase function $|y|^{2d} $ replaced by some odd function like $y^3$).  Moreover,  $\int_{\Omega_0} e( t|y|^{2d})\tilde{K}_0(y)dy=0$ can lead to
%$$|\int_{\Omega_0} e( t |y|^{2d}+y\cdot2^j\xi)K_0(y)dy|\les |2^j\xi|$$
%whenever $|2^j\xi|\ll1$, which
%an important to bound some square function estimates. To avoid this procedure, we need to introduce discrete Littlewood-Paley inequality by using  a frequency decomposition first.
%\vskip.1in
Let us denote
\beq\label{noe1}
\psi_{j,k}(\xi):=\psi_{k-j}(\xi)=\psi(2^{j-k}\xi).
\eeq
Note that   for all $k\le 0$,
$$\|\big( \sum_{j\in \N^B}\sup_{\A\in \mathcal{A}_s}|\mathscr{L}_{s,\A}[\psi_{j,k}](D)f|^2\big)^{1/2}\|_{\ell^p(\Z^n)}
\les\  \|\big( \sum_{j\in \N^B}\sup_{\A\in \mathcal{A}_s}|\mathscr{L}_{s,\A}[\psi_{j}](D)f|^2\big)^{1/2}\|_{\ell^p(\Z^n)},$$
which with Lemma \ref{ccz} ($\mathfrak{M}_j=\psi_j$) gives
that
\beq\label{Da1}
\begin{aligned}
\|\big( \sum_{j\in \N^B}\sup_{\A\in \mathcal{A}_s}|\mathscr{L}_{s,\A}[\psi_{j,k}](D)f|^2\big)^{1/2}\|_{\ell^p(\Z^n)}
\les&\   2^{-\gamma_p s}\|f\|_{\ell^p(\Z^n)}
\end{aligned}
\eeq
with $\gamma_p$ given as in (\ref{de11}).
This estimate will be used  in the following arguments.  Write
\beq\label{da1}
\sup_{\A\in \mathcal{A}_s} \sup_{\mu \in I_{j,m}}
|L_{j,\A+\mu}^{(1),s}(D)f|=\sup_{\A\in \mathcal{A}_s} \sup_{1\le |t|<2}
|L_{j,\A+2^{m-2dj}t}^{(1),s}(D)f|.
\eeq
Without loss of generality, we assume $t\in[1,2)$ in (\ref{da1}) since
$t\in (-2,-1]$ can be handled similarly.
Thus, by
  the partition of unity
$
\chi(2^j\xi)+\sum_{k\ge 1}\psi_{j,k}(\xi)=1$,
%with $C_2$ large enough,
we can bound (\ref{da1}) by
\beq\label{bbb2}
\begin{aligned}
&\sup_{\A\in \mathcal{A}_s} \sup_{t \in  [1,2)}|\mathscr{L}_{s,\A}[\phi^{(1)}_{j,2^{m-2dj}t} ~\chi(2^j\cdot)](D)f|\\
+& \sum_{k\ge 1}\sup_{\A\in \mathcal{A}_s} \sup_{t \in  [1,2)}|\mathscr{L}_{s,\A}[\phi^{(1)}_{j,2^{m-2dj}t} ~\psi_{j,k}](D)f|.
\end{aligned}
\eeq
%where %$C_1$ fixed later is sufficiently large,
 %$\rho_j$ is given by
%\beq\label{g1}
%\rho_j(\xi)=\sum_{k\le -m}\psi_{m,j,k}(\xi)=\sum_{l\le 0}\psi(2^{j-l}\xi).
%\eeq
%Simple computations yield that $\rho_j$ satisfies
%\beq\label{1wq}
%|\rho_j(\xi)|\les \min\{2^j|\xi|,(2^j|\xi|)^{-1}\}.
%\eeq
By
 using (\ref{bbb2}), to prove  (\ref{he0}),
it suffices to show  that for any $\e\in (0,1)$,
\begin{align}
\| \sup_{j\in\N^b }\sup_{\A\in \mathcal{A}_s} \sup_{t \in  [1,2)}|\mathscr{L}_{s,\A}[\phi_{j,2^{m-2dj}t}^{(1)}~\chi(2^j\cdot)](D)f|\|_{\ell^2(\Z^n)}&\les\  2^{-cs}\|f\|_{{\ell^2(\Z^n)}} \ \ \ \ {\rm and}\label{ge1}\\
\|\sup_{j\in\N^b } \sup_{\A\in \mathcal{A}_s} \sup_{t \in  [1,2)}|\mathscr{L}_{s,\A}[\phi_{j,2^{m-2dj}t}^{(1)}~\psi_{j,k}](D)f|\|_{\ell^2(\Z^n)}&\les\  2^{\e m-cs-\e k}\|f\|_{{\ell^2(\Z^n)}}\     (k\ge 1).\label{ge2}
%\|\sup_{j\in\N^b } \sup_{\A\in \mathcal{A}_s} \sup_{t \in  [1,2)}|\mathscr{L}_{s,\A}[\phi_{j,2^{m-2dj}t}^{(1)}~\psi_{m,j,k}](D)f|\|_{\ell^p}\les&\  2^{-cs}\|f\|_{{\ell^p}},\     \ \ -m\le k<C_2.\label{ge3}
\end{align}
%for some $c>0$.

We first show (\ref{ge1}). By writing  $\chi(2^j\xi)$ as $\chi(2^j\xi)=\chi(2^j\xi)-\widehat{\theta_j}(\xi)+\widehat{\theta_j}(\xi)
$
with $\theta_j$ given as in Lemma \ref{de1},
and
repeating the arguments yielding (\ref{aA1}), we have
 %(\ref{UU0}) and (\ref{UU1}),
%there is $c_p>0$ such that
\begin{align}
\|\sup_{j\in\N^B}\sup_{\A\in \mathcal{A}_s}|\mathscr{L}_{s,\A}[\chi(2^j\cdot)](D)f|\|_{\ell^p(\Z^n)}\les&\  2^{-\gamma_p s} \|f\|_{\ell^p(\Z^n)}\  \ \ (1<p<\infty). \label{5d3}
%\|\sup_{j\in\N^B}\sup_{\A\in \mathcal{A}_s} |\mathscr{L}_{s,\A,M}[\widehat{\K_j}](D)f|\|_{\ell^p}\les&\  2^{-c_p s} \|f\|_{\ell^p}.\label{5d4}
\end{align}
Since  $t\in [1,2)$ and $|\na |y|^{2d}|\gtrsim1$ whenever $|y|\sim1$,  we  obtain by  a routine computation that\footnote{To adapt our proof for the one-dimensional case with general phase  $y^m$ for all $m\ge3$, we  don't rely on the condition $\Xi_{m,t}=0$.}
\beq\label{op1}
 | \Xi_{m,t}| \les 2^{-m}
\eeq
where $\Xi_{m,t}$ is given by 
$$\Xi_{m,t}:=\int_{1/2\le|y|\le 1} e( 2^m t |y|^{2d}){K}_0(y)dy.$$
Since   $\Xi_{m,t}$ only depends on $m,t$,  we deduce by   (\ref{5d3})  (with $p=2$) and (\ref{op1}) that
\begin{align}
\|\sup_{j\in\N^B}\sup_{\A\in \mathcal{A}_s}\sup_{t \in  [1,2)}|\mathscr{L}_{s,\A}[\Xi_{m,t}~\chi(2^j\cdot)](D)f|\|_{\ell^2(\Z^n)}\les&\  2^{-m} 2^{-c s} \|f\|_{\ell^2(\Z^n)}.\label{5d4}
\end{align}
To prove (\ref{ge1}),  by (\ref{5d4}), it suffices to show
\beq\label{bg1}
 \|\sup_{j\in\N^B}\sup_{\A\in \mathcal{A}_s} \sup_{t \in  [1,2)}|\mathscr{L}_{s,\A}[h_{m,j,t}~ \chi(2^j\cdot)](D)f|\|_{\ell^2(\Z^n)}
\les\ 2^{-c s}\|f\|_{{\ell^2(\Z^n)}},
\eeq
where $h_{m,j,t}$ is given by
$$h_{m,j,t}(\xi):=\phi_{j,2^{m-2dj}t}^{(1)}(\xi)-\Xi_{m,t}=\int_{1/2\le|y|\le 1} e( 2^m t |y|^{2d})\big(e(2^j\xi\cdot y)-1\big){K}_0(y)dy.$$
  Since  $h_{m,j,t}(0)=0$, we may replace  $\chi(2^j\xi)$ by $\sum_{k\le 0}\psi_{j,k}(\xi)$. Thus, to achieve  (\ref{bg1}), it suffices   to  prove
\begin{align}
\| \sup_{j\in\N^B }\sup_{\A\in \mathcal{A}_s} \sup_{t \in  [1,2)}|\mathscr{L}_{s,\A}[h_{m,j,t}~\psi_{j,k}](D)f|\|_{\ell^2(\Z^n)}\les&\  2^{k}2^{-c s}\|f\|_{{\ell^2(\Z^n)}}\label{we1}
\end{align}
for every $k\le 0$.
Using Taylor's expansion, we have
\beq\label{g2}
\begin{aligned}
h_{m,j,t}(\xi)\psi_{j,k}(\xi)
=\ 2^{k}\sum_{l=1}^\infty  2^{k(l-1)}\frac{(2\pi i)^l}{l!}\psi_{j,k}(\xi)
h_{m,j,t}^{k,l}(\xi),
%=&\ 2^{m+k}\sum_{l=1}^\infty  2^{(m+k)(l-1)}\frac{(2\pi i)^l}{l!}
%\sum_{v\in S_l} \Xi_{m,t,l}~ \psi_{m,j,k,l}(\xi).
\end{aligned}
\eeq
where $h_{m,j,t}^{k,l}(\xi)=\int_{1/2\le|y|\le 1} e(2^m t|y|^{2d})
\big(y\cdot \frac{\xi}{2^{k-j}}\big)^l {K}_0(y)dy$.
Expanding the term $\big(y\cdot \frac{\xi}{2^{k-j}}\big)^l$ in the expression for $h_{m,j,t}^{k,l}$, we can interpret $\psi_{j,k}(\xi) h_{m,j,t}^{k,l}(\xi)$ as a sum of  $\mathcal{O}(n^l)$ terms resembling $\Xi_{m,t,l}~ \psi_{j,k,l}(\xi)$, where $\Xi_{m,t,l}$ and $\psi_{j,k,l}$ represent variations of $\Xi_{m,t}$ and $\psi_{j,k}$, respectively. Precisely,  we have
\beq\label{op2}
 \|\big( \sum_{j\in \N^B}\sup_{\A\in \mathcal{A}_s}|\mathscr{L}_{s,\A}[\psi_{j,k,l}](D)f|^2\big)^{1/2}\|_{\ell^2(\Z^n)}
\les\   C^l2^{-c  s}\|f\|_{\ell^2(\Z^n)}\ \ \ {\rm and}\
\ \ | \Xi_{m,t,l}|\les C^l 2^{-m},
\eeq
which are similar to (\ref{Da1}) and (\ref{op1}), respectively.
 In order to prove (\ref{we1}),  the above arguments imply that  it suffices to show  that, for each $k\le 0$ and each $l\ge 1$,
\beq\label{o1}
\| \sup_{j\in\N^b }\sup_{\A\in \mathcal{A}_s} \sup_{t \in  [1,2)}|\mathscr{L}_{s,\A}[ \Xi_{m,t,l} ~\psi_{j,k,l}](D)f|\|_{\ell^2(\Z^n)}\les 2^{-cs} C^l\|f\|_{{\ell^2(\Z^n)}}.
\eeq
In fact, (\ref{o1})  is  a direct result of   (\ref{op2}).
%by using  that $\Xi_{m,t,l}$ is a constant in $\xi$ and $x$.  
This ends the proof of (\ref{ge1}).
%This is a direct result of Lemma \ref{ccz} by changing variables $j\to j+m+k$.
%Denote by $  J_{s,m,p}$ the left hand side of the above.

   We next show (\ref{ge2}).
   Since  the support of $\psi_{j,k}$ yields that $2^j\xi$ may be large enough,
   %the function $\phi_{j,2^{m-2dj}t}^{(1)}$ depends on $t$,
   %it seems hard to remove the %supreme over $t\in [1,2)$,
     the proof of (\ref{ge2}) is   more involved.
     %Let $\{\va_i(t)\}_{i=0}^\infty$ be the sequence of Rademacher functions on $[0,1]$.
By  linearization, the square of  the left-hand side of (\ref{ge2}) is bounded by
\beq\label{x1}
   \begin{aligned}
   \int_0^1 \|\sup_{\A \in \mathcal{A}_s} \sup_{t \in  [1,2)}|\mathscr{L}_{s,\A}[\Phi_{m,k}^{t,\tau}](D)f|
   \|_{\ell^2(\Z^n)}^2
   d\tau,
   \end{aligned}
   \eeq
   where  $\Phi_{m,k}^{t,\tau}$ is  given by
   \beq\label{k100}
   \Phi_{m,k}^{t,\tau}(\xi):=\sum_{j\in\N^B} \va_j(\tau) ~\phi_{j,2^{m-2dj}t}^{(1)}(\xi)~ \psi_{j,k}(\xi)
   \eeq
   with   $\{\va_j(\tau)\}_{j=0}^\infty$    the sequence of Rademacher functions  on $[0,1]$.
    We will use Sobolev inequality to control the norm involving the supremum on $t$. Let us denote
         $$\tilde{\Phi}_{m,k}^{t,\tau}(\xi):=2^{-m}\frac{\p}{\p_t}\Phi_{m,k}^{t,\tau}(\xi)=\sum_j \va_j(\tau)~ \tilde{\phi}_{j,2^{m-2dj}t}^{(1)}(\xi)~ \psi_{j,k}(\xi)$$
         with   $\tilde{\phi}_{j,2^{m-2dj}t}^{(1)}$ given by
         $$\tilde{\phi}_{j,2^{m-2dj}t}^{(1)}(\xi):=2^{-m} \frac{\p}{\p t}\phi_{j,2^{m-2dj}t}^{(1)}(\xi)=2\pi i \int_{1/2\le |y|\le 1} e(2^mt |y|^{2d}+2^jy\cdot\xi) |y|^{2d}{K}_0(y)dy\ \ $$
   Using  the interpolation inequality,
   we have
   \beq\label{x2}
   \begin{aligned}
 \sup_{t\in [1,2)}|\mathscr{L}_{s,\A}&[\Phi_{m,k}^{t,\tau}](D)f|^2
   \les\  |\mathscr{L}_{s,\A}[\Phi_{m,k}^{1,\tau}](D)f|^2\\
  &\  +2^m \|\mathscr{L}_{s,\A}[\Phi_{m,k}^{t,\tau}](D)f\|_{L^2_t( [1,2))}
\
     \|\mathscr{L}_{s,\A}[\tilde{\Phi}_{m,k}^{t,\tau}](D)f\|_{L^2_t( [1,2))}.
     \end{aligned}
   \eeq
          To prove (\ref{ge2}),  by  (\ref{x2}),
      it suffices to show  that for all $(t,\tau)\in [1,2)\times [0,1]$    and  each $H\in \{\Phi,\tilde{\Phi}\}$,
   \begin{align}
 \|\sup_{\A\in \mathcal{A}_s}     |\mathscr{L}_{s,\A}[H_{m,k}^{t,\tau}](D)f|\|_{\ell^2(\Z^n)}\les&\   \ 2^{-cs} 2^{-\e k/(2d)}2^{-(1-\e)m/2}   \|f\|_{\ell^2(\Z^n)} \ \   \  (k\ge 1). \label{m31}
   \end{align}
   %Indeed, setting $\e<1/p$, we can get the desired result.
   We only show the details for the case $H=\Phi$ since the case $H=\tilde{\Phi}$ can be bounded  similarly.
   Using   (\ref{de11}) and (\ref{CC1}),  to obtain  (\ref{m31}), it suffices to show
    \begin{align}
\|\mathscr{L}_{s}^\#[\Phi_{m,k}^{t,\tau}](D)f\|_{\ell^2(\Z^n)}\les&\  2^{-\e k/(2d)}2^{-(1-\e)m/2}  \|f\|_{\ell^2(\Z^n)}\ \   \  (k\ge 1).\label{m211}
   \end{align}
 By Proposition \ref{PIW} and the Littlewood-Paley theory,   we reduce the proof of
 (\ref{m211}) to showing
   \begin{align}
\|\big(\sum_{j\in\N^B}|(\phi_{j,2^{m-2dj}t}^{(1)}~\psi_{j,k})(D)f|^2\big)^{1/2}\|_{L^2(\R^n)}\les&\  2^{-\e k/(2d)}2^{-(1-\e)m/2} \|f\|_{L^2(\R^n)},\ \   \  k\ge 1. \label{m212}
 %\|\big(\sum_{j\in\N^B}|(\phi_{j,2^{m-2dj}t}^{(1)}~\psi_{m,j,k})(D)f|^2\big)^{1/2}\|_p\les&\   \    \|f\|_p, \ \  \ -m\le k<C.\label{m312}
   \end{align}
   %For all $k\ge -m$,  by
   %using a crude estimate $|(\phi_{j,2^{m-2dj}t}^{(1)}~\psi_{m,j,k})(D)f|
  % \les |M_{HL}P_{m-j+k}f|$,  Fefferman-Stein and Littlewood-Paley inequalities, we have for each $p\in(1,\infty)$, LHS of (\ref{m212}) is bounded by a constant times
 %  $$% \|\big(\sum_{j\in\N^b}|(\phi_{j,2^{m-2dj}t}\psi_{m,j,k})(D)f|^2\big)^{1/2}\|_p
 %   \|\big(\sum_{j\in\N^B} |M_{HL}P_{m-j+k}f|^2\big)^{1/2}\|_p
 %  \les  \|\big(\sum_{j\in\N^B} |P_{m-j+k}f|^2\big)^{1/2}\|_p\les \|f\|_p.$$
 %  To finish the proof of (\ref{m212}) and (\ref{m312}), by interpolation,  it remains to show  (\ref{m212}) for $p=2$.
%  Setting $C$ large enough, and using  $k\ge C$,   we infer by integration by parts that
 %  $|\phi^{(1)}_{j,2^{m-2dj}t}(\xi)|
 %  \les 2^{-(m+k)}.$
Using the polar coordinate and Van der Corput lemma (see \cite{St93}), we can get $|{\phi}_{j,2^{m-2dj}t}^{(1)}(\xi)|\les 2^{-m/2}$; on the other hand,  we can also obtain  $|{\phi}_{j,2^{m-2dj}t}^{(1)}(\xi)|\les (2^j|\xi|)^{-1/{2d}}$ by Proposition 2.1 in \cite{SW01}.  Thus
\begin{align*}
|{\phi}_{j,2^{m-2dj}t}^{(1)}(\xi)|\les \min\{2^{-m/2}, 2^{-k/{2d}}\}\ \ {\rm whenever}\ 2^j|\xi|\sim 2^k.
\end{align*}
By this estimate  and    Plancherel's identity,   the left-hand side of (\ref{m212}) is bounded by
     $$
     \big(\sum_{j\in\N^B}\|\phi_{j,2^{m-2dj}t}^{(1)}(\xi)~\psi_{j,k}(\xi)~\widehat{f}(\xi)\|_{L^2_\xi}^2\big)^{1/2}
     \les 2^{-\e k/(2d)}2^{-(1-\e)m/2} \|f\|_{L^2(\R^n)}
     $$
    for  any $\e\in [0,1]$,
which yields (\ref{m211}) immediately.  This completes the proof of (\ref{ge2}).
     %%%
\subsubsection{Proof of (\ref{he1})}
We show (\ref{he1})  for all $p\in (1,\infty)$.
Denote
$$\mathcal{Y}_s:=\{{b}/{q}:\ b\in \Z^n \cap [0,q]^n, \  q\in [2^{s-1},2^s)\}$$
satisfying  $\# \mathcal{Y}_s\les 2^{(n+1)s}$.
%we have used (\ref{new1}) and $\# \mathcal{Y}_s\les 2^{(n+1)s}$.
Then, by (\ref{MI11}) and (\ref{new1}), the left hand side of (\ref{he1}) is
\beq\label{DD01}
\begin{aligned}
 &\les\ \sum_{\A\in \mathcal{A}_s}\sum_{\bb\in \mathcal{Y}_s}
\|\sup_{j\in\Z}\sup_{\mu\in I_{j,m}}|\mathcal{F}^{-1}_{\R^n}(\phi_{j,\mu}^{(1)}~\chi_{s,\kappa})*_{\Z^n} (\NE_{-\bb} f)|\|_{\ell^2(\Z^n)}\\
&\les_\e  2^{(n+2+\e)s}  \sup_{\bb\in \mathcal{Y}_s}\|\sup_{j\in\Z}\sup_{\mu\in I_{j,m}}|\mathcal{F}^{-1}_{\R^n}(\phi_{j,\mu}^{(1)}~\chi_{s,\kappa})*_{\Z^n} (\NE_{-\bb} f)|\|_{\ell^2(\Z^n)}.
\end{aligned}
\eeq
%with $(\NE_{-\bb} f)(y)=e(-\bb\cdot y)f(y)$,
By the Stein-Wainger-type estimate \footnote{
Since Propositions 2.1 and 2.2 in \cite{SW01}  work as well,
we only need to  repeat   the arguments yielding Theorem 1 in \cite{SW01}
to obtain this estimate (\ref{ASW1}).}, we have
\beq\label{ASW1}
\|\sup_{j\in\Z}\sup_{\mu\in I_{j,m}}|\mathcal{F}^{-1}_{\R^n}(\phi_{j,\mu}^{(1)}~\chi_{s,\kappa})*_{\R^n} f|\|_{L^2(\R^n)}\les 2^{-cm} \|f\|_{L^2(\R^n)},
\eeq
which with the
 transference principle gives
\beq\label{sa1}
\|\sup_{j\in\Z}\sup_{\mu\in I_{j,m}}|\mathcal{F}^{-1}_{\R^n}(\phi_{j,\mu}^{(1)}~\chi_{s,\kappa})*_{\Z^n} f|\|_{\ell^2(\Z^n)}
\les 2^{-cm}  \|f\|_{\ell^2(\Z^n)}.
\eeq
Note $\|\NE_{-\bb} f\|_{\ell^2(\Z^n)}=\|f\|_{\ell^2(\Z^n)}$.
Then (\ref{he1}) follows by inserting (\ref{sa1}) into (\ref{DD01}).
\section{Major arcs estimate II: Proof of Proposition \ref{892}}
\label{slong3}
In this section, we shall prove  major arcs estimate II in Proposition \ref{892} by employing  the crucial  multi-frequency variational inequality in Lemma \ref{endle} and the techniques proving  major arcs estimate I. Keep the notation (\ref{zhong}) in mind.
 %Rememeber
%that  $L_{l,\la(x)}^{(2),s}$ is   defined by (\ref{f2}).
  %and $\la(x)$  denotes an arbitrary function from $\Z^n$ to  $[0,1]$.
  %Let $\e_\circ$ be a sufficiently small constant, and let  $\la$  denote an arbitrary function from $\Z^n$ to  $[0,1]$.
 \begin{lemma}\label{l6.1}
Let $r\in(2,\infty)$,  $p\in [p_1,p_2]$ and $m\in [1,\infty)$. Then there is a constant    $c_p>0$ such that
\beq\label{A11}
\big\|\big(\sum_{C_0 \le l< j} {\ind {\mathcal{S}_l^m}}(x) \sum_{1\le s\le \e_\circ (l)}
[L^{(2),s}_{l,\la(x),\e_\circ}(D)f](x)\big)_{j\in\N^B}\big\|_{\ell^p(x\in\Z^n; V^r)}
\les 2^{-c_p m}\|f\|_{\ell^p(\Z^n)},
\eeq
where $\e_\circ=\e_\circ(p_1,p_2,\mathcal{C})$, and $L_{l,\la(x),{\e}_\circ}^{(2),s}$ is given by (\ref{f2}).
\end{lemma}
\begin{lemma}\label{l6.2}
Let
 $R\in [1,\infty)$,  $r\in(2,\infty)$ and  $p\in (1,\infty)$. For any $\e>0$ and every $\bar{\e}_\circ\in (0,1)$,  we have
\beq\label{s21}
\big\|\big(\sum_{C_0 \le l< j} {\ind {\mathcal{S}_l^0}}(x) \sum_{1\le s\le \bar{\e}_\circ (l)}[L_{l,\la(x),\bar{\e}_\circ}^{(2),s}(D)f](x)\big)_{j\in\N^B}\big\|_{\ell^p(x\in \B_R; V^r)}
\les_\e R^\e \|f\|_{\ell^p(\Z^n)},
\eeq
where $L_{l,\la(x),\bar{\e}_\circ}^{(2),s}$ is given by (\ref{f2}) with ${\e}_\circ=\bar{\e}_\circ$.
\end{lemma}
 Keep  (\ref{dc1}) and (\ref{le1}) in mind. By  the equality in (\ref{le1}) with $j$ replaced by $l$,  Proposition \ref{892} is a direct consequence of
  the above two lemmas.
In the remainder of this section, we shall prove  Lemma \ref{l6.1} and Lemma \ref{l6.2} in order.
\subsection{Proof of Lemma \ref{l6.1}}
\label{sgd1}
To prove (\ref{A11}),  by interpolation,  it suffices to show the following: for every $\e_\circ'\in (0,1)$,
\beq\label{jj1}
\big\|\big(\sum_{C_0 \le l< j} {\ind {\mathcal{S}_l^m}}(x) \sum_{1\le s\le \e_\circ' (l)}[L_{l,\la(x),\e_\circ'}^{(2),s}(D)f](x)\big)_{j\in\N^B}\big\|_{\ell^2(x\in\Z^n; V^r)}
\les 2^{-c m}\|f\|_{\ell^2(\Z^n)}
\eeq
for some $c>0$;
and  for every  $p\in(1,\infty)$,
\beq\label{jj2}
\big\|\big(\sum_{C_0 \le l< j} {\ind {\mathcal{S}_l^m}}(x) \sum_{1\le s\le \e_\circ (l)}[L_{l,\la(x),\e_\circ}^{(2),s}(D)f](x)\big)_{j\in\N^B}\big\|_{\ell^p(x\in\Z^n; V^r)}
\les \|f\|_{\ell^p(\Z^n)},
\eeq
where $\e_\circ=\e_\circ(p_1,p_2,\mathcal{C})$.
We first  prove  (\ref{jj1}).
Define
\beq\label{not2}
v_s=v_s(\e_\circ'):=\max\{C_0, 2^{{\lfloor 1/\e_\circ' \rfloor}^{-1}s}\}.
\eeq
By Minkowski's inequality,
it is easy to see that
  (\ref{jj1})  follows from
\beq\label{jk1}
\big\|\big(\sum_{v_s \le l\le j} {\ind {\mathcal{S}_l^m}}(x)
{\ind {\{1\le s\le \e_\circ' (l)\}}}~
 [L_{l,\la(x),\e_\circ'}^{(2),s}(D)f](x)\big)_{j\in\N^B}\big\|_{\ell^2(x\in\Z^n; V^1)}
\les 2^{-c(s+m)}\|f\|_{\ell^2(\Z^n)}.
\eeq
%It thus  remains to prove (\ref{jk1}).
By
 the inequality in (\ref{le1}) (with $j=l$),    the left-hand side of (\ref{jk1}) is
 \beq\label{A01}
\les  \|\sup_{l\ge v_s}\sup_{\A\in\mathcal{A}_s}\sup_{\mu\in I_{l,m}}
|L_{l,\A+\mu,\e_\circ'}^{(2),s}(D)f|\|_{\ell^2(\Z^n)}.
 \eeq
 In addition,
 by performing the arguments yielding (\ref{hee0}),  we may infer
 \beq\label{A00}
\|\sup_{l\ge v_s}\sup_{\A\in\mathcal{A}_s}\sup_{\mu\in I_{l,m}}
|L_{l,\A+\mu,\e_\circ'}^{(2),s}(D)f|\|_{\ell^2(\Z^n)}\les 2^{-c(s+m)}\|f\|_{\ell^2(\Z^n)}.
\eeq
As a result, the desired (\ref{jk1}) follows by combining  (\ref{A00}) with (\ref{A01}).
This finishes the proof of (\ref{jj1}).

For the proof of   (\ref{jj2}),
it suffices to show  that for every $ p\in(1,\infty)$,
\beq\label{jj21}
\big\|{\ind {\mathcal{S}_l^m}}(x) \sum_{1\le s\le \e_\circ (l)}[L_{l,\la(x),\e_\circ}^{(2),s}(D)f](x)\big\|_{\ell^p(x\in\Z^n; \ell^1(l\in\N^B))}
\les \|f\|_{\ell^p(\Z^n)}.
\eeq
We next utilize  the first trick mentioned  in Subsection \ref{diff}. 
Using  (\ref{ds1}), %and denoting $\Lambda_{l,\la,m}:={\Lambda_{l,\la}}\cap {\mathcal{S}_l^m}$,  
we have
\beq\label{jj41}
{\ind {\mathcal{S}_l^m}}(x) \sum_{1\le s\le \e_\circ (l)}L_{l,\la(x),\e_\circ}^{(2),s}(\xi)= {\ind {\Lambda_{l,\e_\circ,\la,m}}}(x) \mb_{l,\la(x)}(\xi)-   {\ind {\Lambda_{l,\e_\circ,\la,m}}}(x)  E^{(2)}_{l,\la(x),\e_\circ}(\xi),
\eeq
where the set $\Lambda_{l,\e_\circ,\la,m}$ is given by (\ref{NNM1}) with $j=l$.
By  using the inequality in (\ref{le1}),  similar   arguments as  yielding (\ref{bc1}), 
and the estimate  (\ref{100}),
we obtain that  for every $p\in(1,\infty)$,
\begin{align}
\big\|{\ind {\mathcal{S}_l^m}}(x) \big(\mb_{l,\la(x)}(D)f\big)(x)\big\|_{\ell^p(x\in\Z^n; \ell^1(l\in\N^B))}
\les&\  \|\sup_{l\in\N^B}\sup_{\la\in [0,1]}|\mb_{l,\la}(D)f|\|_{\ell^p(\Z^n)}\les  \|f\|_{\ell^p(\Z^n)},
\label{jj22}\\
\big\|{\ind {\Lambda_{l,\e_\circ,\la,m}}}(x)  \big(E^{(2)}_{l,\la(x),\e_\circ}(D)f\big)(x)\big\|_{\ell^p(x\in\Z^n; \ell^1(l\in\N^B))}
\les&\   \sum_{l\in\N^B}\|\sup_{\la\in X_{l,\e_\circ}}|E^{(2)}_{l,\la,\e_\circ}(D)f|\|_{\ell^p(\Z^n)}
\les \|f\|_{\ell^p(\Z^n)}.\label{jj23}
\end{align}
Finally, the desired (\ref{jj21}) follows from the combination of
 (\ref{jj41}), (\ref{jj22}) and (\ref{jj23}).

\subsection{Proof of Lemma \ref{l6.2}}
\label{sgd2}
Since the value of $\bar{\e}_\circ$ is not crucial, we omit it from the notation when it doesn’t impact the clarity of the context. In addition, since the value of $\ka$
is not important, we will use the notation (\ref{not32}).
Recall $\phi_{l,\mu(x)}^{(2)}(\xi)=\int_{\R^n} e\big(\mu(x)|y|^{2d}+y\cdot\xi\big)~{K}_l(y) dy.$
%We first have
%$$\LL_{j,\la(x),M}^s (\xi)=\mathscr{L}_{s,\A,M}[V_{j,\la(x)-\A,M}^*](\xi)$$
Taylor expansion gives
$$
\begin{aligned}
\phi^{(2)}_{j,\mu(x)}(\xi)=&\ \sum_{k=0}^\infty  \frac{ (2\pi i )^k}{k!}  (2^{2dl}\mu(x))^k
\int
e(y\cdot 2^l\xi) |y|^{2dk} K_0(y)dy\\
=&\ \sum_{k=0}^\infty  \frac{ (2\pi i )^k}{k!}  (2^{2dl}\mu(x))^k \widehat{K_{0,k}}(2^l\xi),
\end{aligned}
$$
where $K_{0,k}(y)=|y|^{2dk} K_0(y)$. Let  $$\phi^{\circ,k}(\frac{\mu(x)}{2^{-2dl}}):=
{\ind {\mathcal{S}_l^0}}(x)(2^{2dl}\mu(x))^k\ \ \ (k\in\N_0),\ \ \ {\rm and}\  \ \ \mathscr{U}_l':=[1, \bar{\e}_\circ (l)]\cap \Z.$$
To achieve  (\ref{s21}),  it suffices to show that there exists a constant  $c_p>0$ such that for every $k\ge 0$,
\beq\label{aas221}
\|\big(\sum_{v_s \le l< j}
\phi^{\circ,k}(\frac{\mu(x)}{2^{-2dl}}) {\ind {{s\in \mathscr{U}_l'}}}~
\mathscr{L}_{s,\A} [\widehat{K_{0,k}}(2^l\cdot)](D)f
\big)_{j\in\N^B}\|_{\ell^p(x\in \B_R;\V^r)}
\les_\e  C^k 2^{-c_ps } R^\e\|f\|_{{\ell^p}(\Z^n)}.
\eeq
For the case $k\ge 1$,
by  $\sum_{l\in\Z}|\phi^{\circ,k}(\frac{\mu(x)}{2^{-2dl}})|\les1$,  the  $\V^r$ norm on the left-hand side of (\ref{aas221}) is
\beq\label{LP1}
\les \sum_{l\ge v_s}
|\phi^{\circ,k}(\frac{\mu(x)}{2^{-2dl}})|
\sup_{\A\in \mathcal{A}_s}|\mathscr{L}_{s,\A} [\widehat{K_{0,k}}(2^l\cdot)](D)f|\les \sup_{l\ge v_s} \sup_{\A\in \mathcal{A}_s}|\mathscr{L}_{s,\A} [\widehat{K_{0,k}}(2^l\cdot)](D)f|.
\eeq
 Hence,  to show (\ref{aas221}), by (\ref{LP1}) and
 the equality in (\ref{le1}), it suffices to prove
\begin{align}
\|\sup_{ l\in \N^B}\sup_{\A\in\mathcal{A}_s}|
\mathscr{L}_{s,\A} [\widehat{K_{0,k}}(2^l\cdot)](D)f|
\|_{\ell^p(\Z^n)}
&\les\  C^k 2^{-c_p s }\|f\|_{{\ell^p(\Z^n)}}\ \ (k\ge 1)~~{\rm and}\label{A1}\\
\|\big(\sum_{v_s  \le l< j}
{\ind {(x,s)\in \mathscr{U}_l }}
\mathscr{L}_{s,\A} [\widehat{K_{0}}(2^l\cdot)](D)f
\big)_{j\in\N^B}\|_{\ell^p(x\in \B_R;\V^r)}
&\les_\e   2^{-c_ps } R^\e\|f\|_{{\ell^p(\Z^n)}},\label{A2}
\end{align}
where $ \mathscr{U}_l :=\mathcal{S}_l^0\times  \mathscr{U}_l'$.
%where $K_{0}=K_{0,0}$ is used.
We next prove (\ref{A1}) and (\ref{A2}) in order.
%In fact, the proof of (\ref{A2}) is more complex than that of (\ref{A1}).
\subsubsection{Proof of (\ref{A1}) }
Using the partition of unity $\chi(2^l \xi)+\sum_{J\ge1}\psi(2^{l-J}\xi)=1$, we have
$$
\begin{aligned}
\widehat{K_{0,k}}(2^l\xi)
&=\chi(2^{l}\xi)+\m^{(1)}_{l,k}(\xi)+\sum_{J\ge 1}\m^{(2),J}_{l,k}(\xi),
\end{aligned}$$
where
$$
\m^{(1)}_{l,k}(\xi):=\big(\widehat{K_{0,k}}
(2^l\xi)-1\big)\chi(2^{l}\xi),\ \ \
\m^{(2),J}_{l,k}(\xi):=
\widehat{K_{0,k}}(2^l\xi)\psi(2^{l-J}\xi).
$$
By  (\ref{5d3}),  to prove  (\ref{A1}), it suffices to prove that there is a constant  $c_p>0$ such that for each $k\ge 1$,
\begin{align}
\|\big(\sum_{ l\ge v_s}\sup_{\A\in\mathcal{A}_s}|
\mathscr{L}_{s,\A} [\m^{(1)}_{l,k}](D)f|\big)^{1/2}
\|_{\ell^p(\Z^n)}
\les&\  C^k 2^{-c_ps }\|f\|_{{\ell^p(\Z^n)}}\ \ \ {\rm and}\label{F1}\\
\|\big(\sum_{ l\ge v_s}\sup_{\A\in\mathcal{A}_s}|
\mathscr{L}_{s,\A} [\m^{(2),J}_{l,k}](D)f|^2\big)^{1/2}
\|_{\ell^p(\Z^n)}
\les&\  C^k 2^{-c_pJ}2^{-c_ps }\|f\|_{{\ell^p(\Z^n)}}.\label{F2}
\end{align}
We shall use Lemma \ref{ccz} to prove  (\ref{F1}) and (\ref{F2}).
For $x\in\R^n$,  we have
$|\m^{(1)}_{l,k}(D)f|(x)+|\m^{(2),J}_{l,k}(D)f|(x)\les C^k M_{HL}f(x),$ and then
  $\m^{(1)}_{l,k}$ and $\m^{(2),J}_{l,k}$  satisfy some  vector-valued inequalities  like (\ref{X3}) by the Fefferman-Stein inequality. Moreover,  for $\xi\in\R^n$, we can infer
   $|\m^{(1)}_{l,k}(\xi)|\les C^k \min\{2^l|\xi|,(2^l|\xi|)^{-1}\}$  and
  $|\m^{(2),J}_{l,k}(\xi)|\les C^k \min\{1,(2^{l}|\xi|)^{-1}\}\les C^k 2^{-J/2}\min\{(2^{l}|\xi|)^{1/2},(2^{l}|\xi|)^{-1/2}\}$ (since $2^l|\xi|\sim 2^J$ at this moment).
Therefore,    we can achieve  (\ref{F1}) and (\ref{F2}) by Lemma \ref{ccz}.
\subsubsection{Proof of (\ref{A2}) }
%Next, we consider  (\ref{A2}).
Let $\A(x)$ be an arbitrary function from $\Z^n$ to $\mathcal{A}_s$,
%Since $x\in \B_R$,
 %we can assume that
 %we can assume that   $\la(x)$ is   a $2R$-periodic (in every coordinate) function.  Then $\mu(x)$  is  a $2R$-periodic (in every coordinate) function as well (since the left-hand side  of  (\ref{A2}) equals 0 otherwise). Thus $\A=\la(x)-\mu(x)=:\A(x)$ is $2R$-periodic in every coordinate.
 %Without loss the generality,
 %we assume there is a unique $\A=\A(x)\in \mathcal{A}_s$ such that  $\mu(x)$ is also $2R$-periodic
%$|\mu(x)|\le 2^{-2s}$ since the left-hand side  of  (\ref{A2}) equals 0 otherwise. Then
%Since  $\la(x)$  is a $2R$ periodic function,
% we can see
%$\mu(x)$ is also $2R$-periodic (in every coordinate).  Thus $\A=\la(x)+\A-\la(x):=\A(x)\in \mathcal{A}_s$ is $2R$-periodic (in every coordinate).
  and denote
\beq\label{nota2}
S_{s,j}^{\A(x)} f(x):=\sum_{0\le l< j}\mathscr{L}_{s,\A(x)} [\widehat{K_{0}}(2^l\cdot)](D)f(x)\ \ \  (j\in\N,\   x\in\Z^n).
\eeq
Lemma \ref{endle} gives that
for each  $s\in\N$,  every  $R\ge 1$ and any $\e>0$,
\beq\label{edou1}
\| \big(S_{s,j}^{\A(x)} f\big)_{j\in\N}|\|_{\ell^p(x\in\B_R;\V^r)}
\les_\e R^\e 2^{-\gamma_p s}\|f\|_{\ell^p(\Z^n)}\ \  (1<p<\infty),
\eeq
where  $\gamma_p$ is  given as in (\ref{de11}).
This with (\ref{Ad1}) and linearization  gives
\beq\label{edou12}
\| \sup_{j\in\N^B}\sup_{\A\in \mathcal{A}_s} |S_{s,j}^\A f|\|_{\ell^p(\B_R)}
\les_\e R^\e 2^{-\gamma_p s}\|f\|_{\ell^p(\Z^n)}.\footnote{Although Lemma 4.4 in \cite{KR23} can be employed to eliminate the $R^\epsilon$-loss in (\ref{edou12}),  (\ref{edou12}) suffices for our specific application.
}
\eeq
We will use (\ref{edou1}) and (\ref{edou12})  to prove (\ref{A2}).
%We first show (\ref{b12}).
By utilizing the Abel transform and (\ref{nota2}), we write the sum over $l$ on the left-hand side of (\ref{A2}) ($\A=\A(x)$) as
$$
\begin{aligned}
&\ \sum_{v_s+1\le l< j+1} {\ind {(x,s)\in \mathscr{U}_{l-1} }}
 S_{s,l}^{\A(x)} f(x)-\sum_{v_s\le l< j}
{\ind {(x,s)\in \mathscr{U}_{l} }}S_{s,l}^{\A(x)} f(x)\\
=&\ \LL_{s,{\A(x)}}^{(1)}(x)+\LL_{s,j,{\A(x)}}^{(2)}(x)+\LL_{s,j,{\A(x)}}^{(3)}(x),
\end{aligned}
$$
where $\LL_{s,{\A(x)}}^{(1)}(x)$ and $\LL_{s,j,{\A(x)}}^{(v)}(x) (v=2,3)$ are given by
$$
\begin{aligned}
\LL_{s,{\A(x)}}^{(1)}(x):=&\ - {\ind {(x,s)\in \mathscr{U}_{v_s}}} ~S_{s,v_s}^{\A(x)} f(x),\\
\LL_{s,j,{\A(x)}}^{(2)}(x):=&\
{\ind {(x,s)\in \mathscr{U}_{j-1}}} ~S_{s,j}^{\A(x)} f(x),\\
\LL_{s,j,{\A(x)}}^{(3)}(x):=&\ {\ind{j\ge v_s+2}}\sum_{v_s+1\le l<j}
\big({\ind {(x,s)\in \mathscr{U}_{l-1} }}-{\ind {(x,s)\in \mathscr{U}_{l}}}\big)~ S_{s,l}^{\A(x)} f(x).
\end{aligned}
$$
%Note that the first term $\LL_{s,\A}^{(1)}(x)$ for (\ref{A2})    can be neglected
Since $\LL_{s,\A}^{(1)}(x)$ does not depend on $j$,  by  (\ref{edou12}), we have
$$\|\big(\LL_{s,\A(x)}^{(1)}(x)
\big)_{j\in\N^B}\|_{\ell^p(x\in \B_R;\V^r)}\les \| \sup_{j\in\N^B}\sup_{\A\in \mathcal{A}_s} |S_{s,j}^\A f|\|_{\ell^p(\B_R)}\les_\e R^\e 2^{-\gamma_p s}\|f\|_{\ell^p(\Z^n)}.
$$
Thus, to prove (\ref{A2}), it suffices to show  that
\beq\label{Ar1}
\|\big(\LL_{s,j,\A(x)}^{(v)}(x)
\big)_{j\in\N^B}\|_{\ell^p(x\in \B_R;\V^r)}
\les_\e   2^{-c_ps } R^\e\|f\|_{{\ell^p}}\ \ \ (v=2,3).
\eeq
 %for  $2R$-periodic function $\A(x)\in \mathcal{A}_s$.
For the case $v=2$,  using   $\|({\ind {(x,s)\in \mathscr{U}_{j-1} }})_{j\in\N^B}\|_{\V^r} \les1$ and (\ref{var1}),   we deduce
$$\|\big(\LL_{s,j,\A(x)}^{(2)}(x)
\big)_{j\in\N^B}\|_{\V^r}\les \|({\ind {(x,s)\in \mathscr{U}_{j-1} }})_{j\in\N^B}\|_{\V^r} \|  (S_{s,j}^{\A(x)} f(x))_{j\in\N^B} \|_{\V^r}
\les \|  (S_{s,j}^{\A(x)} f(x))_{j\in\N^B} \|_{\V^r},
$$
which with (\ref{edou1}) completes the proof of  (\ref{Ar1})  for the case $v=2$.  As for the case $v=3$, we only need to use (\ref{edou12}). Using  $\sum_{l\in \N^B}
|{\ind {(x,s)\in \mathscr{U}_{l-1} }}-{\ind {(x,s)\in \mathscr{U}_l }}| \les 1$,  we have
$$
\|\big(\LL_{s,j,\A(x)}^{(3)}(x)
\big)_{j\in\N^B}\|_{\ell^p(x\in \B_R;\V^r)}
\les \|\big(\LL_{s,j,\A(x)}^{(3)}(x)
\big)_{j\ge v_s+2}\|_{\ell^p(x\in \B_R;\V^r)}
\les \| \sup_{l\in\N^B}\sup_{\A\in \mathcal{A}_s} |S_{s,l}^\A f|\|_{\ell^p(\B_R)}
$$
which with (\ref{edou12}) yields (\ref{Ar1}) for the case $v=3$.

%Short
\section{Major arcs estimate III: Proof of Proposition  \ref{t31}}
\label{slong4}
In this section, we shall prove major arcs estimate III by using the Plancherel-P\'{o}lya inequality,  the Stein-Wainger-type estimate, the  multi-frequency square function estimate in Lemma \ref{ccz} and  the shifted square  function estimate in  Appendix \ref{app}. In particular, since here the kernel is  rough and variable-dependent,
the method  by  Mirek-Stein-Trojan \cite{MST17} strongly depending   on the numerical inequality (\ref{k01}),  does not work any more.
\subsection{Reduction of Proposition  \ref{t31}}
  Write
$$\phi_{2^j,N,\mu(x)}^{(3)}(\xi)=\chi(2^j\xi)\phi_{2^j,N,\mu(x)}^{(5)}+\phi_{2^j,N,\mu(x)}^{(4)}(\xi)\chi(2^j\xi)+\phi_{2^j,N,\mu(x)}^{(3)}(\xi)\chi^\cc(2^j\xi),$$
where $\chi^\cc:=1-\chi$,
 $$
 \begin{aligned}
 \phi_{2^j,N,\mu(x)}^{(4)}(\xi):=&\ \int_{2^j\le |y|\le N}e(\mu(x)|y|^{2d})\big(e(y\cdot\xi)-1\big)K(y) dy\ \ {\rm and}\\
\phi_{2^j,N,\mu(x)}^{(5)}:=&\ \int_{2^j\le |y|\le N}e(\mu(x)|y|^{2d})K(y) dy.
\end{aligned}
$$
  Note that $\phi_{2^j,N,\mu(x)}^{(5)}$ is independent of  the variable $\xi$. Remember the notation (\ref{noe1}). Since $\phi_{2^j,N,\mu(x)}^{(4)}(0)=0$, we  can use the sum $\phi_{2^j,N,\mu(x)}^{(4)}(\xi)\sum_{k\le 0}\psi_{j,k}(\xi)$ to replace
 $\phi_{2^j,N,\mu(x)}^{(4)}(\xi)\chi(2^j\xi)$. This with
%,  and  $\chi(2^j\xi)=\sum_{k\le 0}\psi(2^{j-k}\xi)$ for all $\xi\neq 0$,
  $\chi^c(\xi)=\sum_{k\ge 1}\psi(2^{-k}\xi)$ yields
$$
\begin{aligned}
\phi_{2^j,N,\mu(x)}^{(3)}(\xi)
=&\ \chi(2^j\xi)\phi_{2^j,N,\mu(x)}^{(5)}+\phi_{2^j,N,\mu(x)}^{(4)}(\xi)\sum_{k\le 0}\psi_{j,k}(\xi)+\phi_{2^j,N,\mu(x)}^{(3)}(\xi)\chi^\cc(2^j\xi).
\end{aligned}
$$
%decompose  $\phi_{2^j,N,\mu(x)}^{(3)}$ as
%$$\phi_{2^j,N,\mu(x)}^{(3)}(\xi)=\phi_{2^j,N,\mu(x)}^{(3)}(\xi) \rho(2^j\xi)+\phi_{2^j,N,\mu(x)}^{(3)}(\xi) \rho^\cc(2^j\xi),$$
%where  $\rho^\cc=1-\rho$.
%We further write  $\phi_{2^j,N,\mu(x)}^{(3)}$ as
%$$
%\begin{aligned}
%\phi_{2^j,N,\mu(x)}^{(3)}(\xi):=&\ \rho(2^j\xi)\phi_{2^j,N,\mu(x)}^{(5)}+\phi_{2^j,N,\mu(x)}^{(4)}(\xi)\rho(2^j\xi)+\phi_{2^j,N,\mu(x)}^{(3)}(\xi)\rho^\cc(2^j\xi)\\
%=&\ \rho(2^j\xi)\phi_{2^j,N,\mu(x)}^{(5)}+\phi_{2^j,N,\mu(x)}^{(4)}(\xi)\sum_{k\le 0}\psi(2^{j-k}\xi)+\phi_{2^j,N,\mu(x)}^{(3)}(\xi)\rho^\cc(2^j\xi),
%\end{aligned}
%$$
%where we have used $\phi_{2^j,N,\mu(x)}^{(4)}(0)=0$ and $\rho(2^j\xi)=\sum_{k\le 0}\psi(2^{j-k}\xi)$ for all $\xi\neq 0$.
%This equality is in the sense of distribution.
By the triangle inequality, we can reduce the proof of
 Proposition  \ref{t31} to showing  the following lemmas.
\begin{lemma}\label{l761}
For every $p\in (1,\infty)$ and every $\tilde{\e}_\circ\in (0,1)$, we have
\beq\label{O0}
\| \sum_{j\in \N^B}\|\big(\sum_{1\le s \le \tilde{\e}_\circ  (j)}[\mathcal{G}_{j,N,\la(x),\tilde{\e}_\circ}^{(5),s}(D)f](x) \big)_{N\in [2^j,2^{j+1})}\|_{V^2}^2\big)^{1/2}\|_{\ell^p(\Z^n)}
\les \|f\|_{\ell^p(\Z^n)}.
\eeq
where  $[\mathcal{G}_{j,N,\la(x),\tilde{\e}_\circ}^{(5),s}(D)f](x):=
\phi_{2^j,N,\mu(x)}^{(5)}\times\mathscr{L}_{s,\A,\kappa}[\chi(2^j\cdot)
](D)f(x).$
%where $\phi(2^j\xi):=\sum_{k\le 0}\psi(2^{j-k}\xi)$.
\end{lemma}
\begin{lemma}\label{l762}
%Let $J_l=\{k\in\Z:\ (-1)^l ~k\le 0\}$ with  $l=3,4$.
Let $p\in (1,\infty)$, $\tilde{\e}_\circ\in (0,1)$ and $k\le 0$. There is a  constant  $c_p>0$ such that
$$\|\big( \sum_{j\in \N^B}\|\big(\sum_{1\le s \le \tilde{\e}_\circ  (j)}[\mathcal{G}_{j,N,\la(x),k,\tilde{\e}_\circ}^{(4),s}(D)f](x) \big)_{N\in [2^j,2^{j+1})}\|_{V^2}^2\big)^{1/2}\|_{\ell^p(\Z^n)}
\les 2^{c_p k}\|f\|_{\ell^p(\Z^n)},$$
where  $[\mathcal{G}_{j,N,\la(x),k,\tilde{\e}_\circ}^{(4),s}(D)f](x):=\big(\mathscr{L}_{s,\A,\kappa}[\phi_{2^j, N,\mu(x),\tilde{\e}_\circ}^{(4),*}~\psi_{j,k}
](D)f\big)(x)$. %with $\psi_{j,k}(\xi)=\psi(2^{j-k}\xi)$.
\end{lemma}
\begin{lemma}\label{l763}
For every $p\in (1,\infty)$ and  every $\tilde{\e}_\circ\in (0,1)$,  we have
$$\|\big( \sum_{j\in \N^B}\|\big(\sum_{1\le s \le \tilde{\e}_\circ  (j)}[\mathcal{G}_{j,N,\la(x),\tilde{\e}_\circ}^{(3),s}(D)f](x) \big)_{N\in [2^j,2^{j+1})}\|_{V^2}^2\big)^{1/2}\|_{\ell^p(\Z^n)}
\les \|f\|_{\ell^p(\Z^n)},$$
where  $[\mathcal{G}_{j,N,\la(x),\tilde{\e}_\circ}^{(3),s}(D)f](x):=\big(\mathscr{L}_{s,\A,\kappa}[\phi_{2^j, N,\mu(x),\tilde{\e}_\circ}^{(3),*}~\chi^\cc(2^{j}\cdot)
](D)f\big)(x).$
\end{lemma}
Since the value of $\tilde{\e}_\circ$ is not critical, we omit  it from  notations like 
$\mathcal{G}_{j,N,\la(x),\tilde{\e}_\circ}^{(i_1),s}$ ($i_1=3,5$), $\mathcal{G}_{j,N,\la(x),k,\tilde{\e}_\circ}^{(4),s}$ and $\phi_{2^j, N,\mu(x),\tilde{\e}_\circ}^{(i_2),*}$ ($i_2=3,4$)  unless clarity demands it or it needs to be emphasized for other reasons; since the value of $\ka$
is not important, we  will apply  the notations (\ref{not32}) in what follows. Furthermore,
we slightly abuse notation $v_s=v_s(\tilde{\e}_\circ)$ given as in  (\ref{not2}) (with $\e_\circ'=\tilde{\e}_\circ$).
%Remember  the notation  $v_s=\max\{C_0, \e_\circ^{-1}s\}$.
\subsection{Proof of Lemma \ref{l761}}
\label{lll761}
It suffices to show that for each $p\in (1,\infty)$ and each $m\ge 1$,
\begin{align}
\| \big(\sum_{j\ge v_s}\|{\ind {\mathcal{S}_j^0}}(x) \big(\mathcal{G}_{j,2^jt,\la(x)}^{(5),s}(D)f(x) \big)_{t\in [1,2)}\|_{V^2}^2\big)^{1/2}\|_{\ell^p(x\in \Z^n)}
\les&\  2^{-\gamma_ps} \|f\|_{\ell^p(\Z^n)}\ \ {\rm and }\label{Y1}\\
\|\big( \sum_{j\ge v_s}\|{\ind {\mathcal{S}_j^m}} (x)\big(\mathcal{G}_{j,2^j t,\la(x)}^{(5),s}(D)f(x) \big)_{t\in [1,2)}\|_{V^2}^2\big)^{1/2}\|_{\ell^p(x\in \Z^n)}
\les&\  2^{-m-\gamma_p s}\|f\|_{\ell^p(\Z^n)}\label{Y2}
\end{align}
with $\gamma_p$ given as in (\ref{de11}). Direct computation gives 
\beq\label{7L90}
\|\big(\mathcal{G}_{j,2^jt,\la(x)}^{(5),s}(D)f(x) \big)_{t\in [1,2)}\|_{V^2}
\les  \|\big(
\phi_{2^j,2^j t,\mu(x)}^{(5)}\big)_{t\in [1,2)}\|_{V^2}~
\sup_{\A\in\mathcal{A}_s}|\mathscr{L}_{s,\A}[\chi(2^j\cdot)
](D)f|(x).
 \eeq
We first prove (\ref{Y1}).
Using $x\in {\mathcal{S}_j^0}$,
Taylor expansion and $\int_{\mathbb{S}^{n-1}}\Omega(\theta)d\sigma=0$, we obtain
\beq\label{GG1}
\begin{aligned}
\phi_{2^j,2^jt,\mu(x)}^{(5)}=&\ \int_{2^j\le |y|\le 2^jt}e(\mu(x)|y|^{2d})K(y)dy
=\ \sum_{l=1}^\infty \frac{(2\pi i)^l}{l!}(\mu(x)2^{2dj})^l I_{j,t}^l
\end{aligned}
\eeq
with $I_{j,t}^l$ defined by   $I_{j,t}^l:=\int_{2^j\le |y|\le 2^jt} (2^{-j}|y|)^{2dl}K(y)dy$, which  satisfies
$$%\beq\label{GG1}
 \|\big(I_{j,t}^l\big)_{t\in [1,2)}\|_{V^2}\les \int_{2^j\le |y|\le 2^{j+1}} (2^{-j}|y|)^{2dl}|K(y)|dy\les C^l.
$$
This with (\ref{GG1}) gives that  ${\ind {\mathcal{S}_j^0}}(x)  \|\big(
\phi_{2^j,2^j t,\mu(x)}^{(5)}\big)_{t\in [1,2)}\|_{\V^2}$ is
\beq\label{MK1}
\begin{aligned}
%{\ind {\mathcal{S}_j^0}}(x)  \|\big(
%\phi_{2^j,2^j t,\mu(x)}^{(5)}\big)_{t\in [1,2]}\|_{\V^2}
\les\  {\ind {\mathcal{S}_j^0}}(x) \sum_{l=1}^\infty \frac{(2\pi )^l}{l!}(|\mu(x)|2^{2dj})^l
 \|(I_{j,t}^l)_{t\in [1,2)}\|_{V^2}
 \les&\ {\ind {\mathcal{S}_j^0}}(x) \sum_{l=1}^\infty \frac{(2\pi )^l}{l!}C^l  (|\mu(x)|2^{2dj})^l.
 \end{aligned}
\eeq
Inserting  (\ref{MK1}) into   (\ref{7L90}), %and using the resulting inequality
we can bound  the left-hand side of (\ref{Y1})  by  a constant times
$$
\begin{aligned}
&\ \sum_{l=1}^\infty \frac{(2\pi )^l}{l!}C^l ~\| \sum_{j\ge v_s}{\ind {\mathcal{S}_j^0}}(x)~ (|\mu(x)|2^{2dj})^l
\sup_{\A\in\mathcal{A}_s}|\mathscr{L}_{s,\A}[\chi(2^j\cdot)
](D)f|\|_{\ell^p(\Z^n)}\\
\les&\ \sum_{l=1}^\infty \frac{(2\pi )^l}{l!} C^l~\| \sup_{j\ge v_s}
\sup_{\A\in\mathcal{A}_s}|\mathscr{L}_{s,\A}[\chi(2^j\cdot)
](D)f|\|_{\ell^p(\Z^n)},
 \end{aligned}
$$
which with (\ref{5d3}) gives the desired (\ref{Y1}).

Next,  we prove (\ref{Y2}).
Let $\mathfrak{H}(y)=|y|^{2d}$.
 Integration by parts gives that  $\phi_{2^j,2^jt,\mu(x)}^{(5)}$ equals
$$
\begin{aligned}
&\ \frac{-i}{2\pi \mu(x) 2^{2dj}}\int_{1\le |y|\le t}\frac{\p}{\p y}\big[e\big(\mu(x)\mathfrak{H}(2^j y) \big)\big] \frac{K(y)}{\mathfrak{H}'(y)} dy\\
= &\ \frac{i}{2\pi \mu(x) 2^{2dj}}
\int_{1\le |y|\le t}e\big(\mu(x)\mathfrak{H}(2^j y) \big)\frac{\p}{\p y}\big(\frac{K(y)}{\mathfrak{H}'(y)}\big) dy
+H_{t,j,\mu(x)}^{(1)}+H_{j,\mu(x)}^{(2)},
\end{aligned}
$$
where $H_{j,\mu(x)}^{(2)}$ is independent of $t$ (so this term does not affect the $V^r$ seminorm), and
$H_{t,j,\mu(x)}^{(1)}$ satisfies
$$
\|\frac{\p}{\p t} H_{t,j,\mu(x)}^{(1)}\|_{L^1_t( [1,2])}
%+2^{-m}\| H_{N,j,\mu(x)}^{(1)}\|_{L^2(N\in [1,2])}
\les (|\mu(x)| 2^{2dj})^{-1}.
$$
This together with  $x\in {\mathcal{S}_j^m}$ and  a routine  computation  gives
that ${\ind {\mathcal{S}_j^m}} (x) \|\big(
\phi_{2^j,2^j t,\mu(x)}^{(5)}\big)_{t\in [1,2)}\|_{V^2}$ is
\beq\label{LM1}
\begin{aligned}
%\les&\  {\ind {\mathcal{S}_j^m}}  \|
%\phi_{2^j,2^j N,\mu(x)}^{(5)}\|_{B_{2,1}^{1/2}}\\
\les\ \frac{1}{|\mu(x)| 2^{2dj}}
\int_{1\le |y|\le 2}\big|\frac{\p}{\p y}\big(\frac{K(y)}{\mathfrak{H}'(y)}\big)\big| dy
+\|\frac{\p}{\p t} H_{t,j,\mu(x)}^{(1)}\|_{L^1(t\in [1,2])}
\les\  2^{-m},
\end{aligned}
\eeq
By (\ref{LM1}), (\ref{7L90}) and the inequality in (\ref{le1}),   the left hand side of (\ref{Y2}) is
$$
\les 2^{-m}\| \sup_{j\ge v_s}
\sup_{\A\in\mathcal{A}_s}|\mathscr{L}_{s,\A}[\chi(2^j\cdot)
](D)f|\|_{\ell^p(\Z^n)}.
$$
This combined  with (\ref{5d3}) gives the desired (\ref{Y2}).
%\les 2^{-m-c_ps}\|f\|_{\ell^p(\Z^n)},
\subsection{Proof of Lemma \ref{l762}}
\label{lll762}
It suffices to show that for every $s\in\N$ and each $k\le 0$, the inequality
\beq\label{I0}
\|\big( \sum_{j\ge v_s}\|\big(\mathcal{G}_{j,N,\la(x),k}^{(4),s}(D)f(x) \big)_{N\in [2^j,2^{j+1})}\|_{V^2}^2\big)^{1/2}\|_{\ell^p(x\in \Z^n)}
\les 2^{c_p (k-s)} \|f\|_{\ell^p(\Z^n)}
\eeq
holds for some  $c_p>0$.
By Taylor expansion, we write  $\phi_{2^j, N,\mu(x)}^{(4)}(\xi)\psi_{j,k}(\xi)$ as%$\phi_{2^j, N,\mu(x),M}^{(4)}\psi(2^{j-k}\xi)$ as
\beq\label{I2}
\sum_{l=1}^\infty
\frac{(2\pi i)^l }{l!}2^{kl}~ \psi_{j,k}(\xi) \int_{2^j\le |y|\le N} e(\mu(x)|y|^{2d})(y\cdot 2^{-k}\xi)^l K(y)
dy.
\eeq
Expanding  $(y\cdot 2^{-k}\xi)^l$, we can express  the product of $\psi_{j,k}(\xi)$ and  the integral on the right-hand side of (\ref{I2})   as the sum of  $\mathcal{O}(n^l)$ terms  similar to
$\bar{\psi}^l_{j,k}(\xi) \bar{I}_{j,N,k,\mu(x)}^{l}$, where
$$\bar{\psi}^l_{j,k}(\xi):=\bar{\psi}^l(2^{j-k}\xi),\ \ \bar{I}_{j,N,k,\mu(x)}^{l}:=\int_{2^j\le |y|\le N} e(\mu(x)|y|^{2d})2^{-jn}\bar{K}_l(2^{-j}y)dy.
$$
Here  $\bar{\psi}^l$ and $\bar{K}_l$  are   variants of $K$ and $\psi$, respectively; and there exists a  constant $\mathcal C$ independent of $l,s,k$ such that 
%they satisfy   the following:
\beq\label{BV1}
\begin{aligned}
|\bar{K}_l(y)|\sim&\  \mathcal C^l,\\
 |\bar{\psi}^l(2^{j}\xi)|\les&\  \mathcal C^l\min\{2^j|\xi|,(2^j|\xi|)^{-1}\}\ \ {\rm and}\\
\|\big(\sum_{j\in\Z}|\bar{\psi}^l(2^{j}D){f}|^2\big)^{1/2}\|_{L^p(\R^n)}\le&\  \mathcal C^l \|f\|_{L^p(\R^n)},
\end{aligned}
\eeq
%$$
%\int_{2^j\le |y|\le N} 2^{-jn}|\bar{K}_l(2^{-j}y)|dy\le C^l\ \ \ {\rm and}\ \ \|(\sum_{j\in\Z}|\bar{\psi}^l(2^{j-k}D){f}|^2)^{1/2}\|_{L^p(\R^n)}\le C^l \|f\|_{L^p(\R^n)},
%$$
%since $N\in [2^j,2^{j+1}]$ at this moment.
%Note that  $\bar{I}_{j,2^jt,k,\mu(x)}^{l}$  is a constant in the frequency variable $\xi$.
whenever $|y|\sim 1$ and $\xi\in\R^n$. Thus, to prove (\ref{I0}),  it suffices to show %that for $k\le 0$ and $s\ge 0$,
\beq\label{I91}
\begin{aligned}
&\ 2^k \|\big( \sum_{j\ge v_s}\sup_{\A\in\mathcal{A}_s}\big|\|\big(\bar{I}_{j,2^jt,k,\mu(x)}^{l}\big)_{t\in [1,2)}\|_{V^2}^2 |\mathscr{L}_{s,\A}[\bar{\psi}^l_{j,k}](D)f|^2\big|\big)^{1/2}\|_{\ell^p(x\in\Z^n)}\\
\les&\  \mathcal C^l 2^{c_p (k-s)} \|f\|_{\ell^p(\Z^n)}.
\end{aligned}
\eeq
By the change of  variables $y\to 2^jy$, we have $$\bar{I}_{j,2^jt,k,\mu(x)}^{l}=\int_{1\le |y|\le t} e(\mu(x)2^{2dj}|y|^{2d})\bar{K}_l(y)dy,$$
%which satisfies a trivial bound  $|\p_\tau (I_{j,2^j\tau}^{k})(x)|\les C^l$ (since (\ref{BV1}))
%whenever $\tau\sim1$. Thus
which with (\ref{BV1})$_1$ gives
\beq\label{I3}
\|\big(\bar{I}_{j,2^jt,k,\mu(x)}^{l}\big)_{t\in [1,2)}\|_{V^2}\les
\int_{1\le |y|\le 2} |\bar{K}_l(y)| dy
\les  \mathcal C^l.
\eeq
%with some bounded function  $|F_l|\les C^l$.
On the other hand,  invoking $k\le 0$, we   infer  by  Lemma \ref{ccz} along with  (\ref{BV1})$_2$ and (\ref{BV1})$_3$  that
\beq\label{I92}
\begin{aligned}
\|\big( \sum_{j\ge v_s}|\sup_{\A\in\mathcal{A}_s}|\mathscr{L}_{s,\A}[\bar{\psi}^l_{j,k}](D)f|^2\big)^{1/2}\|_{\ell^p(\Z^n)}
\le &\  \|\big( \sum_{j\in \N^B}|\sup_{\A\in\mathcal{A}_s}|\mathscr{L}_{s,\A}[\bar{\psi}^l(2^{j}\cdot)](D)f|^2\big)^{1/2}\|_{\ell^p(\Z^n)}\\
\les&\  \mathcal C^l 2^{-\gamma_p s} \|f\|_{\ell^p(\Z^n)}.
\end{aligned}
\eeq
 Finally, (\ref{I91}) follows from  (\ref{I3}) and (\ref{I92}).
This ends the proof of Lemma \ref{l762}.
 %Note that  for any $\e\in (0,1)$,
% \beq\label{Con1}
% |\bar{\psi}^l(2^{j}\xi)|\les\ C^l
% \les C^l \min\{2^{k\e}(2^j|\xi|)^{-\e}, 2^{-k\e}(2^j|\xi|)^{\e} \}
% \les C^l 2^{-\e k}\min\{(2^j|\xi|)^{-\e}, (2^j|\xi|)^{\e} \}
% \eeq
% since $2^j|\xi|\sim 2^k$ at this moment. Moreover, it satisfies the square function estimate with $B=1$.
 % Using Lemma \ref{ccz} and choosing $\e$ small enough in (\ref{Con1}), we have
 % \beq\label{I8}
%\|(\sum_{j\in\Z^n}\sup_{\A\in\mathcal{A}_s}| \mathscr{L}_{s,\A}[\bar{\psi}^l(2^{j}\cdot)](D)f(x)|^2)^{1/2}\|_{\ell^p}\les C^l2^{-c_p s}  \|f\|_{\ell^p},
 %\eeq
% which yields (\ref{I5}) immediately.
 %This yields
%Since
% $h=\min\{2^{\sqrt{|k|}},2^{ j}\}$, by dividing  the sum over $j\ge a_s$ in (\ref{I5})  into two parts:
% $j\ge \max\{|k|^2,a_s\}$ and $a_s\le j\le |k|^2$,   we may bound the LHS of (\ref{I5}) by   $2^{2\sqrt{|k|}}$ multiplied by the LHS of (\ref{I8}), up to a uniform constant. Thus,
%  LHS of (\ref{I5})  is
 % $\les C^l 2^{2\sqrt{|k|}} 2^{-c_p s}  \|f\|_{\ell^p}$, which yields (\ref{I5}) immediately.
\subsection{Proof of Lemma \ref{l763}}
\label{lll763}
By $\chi^\cc(2^{j}\xi)=\sum_{k\ge 1}\psi_{j,k}(\xi)$,
it suffices to show  that for each $p\in (1,\infty)$,  there is a constant $c_p>0$ such that
\beq\label{HH1}
\|\big( \sum_{j\ge v_ s}\|\big({\mathcal{G}}_{j,2^j\tau,\la(x),k}^{(3),s}(D)f(x) \big)_{\tau\in [1,2)}\|_{V^2}^2\big)^{1/2}\|_{\ell^p(x\in \Z^n)}
\les 2^{-c_p(k+s)}\|f\|_{\ell^p(\Z^n)},
\eeq
where  $\big(\mathcal{G}_{j,2^j\tau,\la(x),k}^{(3),s}f\big)(x)
=\big(\mathscr{L}_{s,\A}[\phi_{2^j, 2^j\tau,\mu(x)}^{(3),*}\psi_{j,k}
](D)f\big)(x)$. By (\ref{rou1}) and the Plancherel-P\'{o}lya inequality, that is, 
$$\|(f_t)_{t\in [1,2]}\|_{V^2}\les \|\p_t f_t\|_{L^1(t\in [1,2])}\ \ {\rm and}\ \  \|(f_t)_{t\in [1,2]}\|_{V^2}\les
\|f_t ~\rho(t)\|_{B_{2,1}^{1/2}(t\in\R)},$$
respectively, 
where the function $\rho$
 is smooth, compactly supported, and equals 1 on the interval $[1,2]$,    (\ref{HH1}) is a direct consequence of the following inequalities:
\beq\label{HH2}
\sup_{1\le \tau\le2}\|\big( \sum_{j\ge v_s}|\p_\tau\big({\mathcal{G}}_{j,2^j\tau,\la(x),k}^{(3),s}(D)f \big) (x)|^2\big)^{1/2}\|_{\ell^p(x\in \Z^n)}
\les k^2 2^{-c_ps} \|f\|_{\ell^p(\Z^n)},
\eeq
for some $c_p>0$,
and
\beq\label{HH3}
\|\big( \sum_{j\ge v_s}\|\big({\mathcal{G}}_{j,2^j\tau,\la(x),k}^{(3),s}(D)f\big)(x) \rho(\tau)\|_{B_{2,1}^{1/2}(\tau\in\R)}^2\big)^{1/2}\|_{\ell^2(x\in\Z^n)}
\les 2^{-c(k+s)}\|f\|_{\ell^2(\Z^n)}
\eeq
 for some $c>0$. Next, we prove (\ref{HH2}) and (\ref{HH3}) in order.
\begin{proof}[Proof of (\ref{HH2})]
Changing  variables $y\to 2^j y$, we write $\phi_{2^j, 2^j\tau,\mu(x)}^{(3)}$ as
$$
\begin{aligned}
\phi_{2^j, 2^j\tau,\mu(x)}^{(3)}(\xi)%=&\ \int_{2^j\le |y|\le 2^j \tau} e(\mu(x)|y|^{2d}+y\cdot\xi) K(y)dy\\
=&\ \int_{1\le |y|\le  \tau} e(\mu(x)2^{2dj}|y|^{2d}+y\cdot2^j\xi) K(y)dy.
\end{aligned}
$$
To compute $\p_\tau (\phi_{2^j, 2^j\tau,\mu(x)}^{(3)})$,
it is necessary to bifurcate the analysis into two scenarios:  $n=1$ and $n\geq 2$.
For the case $n=1$,
we have
$$
\begin{aligned}
\p_\tau (\phi_{2^j, 2^j\tau,\mu(x)}^{(3)})(\xi)=
\big(e(\mu(x)2^{2dj}|\tau|^{2d}+\tau 2^j\xi) +e(\mu(x)2^{2dj}|\tau|^{2d}-\tau 2^j\xi) \big)K(\tau).
\end{aligned}
$$
Thus, to prove (\ref{HH2}),  it suffices to show that  for each $k\ge 1$,
\beq\label{HH21}
\sup_{1\le |\tau|\le2}\|\big( \sum_{j\ge v_s} \sup_{\A\in\mathcal{A}_s}|\mathscr{L}_{s,\A}[\mathfrak{M}_{\tau,j-k}^k
](D)f|^2\big)^{1/2}\|_{\ell^p(\Z^n)}
\les k^2 2^{-c_p s}\|f\|_{\ell^p(\Z^n)},
\eeq
 where  $\{\mathfrak{M}_{\tau,l}^k\}_{l\in\Z}$ are defined by
 $\mathfrak{M}_{\tau,l}^k(\xi):=e(\tau 2^{l+k}\xi) \psi(2^{l}\xi).$
 By a routine computation, we have
\beq\label{sqi2}
|\mathfrak{M}_{\tau,j}^k(\xi)|\les \min\{2^{j}|\xi|,(2^{j}|\xi|)^{-1}\} \ \ \ (\xi\in\R^n).
\eeq
Moreover, by Lemma \ref{Ap1} (for the case $n=1$) and $1\le |\tau|\le2$,
\begin{align}
   \|\big(\sum_{j\in\Z}  |\F^{-1}_{\R^n}(\mathfrak{M}_{\tau,j}^k)*_{\R^n}{f_j})|^2\big)^{1/2}
\|_{L^p(\R^n)}\les&\  \| \big(\sum_{j\in\Z}    |f_j*_{\R^n}\tilde{h}_j(x+\tau 2^{j+k})|^2\big)^{1/2}
\|_{L^p(\R^n)}\\
%\les&\  \big(\sum_{v=0}^2|M^{[2^{k+v}]}f_j|^2\big)^{1/2}\\
%\les&\  \log^2(1+ 2^k)\|\big(\sum_{j\in\Z}   |f_j|^2\big)^{1/2}\|_{L^p(\R^n)}\\
\les&\  k^2 \|\big(\sum_{j\in\Z}    |f_j|^2\big)^{1/2}\|_{L^p(\R^n)},
\label{sqi1}
\end{align}
where ${\tilde{h}_j}$ is given by  $\tilde{h}_j(y)=2^{-jn}\check{\psi}(2^{-j}y)$.
Applying (\ref{sqi1}),   (\ref{sqi2}) and Lemma \ref{ccz} (with $\mathfrak{M}_{j}=\mathfrak{M}_{\tau,j}^k$, $A=1$ and $B=k^2$), we infer
\beq\label{HH211}
\|\big( \sum_{j\in\Z} \sup_{\A\in\mathcal{A}_s}|\mathscr{L}_{s,\A}[\mathfrak{M}_{\tau,j}^k
](D)f|^2\big)^{1/2}\|_{\ell^p(\Z^n)}
\les k^2 2^{-c_p s}\|f\|_{\ell^p(\Z^n)},
\eeq
which yields (\ref{HH21}) by
changing variables $j\to j-k$ on the left-hand side of (\ref{HH211}). Thus we complete the proof of  (\ref{HH2}) for the case  $n=1$.
As for   the case $n\ge2$, we shall use    similar arguments. Rewrite $\phi_{2^j, 2^j\tau,\mu(x)}^{(3)}$ by the polar coordinates  as
$$
\begin{aligned}
\phi_{2^j, 2^j\tau,\mu(x)}^{(3)}(\xi)
=\int_1^{ \tau} \int_{\mathbb{S}^{n-1}}
 e(\mu(x)2^{2dj}r^{2d}+r\theta\cdot 2^j\xi) \Omega(\theta) r^{-1}drd\theta,
\end{aligned}
$$
which yields by a direct computation that
$$
\begin{aligned}
\p_\tau [\phi_{2^j, 2^j\tau,\mu(x)}^{(3)}](\xi)=
\int_{\mathbb{S}^{n-1}}
 e(\mu(x)2^{2dj}\tau^{2d}+2^j\tau\theta\cdot \xi) \Omega(\theta) \tau^{-1}d\theta.
\end{aligned}
$$
By   repeating  the  arguments yielding (\ref{HH211}) and using  Lemma \ref{Ap1} (for the case  $n\ge 2$),  we obtain
\beq\label{HK21}
\sup_{\theta\in \mathbb{S}^{n-1}}\sup_{1\le r\le2} \|\big( \sum_{j\ge v_s} \sup_{\A\in\mathcal{A}_s}|\mathscr{L}_{s,\A}[\mathfrak{N}_{r\theta,j-k}^k
](D)f|^2\big)^{1/2}\|_{\ell^p(\Z^n)}
\les k^2 2^{-c_p s}\|f\|_{\ell^p(\Z^n)},
\eeq
%for all $\tau\in [1,2]\cup[-2,-1]$,
where $\mathfrak{N}_{r\theta,l}^k(\xi):=e(r\theta\cdot 2^{l+k}\xi) \psi(2^{l}\xi)$ whenever $l\in\Z$.
This yields (\ref{HH2}) for the case $n\ge 2$.
\end{proof}

\begin{proof}[Proof of (\ref{HH3})]
By a basic inequality
$$\|g\|_{B_{2,1}^{1/2}(\R)}\les \|g\|_{L^2(\R)}+\|g\|_{L^2(\R)}^{{1}/{2}}
\|g'\|_{L^2(\R)}^{{1}/{2}},$$
we reduce the matter to proving
\beq\label{HH30}
\|\big( \sum_{j\ge v_s}\|\big({\mathcal{G}}_{j,2^j\tau,\la(x),k}^{(3),s}(D)f\big)(x) \Psi(\tau)\|_{L^2(\tau\in\R)}^2\big)^{1/2}\|_{\ell^2(\Z^n)}
\les 2^{-c(k+s)}\|f\|_{\ell^2(\Z^n)}
\eeq
for   $\Psi\in \{\rho,\rho'\}$, and
\beq\label{HH31}
\|\big( \sum_{j\ge v_s}\|\p_\tau\big({\mathcal{G}}_{j,2^j\tau,\la(x),k}^{(3),s}(D)f\big)(x)~ \rho(\tau)\|_{L^2(\tau\in \R)}^2\big)^{1/2}\|_{\ell^2(\Z^n)}
\les k^2 2^{-cs}\|f\|_{\ell^2(\Z^n)}.
\eeq
Note that (\ref{HH31}) can be obtained by arguments yielding (\ref{HH21}) and (\ref{HK21}). As a consequence,  it remains to show (\ref{HH30}), which follows from
\begin{align}
\|\big( \sum_{j\ge v_s}\|{\ind {\mathcal{S}_j^0}}(x)~ \big({\mathcal{G}}_{j,2^j\tau,\la(x),k}^{(3),s}(D)f\big)(x) \Psi(\tau)\|_{L^2}^2\big)^{1/2}\|_{\ell^2(\Z^n)}
\les&\  2^{-c(k+s)}\|f\|_{\ell^2(\Z^n)}\ \ \ {\rm and}\label{HH33}\\
\|\big( \sum_{j\ge v_s}\|{\ind {\mathcal{S}_j^m}}(x)~ \big({\mathcal{G}}_{j,2^j\tau,\la(x),k}^{(3),s}(D)f\big)(x) \Psi(\tau)\|_{L^2}^2\big)^{1/2}\|_{\ell^2(\Z^n)}
\les&\  2^{-c(k+s+m)}\|f\|_{\ell^2(\Z^n)}.\label{HH34}
\end{align}
In the rest of this section, we prove (\ref{HH33}) and (\ref{HH34}).
For (\ref{HH33}), Taylor's expansion gives
$$
\begin{aligned}
\phi_{2^j, 2^j\tau,\mu(x)}^{(3)}(\xi)=&\  \sum_{l=0}^\infty
\frac{(2\pi i)^l }{l!}~ (\mu(x)2^{2dj})^l~\mathcal{I}_{\tau}^{j,l}(\xi),
\end{aligned}
$$
where $\mathcal{I}_{\tau}^{j,l}(\xi):=\int_{1\le |y|\le  \tau} e(y\cdot2^j\xi) |y|^{2dl} K(y)dy.$
Thus it suffices to show
\beq\label{HH35}
\|\big( \sum_{j\ge v_s} \sup_{\A\in\mathcal{A}_s}|\mathscr{L}_{s,\A}[\mathfrak{M}_{1,\tau,j-k}^{l,k}
](D)f(x)|^2\big)^{1/2}\|_{\ell^2(\Z^n)}
\les C^l 2^{-c(k+s)}\|f\|_{\ell^2(\Z^n)},
\eeq
where  $\{\mathfrak{M}_{1,\tau,j}^{l,k}\}_{j\in\Z}$ are  a sequence of functions  defined by
$$\mathfrak{M}_{1,\tau,j}^{l,k}(\xi):=\mathcal{I}_{\tau}^{j+k,l}(\xi)\psi(2^{j}\xi).$$
We first deduce  by integration by parts 
\beq\label{sqi5}
|\mathfrak{M}_{1,\tau,j}^{l,k}(\xi)|\les C^l 2^{-k} \min\{2^{j}|\xi|,(2^{j}|\xi|)^{-1}\}\ \ \ (\xi\in\R^n);
\eeq
furthermore,
by the  Fefferman-Stein inequality and $1\le |\tau|\le2$, we have
\beq\label{sqi9}
\begin{aligned}
   \|\big(\sum_{j\in\Z}  |\F^{-1}_{\R^n}(\mathfrak{M}_{1,\tau,j}^{l,k})*_{\R^n}{f}_j)|^2\big)^{1/2}
\|_{L^p(\R^n)}\les&\  C^l\| \big(\sum_{j\in\Z}    |M_{HL}f_j|^2\big)^{1/2}
\|_{L^p(\R^n)}\\
\les&\  C^l\| \big(\sum_{j\in\Z}    |f_j|^2\big)^{1/2}
\|_{L^p(\R^n)}.
\end{aligned}
\eeq
Thus, by (\ref{sqi5}), (\ref{sqi9}) and  Lemma \ref{ccz} (with $\mathfrak{M}_{j}=\mathfrak{M}_{1,\tau,j}^{l,k}$), %, $A=C^l$ and $B=C^l$),
we can infer
\beq\label{Ww1}
\|\big( \sum_{j\in\Z} \sup_{\A\in\mathcal{A}_s}|\mathscr{L}_{s,\A}[\mathfrak{M}_{1,\tau,j}^{l,k}
](D)f(x)|^2\big)^{1/2}\|_{\ell^2(\Z^n)}
\les C^l 2^{-c(k+s)}\|f\|_{\ell^2(\Z^n)},
\eeq
which yields (\ref{HH35}) by changing variables $j\to j-k$ on the left-hand side of (\ref{Ww1}).

Next, we prove (\ref{HH34}).
By Fubini's Theorem and the inequality in (\ref{le1}),  it suffices to show that there is a constant $c>0$ such that   for all $1\le |\tau|\le 2$, %the inequality
$$
\begin{aligned}
\|\sup_{j\ge v_s}\sup_{\A\in\mathcal{A}_s}\sup_{\mu\in I_{j,m}}
|\mathscr{L}_{s,\A}[\phi_{2^j, 2^j\tau,\mu}^{(3)}\psi_{j,k}
](D)f|\|_{\ell^2(\Z^n)}
\les 2^{-c(k+m+s)}\|f\|_{\ell^2(\Z^n)}.
\end{aligned}
$$
 Performing the arguments yielding (\ref{ge2}), we also obtain  that for any $\e\in (0,1)$,
\beq\label{BC1}
\|\sup_{j\ge v_s}\sup_{\A\in\mathcal{A}_s}\sup_{\mu\in I_{j,m}}
|\mathscr{L}_{s,\A}[\phi_{2^j, 2^j\tau,\mu}^{(3)}\psi_{j,k}
](D)f|\|_{\ell^2(\Z^n)}
\les 2^{\e m} 2^{-c(k+s)}\|f\|_{\ell^2(\Z^n)}.
\eeq
On the other hand,  by the Stein-Wainger-type theorem, we have
%\beq\label{ge123}
$$
\|\sup_{j\in\Z}\sup_{\mu\in I_{j,m}}|\F^{-1}_{\R^n}(\phi_{2^j, 2^j\tau,\mu}^{(3)}\chi_{s,\kappa}\psi_{j,k})
*_{\R^n} f|\|_{L^2(\R^n)}
\les 2^{-c m}\|f\|_{L^2(\R^n)},
$$
which with  transference principle  gives
%\beq\label{ge124}
$$\|\sup_{j\in\Z}\sup_{\mu\in I_{j,m}}|\F^{-1}_{\R^n}(\phi_{2^j, 2^j\tau,\mu}^{(3)}\chi_{s,\kappa}\psi_{j,k})
*_{\Z^n} f|\|_{\ell^2(\Z^n)}
\les 2^{-c m}\|f\|_{\ell^2(\Z^n)}.
$$
This with the arguments leading to (\ref{he1}) gives
that for some $C>0$, 
\beq\label{ge18}
\|\sup_{j\ge v_s}\sup_{\A\in\mathcal{A}_s}\sup_{\mu\in I_{j,m}}
|\mathscr{L}_{s,\A}[\phi_{2^j, 2^j\tau,\mu}^{(3)}\psi_{j,k}
](D)f|\|_{\ell^2(\Z^n)}
\les\  2^{Cs-cm}\|f\|_{\ell^2(\Z^n)}.
\eeq
Thus, by taking $\e$ in (\ref{BC1}) small enough,  (\ref{HH34}) follows from the combination of
(\ref{ge18}) and (\ref{BC1}).
\end{proof}
\section{Proof of Theorem  \ref{co1}}
\label{slong5}
In this section, we prove Theorem  \ref{co1}. We will 
use  the  Gauss sum bounds to  establish an inequality (\ref{L149}) below, which is 
the second trick mentioned in subsection \ref{diff}.

By following the proof of Theorem \ref{t1} line by line,  it is easy to check that  Theorem \ref{co1} follows if we can remove the $R^\e$-loss in  (\ref{dou12}).  In other words, to achieve
   Theorem  \ref{co1},
% to achieve the goal,
it suffices to prove   that  for  every $(r,p)\in (2,\infty)\times [1+2/n,\infty)$, %with sufficiently small $\eta_0>0$,
there exists a constant $c_p>0$ such that
  \beq\label{DDou12}
I_{s,p,r}:=\big\| \sup_{\A\in\mathcal{A}_s}\big\|\big(\mathscr{L}_{s,\A}[\mathscr{V}_j~ \mathscr{B}](D)f \big)_{j> 2^{C_1s}}\big\|_{V^r}
\big\|_{\ell^p(\Z^n)}
\les\ 2^{-c_p s}\|f\|_{\ell^p(\Z^n)},
\eeq
where  $C_1$ is defined as in (\ref{L1}),
$\mathscr{V}_j$ and $\mathscr{B}$ are gievn as in Lemma \ref{ccz2}.  Modifying the arguments yielding Lemma \ref{ccz2} (or the arguments yielding \cite[Lemma 7.2]{KR22}),
we can obtain  that for every $p\in(1,\infty)$, there exists a constant  $c_p>0$ such that  \beq\label{AZZZ1}
I_{s,p,\infty}\les 2^{-c_p s}\|f\|_{\ell^p(\Z^n)} \quad
{\rm with}\quad I_{s,p,\infty}:=I_{s,p,r}|_{r=\infty}.
\eeq  
By interpolation, to show (\ref{DDou12}), it suffices to prove  that for  every $(r,p)\in (2,\infty)\times [1+2/n,\infty)$,
  \beq\label{DDD12}
I_{s,p,r}\les \|f\|_{\ell^p(\Z^n)}.
\eeq
%for any  $\e>0$. 
In fact, as we shall see later, for  the case $p\in (1+2/n,\infty)$,  the right-hand side of   (\ref{DDD12})
can be improved to  $2^{-c_ps} \|f\|_{\ell^p(\Z^n)}$ with some $c_p>0$. So 
we do not need  (\ref{AZZZ1}) as a black box.
Before we go ahead, we need first     the following lemma,
which can be seen as an improvement of (\ref{de11}).% by Krause-Roos.
\begin{lemma}\label{L147}
Let $n\ge 2$, $s\in\N$ and $d=1$. Then for every $p\in[1,\infty]$, we have
\beq\label{L149}
\|\big(\sum_{\A\in\mathcal{A}_s}\big|\mathscr{L}_{s,\A}[1](D)f\big|^p\big)^{1/p}\|_{\ell^p(\Z^n)}
\les \mathcal{W}_{p,s} \|f\|_{\ell^p(\Z^n)}
\eeq
with $\mathcal{W}_{p,s}:= 2^{\frac{2s}{p}-\frac{ns}{2}\min\{\frac{2}{p},\frac{2}{p'}\}}$.
\end{lemma}
 While the lemma remains valid for $n=1$, the lack of control over the growth of the upper bound in that setting explains why our analysis is restricted to $n\ge2$.
\begin{proof}[Proof of Lemma \ref{L147}]
Let  $\A=a/q\in \mathcal{A}_s$, and $\bb=b/q=(b_1,\dots,b_n)/q\in \frac{1}{q}\Z^n$.  We have
$$S(\A,\bb)=\frac{1}{q^n}\sum_{r=(r_1,\dots,r_n)\in [q]^n}e(\frac{a}{q}|r|^{2}+\frac{b}{q}\cdot r)=\prod_{k=1}^n\Big\{\frac{1}{q}\sum_{r_k\in [q]}e(\frac{a}{q}r_k^{2}+\frac{b_k}{q}r_k)\Big\}.$$
%Particularly,  we may assume $(a,q)=1$ since $S(\A,\bb)=0$ otherwise.
Since $q\sim 2^s$ and $(a,b,q)=1$, by applying  the Gauss sum bounds (see, e.g., \cite[Lemma 2.4 with $d=2$]{BCS24}),  we obtain that for all $1\le k\le n$,
$\big|{q}^{-1}\sum_{r_k\in [q]}e\big({a}~r_k^{2}/q+{b_k}~r_k/{q}\big)\big|\les\   2^{-{s}/{2}},$
which  yields
$$|S(\A,\bb)|\les\   2^{-ns/2}\ \ {\rm whenever}\ \ \A\in \mathcal{A}_s,\ \bb\in {q}^{-1}\Z^n.$$
This  with
 Plancherel's identity gives that for each $\A=a/q\in \mathcal{A}_s$,
\beq\label{END1}
\|\mathscr{L}_{s,\A}[1](D)f\|_{\ell^2(\Z^n)}^2\les \sum_{\bb\in \frac{1}{q}\Z^n}\|S(\A,\bb) \chi_{s,\kappa}(\xi-\bb)\F_{\Z^n} f\|_{L^2_\xi(\T^n)}^2\les\   2^{-n s}\|f\|_{\ell^2(\Z^n)}^2.
\eeq
On the other hand, for every $p\in[1,\infty]$,  we have by (\ref{zhu2}) that
\beq\label{END2}
\|\mathscr{L}_{s,\A}[1](D)f\|_{\ell^p(\Z^n)}
\les \||\F_{\R^n}^{-1}\big(\chi_{s,\kappa}\big)|* |f|\|_{\ell^p(\Z^n)}
\les \|f\|_{\ell^p(\Z^n)}.
\eeq
By taking the square root of  (\ref{END1}) and subsequently interpolating the resultant inequality with (\ref{END2}),
$$
\|\mathscr{L}_{s,\A}[1](D)f\|_{\ell^p(\Z^n)}
\les 2^{-\frac{ns}{2}\min\{2/p,2/p'\}}\|f\|_{\ell^p(\Z^n)},
$$
which with Fubini's theorem and (\ref{new1})  completes  the proof of Lemma \ref{L147}.
\end{proof}
\begin{proof}[Proof of (\ref{DDD12})]
Let $V_{s,1}$ be a constant depending only on $s$ given as in (\ref{KL2}).  Following the proof of (\ref{c93}) line by line and using   linearization, we can also get that  for each $p\in (1,\infty)$,
$$
\big\| \sup_{\A\in\mathcal{A}_s}\big\|\big(\mathscr{L}_{s,\A}[(\mathscr{V}_j-\NE_u\VV_j)~ \mathscr{B}](D)f \big)_{j> 2^{C_1s}}\big\|_{V^1}
\big\|_{\ell^p(\Z^n)}
\les\ 2^{- s}\|f\|_{\ell^p(\Z^n)}
$$
 for all $u\in [V_{s,1}]^n$.
Thus,
to prove (\ref{DDD12}), it suffices to show that for  every $(r,p)\in (2,\infty)\times [1+2/n,\infty)$,
  \beq\label{DDou133}
V_{s,1}^{-n}\sum_{u\in [V_{s,1}]^n}\sum_{\A\in\mathcal{A}_s} \big\| \big\|\big(\mathscr{L}_{s,\A}[(\NE_u\VV_j)~ \mathscr{B}](D)f \big)_{j> 2^{C_1s}}\big\|_{V^r}
\big\|_{\ell^p(\Z^n)}^p
\les\ \|f\|_{\ell^p(\Z^n)}^p.
\eeq
%for any $\e>0$.
%We first show  the case $p\in (1+1/n,\infty)$.
By  Lemma \ref{L147} and  Proposition \ref{PIW} with $m=\NE_u\mathscr{B}$, we have  for any $u\in [V_{s,1}]^n$
$$\|\big(\sum_{\A\in\mathcal{A}_s}|\mathscr{L}_{s,\A}[\NE_u \mathscr{B}](D)f |^p\big)^{1/p}\|_{\ell^p(\Z^n)}\les \mathcal{W}_{p,s}\|\mathscr{L}_{s}^\#[\NE_u \mathscr{B}](D)f \|_{\ell^p(\Z^n)}
\les \mathcal{W}_{p,s} \|f\|_{\ell^p(\Z^n)}.$$
Note that the restrictions  $p\in [1+2/n,\infty)$ and $n\ge 2$ can lead to  ${2s}/{p}-{ns}\min\{{1}/{p},{1}/{p'}\}\le 0$, which  yields
\beq\label{DC1}
\|\big(\sum_{\A\in\mathcal{A}_s}|\mathscr{L}_{s,\A}[\NE_u \mathscr{B}](D)f |^p\big)^{1/p}\|_{\ell^p(\Z^n)}\les \|f\|_{\ell^p(\Z^n)}.
\eeq
Using similar  arguments as reducing the proof of  (\ref{c96}) to proving  (\ref{ll99}),  we can also achieve (\ref{DDou133})   from (\ref{DC1}).   This ends the proof of (\ref{DDD12}).
\end{proof}
\begin{remark}\label{r100}
 From the preceding proof, it's evident that we can broaden the scope from $p\in [1+2/n,\infty)$ to $p\in (1+2/n-c,\infty)$, where $c$ is a small positive value.
\end{remark}
\section*{Acknowledgements}
 The authors would like to thank  Dashan Fan for his  discussions on topics related to this paper. We  thank  the anonymous referees for careful reading of the manuscript and constructive comments. 
 This work was supported by the National Key Research and Development Program of China (No. 2022YFA1005700), and NSC of China  (No. 11901301, 12571109).
 
\appendix
\section{Shifted square function estimate}
\label{app}
In this section,  we introduce a shifted square function estimate which plays an important role in proving major arcs estimate III.  Suppose $\sigma\ge 0$, we define   the shifted maximal operator $M^{[\sigma]}$  by
\beq\label{dsp1}
M^{[\sigma]}g(z):=\sup_{z\in I\subset \R}\frac{1}{|I|}
\int_{I^{(\sigma)}}|g(z')|dz',
\eeq
where $I^{(\sigma)}$ denotes a shift of the bounded interval $I=[a,b]$ given by
$$I^{(\sigma)}:=\big[a-\sigma|I|,b-\sigma|I|\big]\cup \big[a+\sigma|I|,b+\sigma|I|\big].$$
By using  Theorem 3.1 in \cite{GHLR17} (see \cite{M14,L18} for the scalar version),  we obtain that  for every $k\in\Z_+$ and each $p\in (1,\infty)$,
\begin{equation}\label{sme1}
\Big\|\big(\sum_{j\in\Z}|M^{[2^k]}f_j|^2\big)^{1/2}\Big\|_{L^p(\R)}
\lesssim\ k^2 \Big\|\big(\sum_{j\in\Z}|f_j|^2\big)^{{1}/{2}}\Big\|_{L^p(\R)}.
\end{equation}
\begin{lemma}\label{Ap1}
Let $n$ be a positive integer.
Let $h$ be a Schwartz function on $\R^n$ with $h_j(y)=2^{-jn}h(2^{-j}y)$, and let $1\le |\tau|\le 2$. Then for every $k\in \Z_+$ and $p\in (1,\infty)$, we have
\beq\label{APP1}
\|\big(\sum_{j\in\Z}|f_j*_{\R^n}h_j(\cdot-\tau\theta 2^{j+k})|^2\big)^{1/2}\|_{L^p(\R^n)}
\les k^{2} \|\big(\sum_{j\in\Z}|f_j|^2\big)^{1/2}\|_{L^p(\R^n)}
\eeq
with the implicit constant independent of $k$,
where $\theta=1$ when $n=1$, and $\theta\in \mathbb{S}^{n-1}$ when $n\ge 2$.
\end{lemma}
\begin{proof}
We can assume that $k$ is significantly large; otherwise, the outcome directly follows from the Fefferman-Stein inequality. Initially, we demonstrate the scenario for $n=1$.  Since $h$ is  a Schwartz function, we have
\begin{align*}
|f_j*_{\R}h_j(x-\tau 2^{j+k})|
%\les&\  2^{-jn}\int \big|f_j(y)\big|\big|h\big(2^{-j}(x-y-\tau 2^{k+j})\big)\big|dy\\
\les&\ 2^{-j}\int_\R \big|f_j(y)|\big\langle  \big|2^{-j}(x-y-\tau 2^{k+j})\big|\big\rangle^{-2}dy\\
\les&\ 2^{-j}\int_{|x-y-\tau 2^{k+j}|\le 2^j} \big|f_j(y)\big| dy
+ \sum_{l\ge 0} 2^{-j-2l}  \int_{|x-y-\tau 2^{k+j}|\le 2^{j+l+1}} \big|f_j(y)\big| dy.
\end{align*}
Using the definition  (\ref{dsp1}), we have
$$2^{-j}\int_{|x-y-\tau 2^{k+j}|\le 2^j} \big|f_j(y)\big| dy\les \sum_{|v|\le2} M^{[2^{k+v}]}f_j(x).$$
For  the second term,  we need to split the sum $\sum_{l\ge 0}$ into $\sum_{l> k-2}$ and $\sum_{0\le l\le k-2}$. For $l>k-2$, 
$$\int_{|x-y-\tau 2^{k+j}|\le 2^{j+l+1}} \big|f_j(y)\big| dy
\les \int_{|x-y|\le 2^{j+l+1}} \big|f_j(y)\big| dy\les  2^{j+l} M_{HL}f_j(x), $$
where $M_{HL}$ is the continuous    Hardy-Littlewood maximal function.
For ${0\le l\le k-2}$, we have
$$\int_{|x-y-\tau 2^{k+j}|\le 2^{j+l+1}} \big|f_j(y)\big| dy
\les  2^{j+l} \sum_{|v|\le l+2} M^{[2^{k+v}]}f_j(x).
$$
Thus we have
\beq\label{APP2}
|f_j*_{\R}h_j(x-\tau 2^{j+k})|
\les M_{HL}f_j(x)%+\sum_{v=0, \pm1} M^{[2^{k+v}]}f_j(x)
+\sum_{l\ge 0}2^{-l}\sum_{|v|\le l+2} M^{[2^{k+v}]}f_j(x),
\eeq
which with (\ref{sme1}) and the Fefferman-Stein inequality gives
 (\ref{APP1}) for the case $n=1$.

The case $n\ge 2$ can be achieved by  similar arguments. Let $\{e_j\}_{j=1}^n$ be the usual unit vectors in $\mathbb{S}^{n-1}$. By the method of rotation, we may reduce the matter to the case $\theta=e_1$, that is, it suffices to show
 \beq\label{APP10}
\|\big(\sum_{j\in\Z}|f_j*_{\R^n}h_j(\cdot-\tau e_1 2^{j+k})|^2\big)^{1/2}\|_{L^p(\R^n)}
\les k^{2} \|\big(\sum_{j\in\Z}|f_j|^2\big)^{1/2}\|_{L^p(\R^n)},
\eeq
 Let  $M_{HL,i}$ and $M^{[\cdot]}_i$ ($i=1,\dots,n$)  denote
 the  continuous Hardy-Littlewood maximal operator and  the   shifted maximal operator  applied in the
 $i$-th  variable,  respectively,  and define $\bar{M}^{[\cdot]}_1:=M_{HL,1}+M^{[\cdot]}_1$.
By following the arguments yielding  (\ref{APP2}), we have
\beq\label{APP11}
|f_j*_{\R^n}h_j(x-\tau e_1 2^{j+k})|
\les \sum_{l\ge 0}2^{-l}\sum_{|v|\le l+2} \bar{M}^{[2^{k+v}]}_1\circ M_{HL,2}\circ\cdots\circ{M_{HL,n}}f_j(x).
\eeq
We can see that
 $\bar{M}^{[\cdot]}_1$ satisfies a square function estimate like (\ref{sme1}), which with
 (\ref{APP11}) and the Fefferman-Stein inequality gives
  (\ref{APP10}) for the case $n\ge 2$.

\end{proof}

\section{Proof of \eqref{azq40}}
\label{appendixB}
In this section, we  prove \eqref{azq40}.  It suffices to prove 
\beq\label{details-azq40}
\|\sup_{\lambda\in X_{j_\circ,\e_\circ}}| E_{\Pi,N_\circ,\la,\e_\circ,\ka}(D)f|\|_{\ell^p(\Z^n)}
%\les \|\sup_{u\in X_{j_\circ,\e_\circ}}|E_{j_\circ,u,\e_\circ}(D)f|\|_{\ell^2(x\in\Z^n)}
\les 2^{-\gamma_{1,p} j_\circ}\|f\|_{\ell^p(\Z^n)}.
\eeq
Claim that there exists a constant  $c_0>0$ such that 
\beq\label{d-upper}
|E_{\Pi,N_\circ,\la,\e_\circ,\ka}(\xi)|\les 2^{-c_0 j_\circ}.
\eeq
We now proceed with the proof, assuming the validity of \eqref{d-upper}. Using \eqref{azq2}, 
we can obtain by a routine  computation that  
$|\partial_\la E_{\Pi,N_\circ,\la,\e_\circ,\ka}(\xi)|
\les 2^{2dj_\circ}.$
Applying  \cite[Lemma 5.1]{KR22} with 
$$(A,B,N,\delta)=(2^{-c_0 j_\circ},2^{2dj_\circ},j_\circ^{2\lfloor 1/\e_\circ \rfloor},2^{-2dj_\circ+1}j_\circ^{\lfloor 1/\e_\circ \rfloor})$$
(up to a multiplicative constant), 
we achieve 
\beq\label{details-azq40-2ind}
\|\sup_{\lambda\in X_{j_\circ,\e_\circ}}| E_{\Pi,N_\circ,\la,\e_\circ,\ka}(D)f|\|_{\ell^2(\Z^n)}
%\les \|\sup_{u\in X_{j_\circ,\e_\circ}}|E_{j_\circ,u,\e_\circ}(D)f|\|_{\ell^2(x\in\Z^n)}
\les j_\circ^{\frac{3}{2}\lfloor 1/\e_\circ \rfloor} 2^{-c_0 j_\circ/2}\|f\|_{\ell^2(\Z^n)}.
\eeq
From \eqref{MI10}, \eqref{exp1} and \eqref{azq2}, we infer that 
 there exists $C_n>0$ (independent of $j_\circ$) such that  
\beq\label{details-azq40-pind}
\|\sup_{\lambda\in X_{j_\circ,\e_\circ}}| E_{\Pi,N_\circ,\la,\e_\circ,\ka}(D)f|\|_{\ell^q(\Z^n)}
\les j_\circ^{C_n\lfloor 1/\e_\circ \rfloor} \|f\|_{\ell^q(\Z^n)}
\eeq
 for every $q\in [1,\infty]$. \eqref{details-azq40}  is obtained   by interpolation between  \eqref{details-azq40-2ind} with \eqref{details-azq40-pind}.

We finally prove \eqref{d-upper}, which combines Proposition \ref{p12} of this paper with exponential sum estimates from Stein-Wainger \cite{SW99}. In fact, the proof follows the arguments in \cite[Section 5]{KR22} (see also \cite[Section 6]{KR23}) with only minor modifications. More precisely, the only changes are that  $j$ is replaced by $j_\circ$ throughout, and 
  every occurrence of $2^{j\varepsilon_1}$ is replaced by $j_\circ^{\lfloor 1/\e_\circ \rfloor}$ (In particular, $2^{j \varepsilon_2}$ is replaced by $2^{j_\circ \varepsilon_2}$; the case distinction $1\le s_0\le \varepsilon_1 j$ and $ \varepsilon_1 j<s_0\le \varepsilon_2j $
 is  replaced by $1\le {s_0}\le  \e_\circ(j_\circ)$ and  $\e_\circ(j_\circ)<{s_0}\le {\varepsilon_2 j_\circ}$ respectively).


\begin{thebibliography}{99}


\bibitem{BCD} H. Bahouri, J.-Y. Chemin, R. Danchin, {Fourier Analysis and Nonlinear Partial Differential Equations}, Grundlehren der mathematischen Wissenschaften,  Springer, Heidelberg, 2011.

%\bibitem{BORSS22}
%D. Beltran, R.  Oberlin, L. Roncal, A. Seeger, B. Stovall,  Variation bounds for spherical averages, {\it  Math. Ann. \bf 382} (2022),  459-512.

\bibitem{BCS24}
R.C. Baker, C. Chen,  I.E. Shparlinski,  Bounds on the norms of maximal operators on Weyl sums, {\it  J. Number Theory \bf 256}  (2024), 329-353.



\bibitem{Bour89}
J. Bourgain, Pointwise ergodic theorems for arithmetic sets, {\it  Inst. Hautes \'{E}tudes Sci. Publ. Math. \bf 69} (1989), 5-45.

%\bibitem{BD15}
 %J. Bourgain,  C. Demeter,  The proof of the $l^2$ decoupling conjecture, {\it Ann. of Math.  \bf 182} (2015),  351-389.







\bibitem{BMSW18}
 J. Bourgain, M. Mirek, E.M. Stein, B. Wr\'{o}bel,  On dimension-free variational inequalities for averaging operators in $\R^d$, {\it  Geom. Funct. Anal. \bf 28} (2018),  58-99.



\bibitem{BMSW19}
J. Bourgain, M. Mirek, E.M. Stein,  B. Wr\'{o}bel, Dimension-free estimates for discrete Hardy-Littlewood averaging operators over the cubes in $\Z^d$, {\it  Amer. J. Math. \bf 141} (2019),  857-905.




%\bibitem{CJRW00}
%J. T. Campbell, R. L. Jones, K. Reinhold, and M. Wierdl,  Oscillation and variation for the Hilbert transform, {\it Duke Math. J. \bf 105:1} (2000), 59-83.


%\bibitem{CJRW03}
%J. T. Campbell, R. L. Jones, K. Reinhold, and M. Wierdl, Oscillation and variation for singular integrals in higher dimensions, {\it Trans. Amer. Math. Soc. \bf 355:5} (2003), 2115-2137.




%\bibitem{CSWW99}
%A. Carbery, A. Seeger, S. Wainger, J. Wright,  Classes of singular integral operators along variable lines, {\it J. Geom. Anal. \bf 9} (1999), 583-605.


\bibitem{Car66}
L. Carleson, On convergence and growth of partial sums of Fourier series,  {\it Acta Math. \bf 116} (1966),  135-157.




 \bibitem{CHKL18}
L.  Cladek, K. Henriot, B.  Krause,  I.  {\L}aba,  M.  Pramanik,  A discrete Carleson theorem along the primes with a restricted supremum, {\it Math. Z. \bf 289}  (2018), 1033-1057.

\bibitem{Fe73}
C. Fefferman, Pointwise convergence of Fourier series, {\it Ann.  Math. \bf 98} (1973),  551-571.



%\bibitem{FITW20}
%C. Fefferman, A.  Ionescu, T. Tao, S. Wainger,  Analysis and applications: the mathematical work of Elias Stein. With contributions from Loredana Lanzani, Akos Magyar, Mariusz Mirek, et al. Bull. Amer. Math. Soc. (N.S.) 57 (2020), 523-594.


%\bibitem{FS72}
%C. Fefferman, E.M. Stein, $H^p$ spaces of several variables, {\it Acta Math. \bf 129}  (1972), 137-193.


\bibitem{FZ23}
P.K. Friz,   P. Zorin-Kranich,  Rough semimartingales and $p$-variation estimates for martingale transforms, {\it  Ann. Probab. \bf 51}  (2023),  397-441.


\bibitem{Gra14}
L. Grafakos, Classical  Fourier Analysis, third edition, Graduate Texts in Mathematics, vol. 249, Springer, New York, 2014.




%\bibitem{GT01}
%D. Gilbarg, N.S. Trudinger,  Elliptic partial differential equations of second order. Reprint of the 1998 edition. Classics in Mathematics. Springer-Verlag, Berlin, 2001. xiv+517 pp. ISBN: 3-540-41160-7.

\bibitem{GHLR17}
S. Guo, J. Hickman, V. Lie, J. Roos,  Maximal operators and Hilbert transforms along variable non-flat homogeneous curves, {\it  Proc. Lond. Math. Soc. \bf 115} (2017),  177-219.


\bibitem{GRY20}
S. Guo, J. Roos, P.-L. Yung,  Sharp variation-norm estimates for oscillatory integrals related to Carleson's theorem, {\it Anal. PDE \bf 13} (2020), 1457-1500.



 \bibitem{IMMS23}
A. D. Ionescu,  \'{A}. Magyar,  M. Mirek,  T. Z. Szarek,  Polynomial averages and pointwise ergodic theorems on nilpotent groups, {\it  Invent. Math. \bf 231}  (2023),  1023-1140.


\bibitem{IW05}
A. Ionescu, S. Wainger, $L^p$ boundedness of discrete singular Radon transforms,  {\it J. Amer. Math. Soc. \bf 19}   (2005),  357-383.


\bibitem{JKRW98}
R.L. Jones, R. Kaufman, J.M. Rosenblatt, M. Wierdl,  Oscillation in ergodic theory, {\it  Ergod. Theory Dyn. Syst. \bf 18} (1998),  889-935.



\bibitem{JSW08}
R.L. Jones, A. Seeger, J. Wright,
Strong variational and jump inequalities in harmonic analysis, {\it
Trans. Amer. Math. Soc. \bf 360} (2008),  6711-6742




\bibitem{K24}
B. Krause,  Discrete analogues in harmonic analysis: a theorem of Stein-Wainger, {\it J. Funct. Anal. \bf 287}  (2024),  Paper No. 110498, 49 pp.




\bibitem{KL17}
B. Krause, M.T.  Lacey,  A discrete quadratic Carleson theorem on $l^2$ with a restricted supremum, {\it  Int. Math. Res. Not. IMRN} (2017),  3180-3208.


 \bibitem{KMT22}
B. Krause, M. Mirek, T.  Tao,  Pointwise ergodic theorems for non-conventional bilinear polynomial averages, {\it Ann. of Math. (2) \bf 195}  (2022),   997-1109.


 \bibitem{KR22}
 B. Krause, J.  Roos,  Discrete analogues of maximally modulated singular integrals of Stein-Wainger type, {\it J. Eur. Math. Soc. (JEMS) \bf 24} (2022), 3183-3213.

\bibitem{KR23}
 B. Krause, J.  Roos,  Discrete analogues of maximally modulated singular integrals of Stein-Wainger type: $l^p$ bounds for $p>1$, {\it  J. Funct. Anal. \bf 285} (2023), Paper No. 110123, 19 pp.







%\bibitem{K18}
%B. Krause,
%Discrete Analogoues in Harmonic Analysis: Maximally Monomially Modulated Singular Integrals Related to Carleson's Theorem, arXiv:1803.09431[math.CA].


%\bibitem{K23}
%B. Krause,
%Discrete Analogues in Harmonic Analysis: A Theorem of Stein-Wainger, arXiv:2210.06076[math.CA].




\bibitem{LT20}
M. Lacey,  C. Thiele, A proof of boundedness of the Carleson operator,  {\it Math. Res. Lett. \bf 7:4} (2000), 361-370.



\bibitem{Le76}
D. L\'{e}pingle,  La variation `d'ordre` $p$ des semi-martingales, {\it Z. Wahrscheinlichkeitstheorie und Verw. Gebiete \bf 36(4)} (1976), 295-316.




\bibitem{LL12}
 A. Lewko,  M.  Lewko,  Estimates for the square variation of partial sums of Fourier series and their rearrangements, {\it  J. Funct. Anal. \bf 262}  (2012),  2561-2607.
%\bibitem{LY22}
%J. Li, H. Yu,  $L^p$ boundedness of Carleson \& Hilbert transforms along plane curves with certain curvature constraints, {\it  J. Fourier Anal. Appl. \bf 28} (2022), Paper No. 11, 33 pp.


%\bibitem{LV20}
%V. Lie,  The polynomial Carleson operator, {\it  Ann. of Math. \bf 192} (2020),  47-163.


%\bibitem{LV24}
%V. Lie,  A unified approach to three themes in harmonic analysis (I \& II), {\it  Adv. Math. \bf 437} (2024), Paper No. 109385.




\bibitem{L18}
V. Lie,  On the boundedness of the bilinear Hilbert transform along ``non-flat" smooth curves, The Banach triangle case ($L^r$, $1\le r<\infty$), {\it  Rev. Mat. Iberoam. \bf 34} (2018),  331-353.


 \bibitem{MM18}
M. Mirek, Square function estimates for discrete Radon transforms, {\it  Anal. PDE \bf 11} (2018), 583-608.

\bibitem{MST17}
M. Mirek, E.M.  Stein, B. Trojan, $l^p(\Z^d)$-estimates for discrete operators of Radon type: variational estimates, {\it Invent. Math. \bf 209} (2017), 665-748.


\bibitem{MST199}
 M. Mirek, E.M.  Stein, B. Trojan, $l^p(\Z^d)$-estimates for discrete operators of Radon types I: maximal functions and vector-valued estimates, {\it  J. Funct. Anal.  \bf 277} (2019),  2471-2892.


\bibitem{MSZ20}
M. Mirek, E.M. Stein, P. Zorin-Kranich,  A bootstrapping approach to jump inequalities and their applications, {\it Anal. PDE \bf 13} (2020),  527-558.

\bibitem{MSZ203}
M. Mirek, E.M. Stein, P. Zorin-Kranich,   Jump inequalities via real interpolation, {\it Math. Ann. \bf 376}  (2020), 797-819.


\bibitem{MSZ202}
M. Mirek, E.M. Stein, P. Zorin-Kranich, Jump inequalities for translation-invariant operators of Radon type on $\Z^d$,  {\it Adv. Math. \bf 365} (2020), 107065, 57 pp.

\bibitem{MT16}
M. Mirek,  B. Trojan,  Discrete maximal functions in higher dimensions and applications to ergodic theory, {\it Amer. J. Math. \bf 138} (2016), 1495-1532.

\bibitem{M14}
C. Muscalu, Calder\'{o}n commutators and the Cauchy integral on Lipschitz curves revisited I. First commutator and generalizations, {\it Rev. Mat. Iberoam. \bf 30} (2014) 727-750.



 \bibitem{NOT10}
F. Nazarov,  R.  Oberlin, C.  Thiele, A Calder\'{o}n Zygmund decomposition for multiple frequencies and an application to an extension of a lemma of Bourgain, {\it  Math. Res. Lett. \bf 17} (2010),  529-545.


\bibitem{OSTTW12}
R. Oberlin, A. Seeger, T. Tao, C. Thiele, and J. Wright, A variation norm Carleson theorem,  {\it J. Eur. Math. Soc. \bf 14:2} (2012), 421-464.



\bibitem{PX88}
G. Pisier, Q. H. Xu, The strong p-variation of martingales and orthogonal series, {\it Probab. Theory Related Fields \bf 77:4} (1988), 497-514.

\bibitem{PP36}
M. Plancherel,  G. P\'{o}lya, Fonctions enti\`{e}res et int\'{e}grales de Fourier multiples, {\it Comment. Math. Helv. \bf 9:1} (1936), 224-248.

\bibitem{PP37}
M. Plancherel, G. P\'{o}lya, Fonctions enti\`{e}res et int\'{e}grales de Fourier multiples, II,  {\it Comment. Math. Helv. \bf 10:1} (1937), 110-163.

\bibitem{Wo22}
W. S{\l}omian,  Bootstrap methods in bounding discrete Radon operators,
{\it J. Funct. Anal. \bf 283}  (2022), Paper No. 109650, 30 pp.



\bibitem{St93}
E. Stein, Harmonic analysis: real-variable methods, orthogonality, and oscillatory integrals. With the assistance of Timothy S. Murphy, Princeton Mathematical Series 43, Monographs in Harmonic Analysis, III (Princeton University Press, Princeton, NJ, 1993) xiv+695.



\bibitem{SW99}
E.M. Stein, S.  Wainger,  Discrete analogues in harmonic analysis. I. $l^2$ estimates for singular Radon transforms, {\it  Amer. J. Math.  \bf 121} (1999),  1291-1336.

\bibitem{SW01}
E.M. Stein, S.  Wainger, Oscillatory integrals related to Carleson's theorem, {\it  Math. Res. Lett. \bf 8} (2001),  789-800.




 \bibitem{TT21}
T. Tao,  The Ionescu-Wainger multiplier theorem and the adeles, {\it  Mathematika  \bf 67} (2021), 647-677.



%\bibitem{Tr83}
%H. Triebel, Theory of function spaces, Monographs in Mathematics, vol. 78. Birkh\"{a}user Verlag, Basel (1983).



\bibitem{W00}
T. Wolff, Local smoothing type estimates on $L^p$ for large $p$, {\it Geom. Funct. Anal. \bf 10} (2000), 1237-1288.






\bibitem{ZK15}
P. Zorin-Kranich,  Variation estimates for averages along primes and polynomials, {\it  J. Funct. Anal. \bf 268} (2015)  210-238.


%\bibitem{ZK21}
%P. Zorin-Kranich,  Maximal polynomial modulations of singular integrals, {\it  Adv. Math. \bf 386} (2021), Paper No. 107832, 40 pp.














\end{thebibliography}
\end{document}